\documentclass[12pt]{amsart}

\usepackage{graphicx}
\usepackage{amsmath}
\usepackage{amscd}
\usepackage{amsfonts}
\usepackage{amssymb}
\usepackage{color}
\usepackage{fullpage}
\usepackage{mathtools}
\usepackage{graphicx}
\usepackage{caption}
\usepackage{subcaption}
\usepackage[pagebackref,hypertexnames=false, colorlinks, citecolor=red,linkcolor=blue, urlcolor=red]{hyperref}

\numberwithin{equation}{section}

\newtheorem{theorem}{Theorem}[section]
\newtheorem{lemma}[theorem]{Lemma}
\newtheorem{proposition}[theorem]{Proposition}

\newtheorem{corollary}[theorem]{Corollary}
\newtheorem{problem}[theorem]{Problem}

\theoremstyle{definition}
\newtheorem{definition}[theorem]{Definition}

\newtheorem{remark}[theorem]{Remark}

\newcommand{\be}{\begin{equation}}
\newcommand{\ee}{\end{equation}}
\newcommand{\bes}{\begin{equation*}}
\newcommand{\ees}{\end{equation*}}

\newcommand{\cC}{\mathcal{C}}
\newcommand{\cD}{\mathcal{D}}
\newcommand{\cE}{\mathcal{E}}
\newcommand{\cH}{H}

\newcommand{\cB}{\mathcal{B}}

\newcommand{\cS}{\mathcal{S}}
\newcommand{\cT}{\mathcal{T}}

\newcommand{\cW}{\mathcal{W}}

\newcommand{\AND}{\text{ and }}
\newcommand{\FOR}{\text{ for }}
\newcommand{\FORAL}{\text{ for all }}

\newcommand{\bB}{\mathbb{B}}
\newcommand{\bC}{\mathbb{C}}
\newcommand{\bD}{\mathbb{D}}
\newcommand{\bM}{\mathbb{M}}
\newcommand{\bN}{\mathbb{N}}

\newcommand{\bR}{\mathbb{R}}
\newcommand{\bS}{\mathbb{S}}

\newcommand{\bZ}{\mathbb{Z}}

\newcommand{\bMd}{\bM^d}
\newcommand{\bMsad}{\bM^d_{sa}}

\newcommand{\conv}{\operatorname{conv}}

\newcommand{\Int}{\operatorname{int}}

\newcommand{\spn}{\operatorname{span}}
\newcommand{\UCP}{\operatorname{UCP}}

\newcommand{\Wmin}[1]{\cW^{\text{min}}_{#1}}
\newcommand{\Wmax}[1]{\cW^{\text{max}}_{#1}}

\newcommand{\ol}{\overline}

\begin{document}

\title{Minimal and maximal matrix convex sets}

\author[Passer]{Benjamin Passer}
\thanks{The work of B. Passer is partially supported by a Zuckerman Fellowship at the Technion.}
\address{Faculty of Mathematics\\
Technion - Israel Institute of Technology\\
Haifa\; 3200003\\
Israel}
\email{benjaminpas@technion.ac.il}

\author[Shalit]{Orr Moshe Shalit}
\thanks{The work of O.M. Shalit is partially supported by ISF Grants no. 474/12 and 195/16.}
\address{Faculty of Mathematics\\
Technion - Israel Institute of Technology\\
Haifa\; 3200003\\
Israel}
\email{oshalit@technion.ac.il}

\author[Solel]{Baruch Solel}
\address{Faculty of Mathematics\\
Technion - Israel Institute of Technology\\
Haifa\; 3200003\\
Israel}
\email{mabaruch@technion.ac.il}

%%%%%%%%%%%%%%%%%%%%%%%%
\begin{abstract}
To every convex body $K \subseteq \bR^d$, one may associate a minimal matrix convex set $\Wmin{}(K)$, and a maximal matrix convex set $\Wmax{}(K)$, which have $K$ as their ground level.
The main question treated in this paper is: under what conditions on a given pair of convex bodies $K,L \subseteq \bR^d$ does $\Wmax{}(K) \subseteq \Wmin{}(L)$ hold?
For a convex body $K$, we aim to find the optimal constant $\theta(K)$ such that $\Wmax{}(K) \subseteq \theta(K) \cdot \Wmin{}(K)$; we achieve this goal for all the $\ell^p$ unit balls, as well as for other sets.
For example, if $\ol{\bB}_{p,d}$ is the closed unit ball in $\bR^d$ with the $\ell^p$ norm, then
\[
\theta(\ol{\bB}_{p,d}) = d^{1-|1/p - 1/2|}.
\]
This constant is sharp, and it is new for all $p \neq 2$.
Moreover, for some sets $K$ we find a minimal set $L$ for which $\Wmax{}(K) \subseteq \Wmin{}(L)$.
In particular, we obtain that a convex body $K$ satisfies $\Wmax{}(K) = \Wmin{}(K)$ if and only if $K$ is a simplex.

These problems relate to dilation theory, convex geometry, operator systems, and completely positive maps.
We discuss and exploit these connections as well.
For example, our results show that every $d$-tuple of self-adjoint operators of norm less than or equal to $1$, can be dilated to a commuting family of self-adjoints, each of norm at most $\sqrt{d}$.
We also introduce new explicit constructions of these (and other) dilations.
\end{abstract}

\maketitle

%%%%%%%%%%%%%%%%%%%%%%%%%
\section{Introduction}

%%%%%%%%%%%%%%%%%%%%%%%%%
\subsection{Overview}
This paper treats containment problems for matrix convex sets.
A {\em matrix convex set} in $d$-variables is a set $\cS = \cup_n \cS_n$, where every $\cS_n$ consists of $d$-tuples of $n \times n$ matrices, that is closed under direct sums, unitary conjugation, and the application of completely positive maps.
Matrix convex sets are closely connected to operator systems and have been investigated for several decades.
Recently, they appeared in connection with the interpolation problem for UCP maps (see, e.g., \cite{DDSS,HKM13,Zalar}), and also in the setting of relaxation of spectrahedral containment problems \cite{HKM13,HKMS15}.

Given a closed convex set $K \subseteq \bR^d$, one can define several matrix convex sets $\cS$ such that $\cS_1 = K$.
One may ask, to what extent does the ``ground level" $\cS_1$ determine the structure and the size of $\cS$?
For example: given two matrix convex sets $\cS = \cup_n \cS_n$ and $\cT = \cup_n \cT_n$, what does containment at the first level $\cS_1 \subseteq \cT_1$ imply about the relationship between $\cS$ and $\cT$?
Of course, there is (usually) no reason that $\cS \subseteq \cT$ would follow as a consequence of $\cS_1 \subseteq \cT_1$, but in many cases --- given some conditions on $\cS_1$ --- one can find a constant $C$ such that
\[
\cS_1 \subseteq \cT_1 \Longrightarrow \cS \subseteq C \cdot \cT .
\]
One way to treat the problem follows from the observation that given a closed convex set $K \subseteq \bR^d$, there exist a minimal matrix convex set $\Wmin{}(K)$, and a maximal matrix convex set $\Wmax{}(K)$, which have $K$ as their ground level: $K = \Wmin{1}(K) = \Wmax{1}(K)$.
To solve the above problem, one can let $K = \cS_1$, and concentrate on the problem of finding the best constant $C$ for which
\[
\Wmax{}(K) \subseteq C \cdot \Wmin{}(K).
\]
This problem was treated in \cite{DDSS,FNT,HKMS15}.
For example, if $K$ enjoys some symmetry properties, then it was shown that $C = d$ works.
If $K$ is the Euclidean ball $\ol{\bB}_{2,d} \subseteq \bR^d$, then $C = d$ is the optimal constant.

In this paper we find the optimal constant for a large class of sets.
For example, it was known that for the cube $K = [-1,1]^d$ (which is just the unit ball of the $\ell^\infty$ norm), the constant $C = d$ works, but optimality was not known.
We find that the best constant for the cube is actually $C = \sqrt{d}$, as part of a more general technique which can assign unequal norms to the dilations. Moreover, we find sharp constants for the $\ell^p$ balls, $p \in [1,\infty]$.

We also treat the closely related problem of finding conditions on two convex bodies $K,L$ such that $\Wmax{}(K) \subseteq \Wmin{}(L)$.
In particular, we show that $\Wmin{}(K) = \Wmax{}(K)$ if and only if $K$ is a simplex (somewhat improving a result from \cite{FNT}, where this was essentially obtained under the assumption that $K$ is a polytope).

In order to describe our main results more clearly, we now turn to setting the notation and reviewing some preliminaries.

\subsection{Notation and preliminaries}\label{sec:notation}

%%%%%%%%%%%%%%%%%%%%%%%%%
\subsubsection{General background and notation}\label{subsec:general}

Let $d \in \bN$, and let $M_n^d$ and $(M_{n})_{sa}^d$ denote $d$-tuples of matrices or self-adjoint matrices, respectively.
We consider the disjoint unions $\bMd = \cup_n M_n^d$ and $\bMsad = \cup_n(M_n^d)_{sa}$.
The algebra of bounded operators on a Hilbert space $H$ is denoted by $\cB(H)$, $\cB(H)^d$ denotes $d$-tuples of operators, and $\cB(H)_{sa}^d$ denotes $d$-tuples of self-adjoint operators.

We will use basic results on C*-algebras and operator systems freely (see \cite{DavBook} and \cite{PauBook}, respectively), but let us recall a few definitions.
If $A$ and $B$ are C*-algebras, a linear map $\phi : A \to B$ is said to be {\em positive} if $\phi(a) \geq 0$ whenever $a\geq 0$.
If $A$ and $B$ have units, then $\phi$ is said to be {\em unital} if $\phi(1) = 1$.
The $n \times n$ matrix algebra over $A$ is denoted $M_n(A)$, and is also a C*-algebra.
A linear map $\phi : A \to B$ can be promoted to a map $\phi_n : M_n(A) \to M_n(B)$ by acting componentwise, and a map $\phi$ is said to be {\em completely positive} if $\phi_n$ is positive for all $n$.
A unital and completely positive map will be called, briefly, a {\em UCP map}, and the set of all UCP maps between $A$ and $B$ will be denoted $\UCP(A,B)$.

By Stinespring's theorem \cite{Sti55}, for every UCP map $\phi : A \to \cB(H)$, there exists a Hilbert space $K$, an isometry $v : H \to K$, and a unital $*$-representation $\pi : A \to \cB(K)$, such that
\[
\phi(a) = v^* \pi(a) v \,\, , \,\, a \in A.
\]
Moreover, one can arrange that the closed subspace spanned by $\pi(A)vH$ is equal to $K$; under this condition $(K,v, \pi)$ is determined uniquely, and is referred to as the {\em (minimal) Stinespring dilation} of $\phi$.

UCP maps between matrix algebras $M_n = M_n(\bC)$ can be characterized using the unital case of Choi's theorem (see \cite{Choi75}): a map
$\phi : M_n \to M_k$ is UCP if and only if it has the form
\[
\phi(a) = \sum_{j=1}^N t_i a t_i^*\, \, , \,\, a \in M_n,
\]
for some $N \leq kn$ and $t_1, \ldots, t_N \in M_{k,n}$ with $t_1t_1^* + \ldots + t_Nt_N^* = I_{M_k}$. Choi's theorem is a special case of Stinespring dilation.

An {\em operator system} is a vector subspace $S$ of a unital C*-algebra $A$ such that $1_A \in S$ and $S = S^*$. Given a tuple $M = (M_1, \ldots, M_d)$ of elements in the same unital C*-algebra, we let $S(M)$ or $S_M$ denote the operator system generated by $M_1, \ldots, M_d$, and we write $C^*(S_M)$ or $C^*(M)$ for the {\em unital} C*-algebra generated by $M_1, \ldots, M_d$. Note that since an operator system $S$ is defined in reference to a larger unital $C^*$-algebra, we may discuss positivity of elements in $S$; in fact, $S$ is spanned by its positive elements. Since the same reasoning applies to the operator systems $M_n(S) \subseteq M_n(A)$, the definitions of positive, completely positive, and UCP maps carry over to the discussion of linear maps between operator systems. Arveson's extension theorem \cite[Theorem 1.2.3]{Arv69} states that if $S$ is an operator system contained in a unital C*-algebra $A$, then every UCP map of $S$ into $\cB(H)$ extends to a UCP map from $A$ into $\cB(H)$.

%%%%%%%%%%%%%%%%%%%%%%%%%
\subsubsection{Matrix convex sets}\label{subsec:matrixconvex}
We now give some background on matrix convex sets; see \cite[Sections 2,3]{DDSS} for some more details.

A {\em free set $($in $d$ free dimensions$)$} $\cS$ is a disjoint union $\cS = \cup_n \cS_n \subseteq \bMd$, where $\cS_n \subseteq M^d_n$.
In this paper, the focus will be on free sets in $\bMsad$.
Containment is defined in the obvious way: we say that $\cS \subseteq \cT$ if $\cS_n \subseteq \cT_n$ for all $n$.
A free set $\cS$ is said to be {\em open/closed/convex} if $\cS_n$ is open/closed/convex for all $n$.
It is said to be {\em bounded} if there is some $C$ such that for all $n$ and all $A \in \cS_n$, it holds that $\|A_i\|\leq C$ for all $i$.

An {\em nc set} is a free set that is closed under direct sums and under simultaneous unitary conjugation.

An nc set is said to be {\em matrix convex} if it is closed under the application of UCP maps, meaning that whenever
$X$ is in $\cS_n$ and $\phi \in \UCP(M_n, M_k)$, the tuple $\phi(X) := (\phi(X_1), \ldots, \phi(X_d))$ is in $\cS_k$.
The main examples of matrix convex sets are given by {\em free spectrahedra} and {\em matrix ranges}.

A {\em monic linear pencil} is a free function of the form
\[
L(x) = L_A(x) = I + \sum A_j x_j,
\]
where $A \in \cB(H)^d$; in this paper we will concentrate on the case where $A \in \cB(H)_{sa}^d$.
The pencil $L$ acts on a $d$-tuple $X=(X_1,\dots,X_d)$ in $M_n^d$ by
\[
L(X) = I \otimes I_n + \sum_{j=1}^n A_j \otimes X_j .
\]
We write
\[
\cD_L = \cD_{L_A} = \cup_n \cD_L(n) = \cup_n \cD_{L_A}(n),
\]
where
\[
\cD_L(n) = \{X = (X_j) \in (M_{n})_{sa}^d:  L(X)\geq 0\} .
\]
The set $\cD_L$ is said to be a {\em free spectrahedron}.
Some authors use ``free spectrahedron'' for pencils with matrix coefficients, but we allow operator coefficients.
There is also a nonself-adjoint version, but we do not discuss it in this paper.

We shall require the homogenization of a monic pencil $L_A$, which gives a so called {\em truly linear} pencil ${}^hL_A$ in $d+1$ variables, defined by
\[
{}^h L_A(X_1, \ldots, X_d, X_{d+1}) = \sum_{j=1}^d A_j \otimes X_j + I \otimes X_{d+1},
\]
and its corresponding positivity set $\cD_{{}^h L_A} = \{X \in \bM^{d+1}_{sa} : {}^h L_A(X) \geq 0\}$ (see \cite{Zalar}).

The {\em matrix range} \cite[Section 2.4]{Arv72} of a tuple $A$ in $\cB(H)^d$ is defined to be the set
\[
\cW(A) = \cup_n \cW_n(A),
\]
where
\[
\cW_n(A) = \{(\phi(A_1), \ldots, \phi(A_d)) : \phi \in \UCP(C^*(S_A), M_n)\}.
\]
By \cite[Proposition 3.5]{DDSS}, a set $\cS = \cup_n \cS_n \subseteq \bMd$ (respectively, in $\bMsad$) is a closed and bounded matrix convex set if and only if $\cS = \cW(A)$ for some $A \in \cB(H)^d$ (respectively, $A \in \cB(H)_{sa}^d$).

%%%%%%%%%%%%%%%%%%%%%%%%%
\subsubsection{Minimal and maximal matrix convex sets}\label{subsec:minmaxmatrixconvex}

Given a closed convex set $K \subseteq \bR^d$ (or $K \subseteq \bC^d$), one may ask whether it is equal to $\cS_1$ for some matrix convex set $\cS$.
The minimal matrix convex set $\Wmin{}(K)$ and the maximal matrix convex set $\Wmax{}(K)$ for which
\[
\Wmin{1}(K) = \Wmax{1}(K) = K
\]
were described in several places; we follow the conventions of \cite[Section 4]{DDSS} (see also \cite{FNT}, \cite{HKM16} and \cite{PTT11}, noting that analogous constructions in the theory of operator spaces have appeared as far back as \cite{Pau92}).
First, if we let
\be\label{eq:Wmax_def1}
\Wmax{n}(K) = \{ X \in (M_n)_{sa}^d :  \sum_{i=1}^d \alpha_i X_i \le a I_n \text{  whenever  } \sum \alpha_i x_i \leq a \FORAL x \in K\},
\ee
then $\Wmax{}(K) = \cup \Wmax{n}(K)$ is clearly seen to be a closed matrix convex set, which satisfies all the linear inequalities that determine $K$.
Since satisfying the linear inequalities that determine $K$ is a necessary requirement for any matrix convex set that has ground level equal to $K$, we see that $\Wmax{}(K)$ is indeed maximal.
An alternative description for $\Wmax{}(K)$ is
\be\label{eq:Wmax_def2}
\Wmax{n}(K) = \{ X \in (M_n)_{sa}^d :  \cW_1(X) \subseteq K\}.
\ee

The minimal matrix convex set $\Wmin{}(K)$ (called the {\em matrix convex hull of} $K$ in \cite{HKM16}) clearly exists, as the intersection of all matrix convex sets containing $K$, but here is a more useful description.
A $d$-tuple $X \in \cB(H_1)^d$ is said to be a {\em compression} of $A \in \cB(H_2)^d$ if there is an isometry $V:H_1\to H_2$ such that $X_i = V^* A_i V$ for $1 \le i \le d$.
Conversely, $A$ is said to be a {\em dilation} of $X$ if $X$ is a compression of $A$.
We will write $X \prec A$ when $X$ is a compression of $A$.
A tuple $N = (N_1, \ldots, N_d)$ will be called a {\em normal} tuple if $N_1, \ldots, N_d$ are normal commuting operators, and in this case, the unital $C^*$-algebra generated by the $N_i$ is commutative. In other words, if $N$ is normal, $C^*(S_N)$ is of the form
\bes
C(X) = \{f: X \to \bC: f \text{ is continuous}\}
\ees
for a unique compact Hausdorff space $X$, which is the maximal ideal space of $C^*(S_N)$. The \textit{joint spectrum} of $N = (N_1, \ldots, N_d)$ is then the compact set
\bes
\sigma(N) := \{(N_1(x), \ldots, N_d(x)): x \in X\} \subseteq \bC^d,
\ees
and if $N$ acts on a finite dimensional space, $\sigma(N)$ is nothing but the finite set of all $d$-tuples of joint eigenvalues.
For a normal tuple $N$, the individual operators $N_1, \ldots, N_d$ are all self-adjoint if and only if the joint spectrum of $N$ satisfies $\sigma(N) \subseteq \bR^d$. If a normal tuple $N$ happens to consist of self-adjoint operators, we still call it a normal tuple, and avoid calling it a self-adjoint tuple.
In \cite[Proposition 4.3]{DDSS}, it was shown that
\be\label{eq:Wmin_def1}
\Wmin{n}(K) = \{ X \in (M_n)_{sa}^d : \exists N \textrm{ normal, s.t. }  X  \prec N  \text \AND \sigma(N) \subseteq K \} .
\ee
In \cite[Section 4]{FNT}, a slightly different version was given, which we write here as
\be\label{eq:Wmin_def2}
\Wmin{n}(K) = \left\{ \sum x^{(j)} \otimes P_j  \in (M_n)_{sa}^d : x^{(j)} \in K,  P_j \geq 0 \AND \sum P_j = I\right\}.
\ee
(See also \cite[Definition 3.8]{PTT11}, with {\em min} and {\em max} reversed.)
In \cite{FNT} this was shown to be the minimal matrix convex set containing $K$, so this also defines $\Wmin{}(K)$.
One can also see directly --- using Naimark's dilation theorem --- that the definitions \eqref{eq:Wmin_def1} and \eqref{eq:Wmin_def2} are equivalent.

%%%%%%%%%%%%%%%%%%%%%%%%%
\subsection{Main results}\label{subsec:mainresults}

At this point we can finally state our main results.
In Section \ref{sec:convex_geometry} we introduce the dilation constants
\bes
\theta(K, L) := \inf\{C > 0: \Wmax{}(K) \subseteq C \cdot \Wmin{}(L)\} ,
\ees
\bes
\theta(K) := \theta(K, K) ,
\ees
\bes
\mathring{\theta}(K) := \inf\{\theta(T(x + K)): x \in \bR^d, T \in GL_d(\bR)\} .
\ees
In particular, $\mathring{\theta}(K)$ is a shift invariant version of $\theta(K)$.
A leitmotif in this paper is that the numerical values of these constants are determined by the geometry of convex sets. To this end, we first
catalog the natural relationship between these constants and the familiar Banach-Mazur distance $\rho(K,L)$ between convex bodies $K$ and $L$. Namely,
\be\label{eq:theta_leq_rhotheta}
\theta(K) \leq \rho(K, L) \, \theta(L) ,
\ee
and a corresponding inequality holds for the shift invariant constants (Proposition \ref{prop:scaletoscale}).
When $K$ is symmetric ($K = -K$), one would expect that the constants $\theta(K)$ and $\mathring{\theta}(K)$ are equal.
We show that this is indeed the case in Proposition \ref{prop:centered}.

Our first main result is that for a compact convex set $K \subset \bR^d$, $\theta(K) = 1$ if and only if $K$ is a simplex, which holds if and only if $\Wmax{2^{d-1}}(K) = \Wmin{2^{d-1}}(K)$ (Theorem \ref{thm:simplex_unique}).
This will be used in Section \ref{sec:cones} to improve the corresponding part of \cite[Theorem 4.7]{FNT}, which is a result of Fritz, Netzer and Thom. (After a translation from the setting of matrix convex sets over cones to convex bodies, the corresponding part of \cite[Theorem 4.7]{FNT} says that for \textit{polyhedral} $K$, $\theta(K) = 1$ if and only if $K$ is a simplex, and the result does not give control over the matrix level). Just as in \cite{FNT}, bounds on the Banach-Mazur distances between simplices and convex bodies show that $\mathring{\theta}(K) \leq d+2$ for any convex body $K \subset \bR^d$, and $\theta(K) \leq d$ if $K$ is also symmetric (Corollary \ref{cor:dplustwo_symmetric}). This implies Corollary \ref{cor:dilation_sym}:
{\em if $K \subseteq \bR^d$ is a symmetric convex set, $H$ is a Hilbert space,
and $A \in \cB(H)^d_{sa}$ has $\cW_1(A) \subseteq K$, then there exists a normal dilation $A \prec N$ with $\sigma(N) \subseteq d \cdot K$}. In contrast, the scale $d$ was seen in \cite{DDSS} under conditions which did not appear to capture all symmetric sets, but were also not limited to the symmetric setting.
In Theorem \ref{thm:dilate_sym_n}, we give an alternative proof of the fact that $\theta(K) \leq d$ for symmetric $K$ using the techniques of \cite{DDSS}.

The next two sections are devoted to studying a family of convex sets of special interest.
In Section \ref{sec:ball} we completely resolve the dilation problem for the $\ell^2$-ball.
We show that if $B_1 = x +  C_1 \cdot \ol{\bB}_{2,d}$ and $B_2 = y +  C_2 \cdot \ol{\bB}_{2,d}$ are $\ell^2$-balls in $\bR^d$, then $\Wmax{}(B_1) \subseteq \Wmin{}(B_2)$ if and only if there is a $d$-simplex $\Pi$ with $B_1 \subseteq \Pi \subseteq B_2$.
This result appears in Theorem \ref{thm:everythingball}, together with explicit conditions on the centers and radii of the balls:
\[
\Wmax{}(B_1) \subseteq \Wmin{}(B_2) \Longleftrightarrow C_2 \geq \sqrt{||y - x||^2 + C_1^2 (d-1)^2} + C_1 .
\]
Two special cases of the above are the computations $\theta(\ol{\bB}_{2,d}) = d$ and $\theta(x + \ol{\bB}_{2,d}) = \infty$ for $||x|| = 1$, which were obtained in \cite{DDSS}. A key role in the analysis of the dilation problem for the ball is played by a universal $d$-tuple $F = (F_1, \ldots, F_d)$ of $2^{d-1} \times 2^{d-1}$ self-adjoint, mutually anti-commuting unitary matrices.

After solving the dilation problem for $\ell^2$-balls, we turn to treat the $\ell^p$-balls
\bes
\ol{\bB}_{p,d} := \{x \in \bR^d : \|x\|_p\leq 1 \},
\ees
as well as their positive sections
\bes
\ol{\bB}_{p,d}^+ := \ol{\bB}_{p,d} \cap [0,1]^d .
\ees
Note that $\ol{\bB}_{\infty,d} = [-1,1]^d$ and that $\ol{\bB}_{1,d}^+$ is a simplex.
By the results in \cite{DDSS}, it was known that each of these sets has dilation constant at most $d$, but except for the $\ell^2$-ball $\ol{\bB}_{2,d}$, optimality of the constant $d$ was not claimed.
In Theorems \ref{thm:BplusWminmax} and \ref{thm:balldilationgeneral}, we find that
\[
\theta(\ol{\bB}_{p,d}^+) = d^{1-1/p} \quad \,  \textrm{ and } \quad  \, \theta(\ol{\bB}_{p,d}) = d^{1-|1/p-1/2|},
\]
among several other sharp variants of these constants, such as $\theta(\ol{\bB}_{p,d}^+, \ol{\bB}_{1,d}^+) = d^{1-1/p}$. However, note that not all of these dilation constants are obtained from simplex containment. For example, while $\theta([-1, 1]^2) = \sqrt{2}$, there is no simplex $\Pi$ with $[-1,1]^2 \subseteq \Pi \subseteq [-\sqrt{2}, \sqrt{2}]^2$.

The calculation of $\theta(\ol{\bB}_{p,d})$ for all $p$ relies on the special cases $p=1,2,\infty$. While the constants $\theta(\ol{\bB}_{1,d})$ and $\theta(\ol{\bB}_{\infty, d})$ must be equal by duality, we provide explicit dilation constructions that give rise to both constants (Theorems \ref{thm:diamonddilationgeneral} and \ref{thm:cubedilationgeneral}), and these constructions also include positive parameters $a_1, \ldots, a_d$. Namely, if $D(a_1, \ldots, a_d)$ denotes the convex hull of $(\pm a_1, 0, \ldots, 0), \ldots, (0, \ldots, 0, \pm a_d)$, then
\[
\Wmax{}(\ol{\bB}_{1,d}) \subseteq \Wmin{}(D(a_1, \ldots, a_d)) \hspace{-.1 cm} \iff \hspace{-.1 cm} \Wmax{}(\ol{\bB}_{\infty, d}) \subseteq \Wmin{}\left(\prod [-a_j, a_j] \right) \hspace{-.1 cm}  \iff  \hspace{-.1 cm} \sum \cfrac{1}{a_j^2} \leq 1.
\]
In particular, {\em every $d$-tuple of self-adjoint contractions can be dilated to a commuting $d$-tuple of self-adjoints with norm at most $\sqrt{d}$}. The unit balls of the complex $\ell^p$ spaces inside $\bR^{2d}$ are also briefly treated; see Corollary \ref{cor:complex_balls} and the surrounding discussion.

\begin{remark}
The constants we compute in this paper depend on the geometry of the convex sets, and in particular, they depend on the number of variables.
On the other hand, they are independent of the rank of the operators to which one may apply them.
In contrast, the lion's share in \cite{HKMS15} is dedicated to finding constants that are rank dependent, but independent of the number of variables.
By one of the main results of \cite{HKMS15}, any $d$ self-adjoint $n \times n$ matrices can be dilated to $d$ commuting self-adjoints of norm less than or equal to $\vartheta(n)$, which is a constant independent of $d$.
It is known that $\vartheta(n)$ behaves like $\sqrt{n}$, so when $d$ is large relative to $n$, the constant $\vartheta(n)$ from \cite{HKMS15} is better than $\theta([-1,1]^d) = \sqrt{d}$, and inversely.
\end{remark}

We were very much inspired by the paper \cite{FNT} of Fritz, Netzer and Thom.
In that paper, the main objects of study were matrix convex sets living over cones, whereas we usually consider convex bodies.
In Section \ref{sec:cones} we catalog a simple device for passing from one setting to the other.
We use it to import results or send some of our results to their setting, with improvements and alternative proofs when possible. First, we consider the translation of the problem of finite dimensional realizability treated in \cite{FNT}.
Importing results shows that for a  compact convex set $K \subset \bR^d$,
$\Wmin{}(K) = \cW(A)$ for $A \in \bMsad$ if and only if $K$ is a polytope, and if $K$ is a polytope, then
$\Wmax{}(K) = \cW(A)$ for $A \in \bMsad$ if and only if $K$ is a simplex.
However, for the statement regarding $\Wmin{}(K)$, we offer an alternative proof. In Corollary \ref{cor:Wminmaxcones}, we translate our Theorem \ref{thm:simplex_unique} to show {\em for a salient convex cone with nonempty interior $C \subseteq \bR^d$, it holds that $\Wmin{}(C) = \Wmax{}(C)$ if and only if $C$ is a simplicial cone.} This is an improvement of the corresponding part of \cite[Theorem 4.7]{FNT} in that we drop the assumption that $C$ is polyhedral.

If $K,L \subseteq \bR^d$ are compact convex sets and $\Wmax{}(K) \subseteq \Wmin{}(L)$, then we say that $L$ is a {\em dilation hull} of $K$.
In Section \ref{sec:minhulls} we consider the collection of all dilation hulls of a fixed $K$, and we search for minimal elements of this collection, which we call {\em minimal dilation hulls}.
First, we are able to show that minimal dilation hulls always exist in Proposition \ref{prop:minhullsdoexist}, and that the diamond $d \cdot \ol{\bB}_{1,d}$ is a minimal dilation hull of the $\ell^2$-ball $\ol{\bB}_{2,d}$ in Proposition \ref{prop:mindil_B1B2}. Therefore, the dilation of an $\ell^2$-ball to an $\ell^1$-ball is not characterized by simplex containment, in contrast to the case when both sets are $\ell^2$-balls. On the other hand, certain simplices over $\ol{\bB}_{2,d}$ are also minimal dilation hulls (Proposition \ref{prop:minsimplexball}), so the shape of minimal dilation hulls is not unique. Finally, the simplex is a minimal dilation hull of ``simplex-pointed sets" (Theorem \ref{thm:pointednonsense}), and from this it follows that $d^{1-1/p} \cdot \ol{\bB}_{1,d}^+$ is a minimal dilation hull of $\ol{\bB}_{p,d}^+$ (Corollary \ref{cor:mindil_B1plus}).
Many natural questions remain.

%%%%%%%%%%%%%%%%%%%%%%%%%
\section{Additional background material}

%%%%%%%%%%%%%%%%%%%%%%%%%
\subsection{Connection to completely positive maps}\label{subsec:maps}

A major motivation to study matrix convex sets is their relation to unital completely positive (UCP) maps between operator systems.
By \cite[Theorem 5.1]{DDSS} (which essentially goes back to Arveson \cite[Theorem 2.4.2]{Arv72}), $\cW(B) \subseteq \cW(A)$ if and only there exists a UCP map $\phi : S_A \to S_B$ such that $\phi(A_i) = B_i$ for all $i$.
As a consequence, one obtains the following result (which has already been elaborated on in \cite{DDSS,HKMS15}).
%%%%%%%%%%%%%%%
\begin{proposition}
Let $K_1, K_2 \subseteq \bR^d$ be compact convex sets, $A \in \cB(H_1)_{sa}^d$, and $B \in \cB(H_2)_{sa}^d$.
Suppose that $\cW_1(A) \supseteq K_1$ and $\cW_1(B) \subseteq K_2$.
\begin{enumerate}
\item If $\Wmax{}(K_2) \subseteq \Wmin{}(K_1)$, then there exists a UCP map $\phi: S_A \to S_B$ such that $\phi(A_i) = B_i$ for all $i$.
\item If $\theta(K_1) < \infty$, $K_1 = \cW_1(A)$, and there exists a unital positive map with $A_i \mapsto B_i$, then for all $C \geq \theta(K_1)$, the unital map that sends $I \mapsto I$ and $A_i \mapsto \frac{1}{C}B_i$ ($i=1,\ldots,d$) is completely positive.
\item  If $\theta(K_2) < \infty$, $K_2 = \cW_1(B)$, and there exists a unital positive map with $A_i \mapsto B_i$, then for all $C \geq \theta(K_2)$, the unital map that sends $I \mapsto I$ and $A_i \mapsto \frac{1}{C}B_i$ ($i=1,\ldots,d$) is completely positive.
\end{enumerate}
\end{proposition}
\begin{proof}
(1) follows from the above remarks together with the inclusions $\cW(B) \subseteq \Wmax{}(K_2) \subseteq \Wmin{}(K_1) \subseteq \cW(A)$.
For (2), note that if $\cW_1(A) = K_1$ and there is a unital positive map $A_i \mapsto B_i$, then $\cW_1(B) \subseteq \cW_1(A) = K_1$.
Therefore, for $C \geq \theta(K_1)$, there are containments $\cW(B) \subseteq \Wmax{}(K_1)  \subseteq C \cdot \Wmin{}(K_1) \subseteq \cW(C\cdot A)$, and (2) follows.
Likewise, if $\cW_1(B) = K_2$, then the existence of a unital positive map gives $K_2 \subseteq \cW_1(A)$.
Therefore, for $C \geq \theta(K_2)$, there are containments
$\cW(B) \subseteq \Wmax{}(K_2) \subseteq C \cdot  \Wmin{}(K_2) \subseteq \cW(C \cdot A)$.
\end{proof}

%%%%%%%%%%%%%%%%%%%%%%%%%
\subsection{Polar duality}\label{subsec:polardual}

If $\cS$ is an nc subset of $\bMsad$ then, following Effros and Winkler \cite{EW97}, we define the {\em (matrix) polar dual} of $\cS$ to be the set $\cS^\bullet = \cup_n \cS_n^\bullet$, where
\[
\cS_n^\bullet = \{X \in (M_n)^d_{sa} :  \sum_j A_j \otimes X_j \leq I \textrm{ for all } A \in \cS\}.
\]
Recall that the usual {\em (scalar)  polar dual} $C'$ of a set $C \subseteq \bR^d$ is defined to be
\[
C' = \{x \in \bR^d : \sum_j x_j y_j \leq 1 \FOR y \in C \}.
\]
By \cite[Theorem 4.7]{DDSS},
\[
\Wmin{}(C)^\bullet = \Wmax{}(C') ,
\]
and if $0 \in C$, then
\[
\Wmax{}(C)^\bullet = \Wmin{}(C') .
\]

\begin{lemma}\label{lem:duality_inclusion}
Let $0 \in K \subseteq L$ be closed convex subsets of $\bR^d$.
If $\Wmax{}(K) \subseteq \Wmin{}(L)$, then $\Wmax{}(L') \subseteq \Wmin{}(K')$.
In particular,
\[
\Wmax{}(K) \subseteq C\cdot \Wmin{}(K) \Longrightarrow \Wmax{}(K') \subseteq C \cdot \Wmin{}(K') .
\]
\end{lemma}
\begin{proof}
Applying the polar dual to $\Wmax{}(K) \subseteq \Wmin{}(L)$ reverses the inclusion, and we obtain
\[
\Wmax{}(L') = (\Wmin{}(L))^\bullet\subseteq (\Wmax{}(K))^\bullet = \Wmin{}(K') .
\]
The final assertion follows from $(C\cdot K)' = \frac{1}{C}\cdot K'$.
\end{proof}

%%%%%%%%%%%%%%%%%%%%%%%%%
\subsection{Dilations for operators on finite and on infinite dimensional spaces}\label{subsec:finitetoinfinite}

Let $K, L \subseteq \bC^d$ be compact convex sets.
Consider the following three related --- but not tautologically equivalent --- dilation problems:
\begin{enumerate}
\item For every $d$-tuple of matrices $A$ with $\cW_1(A) \subseteq K$, find a $d$-tuple of normal commuting operators $N$ on some Hilbert space, such that $A \prec N$ and $\sigma(N) \subseteq L$.
\item For every $d$-tuple of matrices $A$ with $\cW_1(A) \subseteq K$, find a $d$-tuple of normal commuting operators $N$, acting on a finite dimensional Hilbert space, such that $A \prec N$ and $\sigma(N) \subseteq L$.
\item For every $d$-tuple of operators $A$ with $\cW_1(A) \subseteq K$, find a $d$-tuple of normal commuting operators $N$ on some Hilbert space, such that $A \prec N$ and $\sigma(N) \subseteq L$.
\end{enumerate}

It turns out that these problems are equivalent; thus, whenever we prove a dilation theorem for matrices we obtain one for operators, and whenever we prove a dilation theorem with operators, we obtain, in fact, a theorem that says that we can dilate matrices to  matrices.
For the record, we state this fact here.

\begin{proposition}\label{prop:dilation_finite_infinite}
Let $K,L \subseteq \bC^d$ be compact convex sets.
Then the following conditions are equivalent:
\begin{enumerate}
\item For every $d$-tuple $A \in \cB(H)^d$ with joint numerical range $\cW_1(A) \subseteq K$, there is a normal $d$-tuple $N$ of bounded operators acting on some Hilbert space, such that $\sigma(N) \subseteq L$ and $A \prec N$.
\item For every $d$-tuple $A \in \bMd$ with joint numerical range $\cW_1(A) \subseteq K$, there is a normal $d$-tuple $N$, acting on a finite dimensional Hilbert space, such that $\sigma(N) \subseteq L$ and $A \prec  N$.
\item $\cW^{max}(K) \subseteq \cW^{min}(L)$.
\end{enumerate}
Moreover, if $A \in (M_n)^d$ has a normal dilation $A \prec N$ such that $\sigma(N) \subseteq L$, then there exists a normal dilation $A \prec \tilde{N}$ of bounded operators, acting on some Hilbert space of dimension at most $2n^3(d+1)+1$, such that $\sigma(\tilde{N}) \subseteq  \ol{\operatorname{ext}(L)}$.
\end{proposition}
\begin{proof}
(1) $\Rightarrow$ (2) follows from Theorem 7.1 of \cite{DDSS}, which says that if there exists any normal dilation $A \prec N$ for $A \in (M_n)^d$, then there exists a normal dilation $A \prec \tilde{N}$, where $\tilde{N}$ is a tuple of operators on a space of dimension $2n^3(d+1)+1$, and $\sigma(\tilde{N}) \subseteq \sigma(N)$.

(2) $\Rightarrow$ (3) is immediate from the characterizations \eqref{eq:Wmax_def2} and \eqref{eq:Wmin_def1}.

Now for (3) $\Rightarrow$ (1), fix $A \in \cB(H)^d$ such that $\cW_1(A) \subseteq K$, and assume $\dim H = \infty$, as this is the only case requiring consideration.
Let $C(L)$ denote the space of continuous complex-valued functions on $L$, and let $Z = (Z_1, \ldots, Z_d) \in C(L)^d$ be the normal tuple of coordinate functions on $L$.
We first prove that there exists a UCP map $\phi: S(Z) \to \cB(H)$ that maps every $Z_i$ onto $A_i$. Let $P_n$ be an increasing net of finite dimensional projections weakly converging to the identity, and put $A^{(n)} = P_n A P_n$, which we may view as an operator on either $H$ or $P_n H$.
By assumption, for all $n$ there is a normal tuple $N^{(n)}$ with $\sigma(N^{(n)}) \subseteq L$ such that $A^{(n)} \prec N^{(n)}$. By \cite[Corollary 4.4]{DDSS}, the matrix range of a normal tuple $N$ is equal to $\Wmin{}(\sigma(N))$.
Thus, $\cW(N^{(n)}) \subseteq \Wmin{}(L) = \cW(Z)$.
It follows from \cite[Theorem 5.1]{DDSS} that there exists a UCP map $S(Z) \to S(N^{(n)})$ mapping $Z_i$ to $N^{(n)}_i$, and after compressing we obtain a UCP map $\psi_n: S(Z) \to S (A^{(n)})$ that maps $Z_i$ to $A^{(n)}_i$ for all $i$.
By Arveson's extension theorem, $\psi_n$ extends to a UCP map $\psi_n : C(L) \to \cB(P_n H)$. Enlarging the Hilbert space gives a contractive and completely positive (CCP) map $\phi_n : C(L) \to \cB(H)$ which sends the constant function $1$ to $P_n$ and $Z_i$ to $A^{(n)}_i$ for all $i$.
Any bounded point-weak limit point $\phi$ of the net $\{\phi_n\}$ in $\textrm{CCP}(C(L),\cB(H))$ is a UCP map that sends $Z_i$ to $A_i$ for all $i$. Finally, we let $\pi : C(L) \to \cB(\widetilde{H})$ be the Stinespring dilation of $\phi$, so the tuple $N := \pi(Z) = (\pi(Z_1), \ldots, \pi(Z_d))$ is a normal dilation for $A$ with $\sigma(N) \subseteq \sigma(Z) = L$. This completes  the proof of (3) $\Rightarrow$ (1).

We now prove the final claim.
Given $A \in (M_n)^d$ and a normal dilation $N \in \cB(K)^d$ with $\sigma(N) \subseteq L$, we note that as above, there is a UCP map $\phi : C(L) \to \cB(K)$ that sends $Z_i$ to $N_i$, where $Z_1, \ldots, Z_d$ are the coordinate functions in $C(L)$. The operator system $S(Z)$ generated by $Z_1, \ldots, Z_d \in C(L)$ is completely isometrically isomorphic to the operator system $S(W)$ generated by the coordinate functions $W_1, \ldots, W_d \in C(\ol{\operatorname{ext}(L)})$. Thus, there is a UCP map $\psi : S(W) \to S(Z)$ mapping $W_i$ to $Z_i$. This gives rise to a UCP map  $\theta = \phi \circ \psi : S(W) \to \cB(K)$ mapping $W_i$ to $N_i$. Applying Arveson extension and then Stinespring dilation to $\theta$, we obtain a unital $*$-representation $\pi : C(\ol{\operatorname{ext}(L)}) \to \cB(\widetilde{K})$ with $N_i  = \theta(W_i)= P_{K}\pi(W_i) P_K$. Setting $M_i = \pi(W_i)$ shows that there is a normal dilation $M$ of $N$, and hence of $A$, with $\sigma(M) \subseteq \ol{\operatorname{ext}(L)}$. Finally, we apply \cite[Theorem 7.1]{DDSS} to the normal dilation $A \prec M$ to obtain a normal dilation $A \prec \tilde{N}$, where $\tilde{N}$ acts on a space of dimension at most $2n^3(d+1)+1$ and $\sigma(\tilde{N}) \subseteq \sigma(M) \subseteq \ol{\operatorname{ext}(L)}$.

\end{proof}

We will also have need of the following reformulation of \cite[Theorem 7.11]{DDSS}. Recall that a set $K \subseteq \bR^d$ is a \textit{convex body} if it is compact, convex, and has nonempty interior.

%%%%%%%%%%%%%%%%%%%%%%%%%
\begin{theorem} \label{general_dilation}
Let $K$ and $L$ be convex bodies in $\bR^d$.
Assume that there is  a $k$-tuple of real $d\times d$ rank one matrices $\lambda$
such that $I_d \in \conv \{\lambda^{(1)},\ldots,\lambda^{(k)}\}$ and such that $\lambda^{(m)}K\subseteq L$ for all $1\leq m  \leq k$.
Then $\Wmax{}(K) \subseteq \Wmin{}(L)$.
That is, for every $X$ such that $\cW_1(X) \subseteq K$, there is a $d$-tuple $T=(T_1,\ldots,T_d)$ of self-adjoint matrices such that
\begin{enumerate}
\item[(1)] $\{T_1,\ldots,T_d\}$ is a commuting family of operators,
\item[(2)] $\sigma(T) \subseteq L$,
\item[(3)] $X \prec T$.
\end{enumerate}
\end{theorem}
It is worth noting that, given $X \in \cB(H)_{sa}^d$, the proof of this theorem provides a concrete construction of the normal dilation $T$ acting on $H \otimes \bC^k$.

%%%%%%%%%%%%%%%%%%%%%%%%%
\begin{remark}
Applying the above theorem to $\ol{\bB}_{1,d}$ and $\ol{\bB}_{\infty,d}$, respectively, one obtains Theorems 10.6 and and 10.10 from \cite{DDSS}, which show that $\Wmax{}(\ol{\bB}_{1,d}) \subseteq \Wmin{}(\ol{\bB}_{\infty,d})$ and that $\Wmax{}(\ol{\bB}_{\infty,d}) \subseteq d \cdot \Wmin{}(\ol{\bB}_{1,d})$, respectively.
\end{remark}

%%%%%%%%%%%%%%%%%%%%%%%%%%%%%
\section{Change of variables and convex geometry}\label{sec:convex_geometry}

\subsection{Affine change of variables}

For $A = (A_1, \ldots, A_d)$ and $M \in M_{m,d}(\bC)$, let $MA$ be the tuple $B = (B_1, \ldots, B_m)$ given by
\[
B_i = \sum_{j=1}^d M_{ij} A_j.
\]
Likewise, given $y \in \bC^m$, we consider the affine transformation $F(x) = Mx + y$ and its action on tuples
\[
F(A) = MA + y,
\]
where $y$ is shorthand for $(y_1 I, \ldots y_m I)$.

Now fix $A \in \cB(H)^d$ and put $B = F(A)$, where $F$ is an affine map $F(x) = Mx + y$ as above. Then $C^*(S_B) \subseteq C^*(S_A)$, so every $\phi \in \UCP(C^*(S_A), M_n)$ restricts to a UCP map $C^*(S_B) \to M_n$. Also, every $\phi \in \UCP(C^*(S_B), M_n)$ extends to a UCP map $C^*(S_A) \to M_n$ by Arveson's extension theorem.
If $X = \phi(A) \in \cW(A)$, then
\[
F(X) = M\phi(A) + y = \phi(MA + y) = \phi(B) \in \cW(B).
\]
On the other hand, if $X = \phi(B)$ then
\[
X = \phi(MA + y) = M \phi(A)  + y = F(\phi(A)).
\]
We conclude
\be\label{eq:FcWA}
\cW(B) = F(\cW(A)).
\ee

%%%%%%%%%%%%%%%%%%%%%%%%%%%%%%%%%
\begin{lemma}\label{lem:affinemapWminmax}
Let $F: \bR^d \to \bR^d$ be an affine and invertible map and let $n \in \bN$.
If $\Wmax{n}(K) \subseteq \Wmin{n}(L)$, then $\Wmax{n}(F(K)) \subseteq \Wmin{n}(F(L))$.
In particular, if $F$ is linear and invertible with $\Wmax{}(K) \subseteq C \cdot \Wmin{}(K)$, then $\Wmax{}(F(K)) \subseteq C \cdot \Wmin{}(F(K))$.
\end{lemma}
\begin{proof}
Let $X \in \Wmax{n}(F(K))$.
This means that $\cW_1(X) \subseteq F(K)$.
Then by \eqref{eq:FcWA},
\[
\cW_1(F^{-1}(X)) = F^{-1}(\cW_1(X)) \subseteq K,
\]
so $Y:=F^{-1}(X) \in \Wmax{n}(K)$.
By assumption, $Y$ has a normal dilation $N$ with $\sigma(N) \subseteq L$.
But then $F(N)$ is a normal dilation for $F(Y) = X$, and $\sigma(F(N)) \subseteq F(L)$.
\end{proof}

Let $H$ be an affine subspace in $\bR^d$ and let $K$ be a closed convex subset in $\bR^d$, so $H\cap K$ is a closed and convex set.
If the dimension of $H$ is $k$, then $H$ can be identified with $\bR^k$. One can also define the projection onto $H$, in the direction orthogonal to $H$; we let $P_H(K)$ denote the projection of $K$ under this orthogonal projection.

%%%%%%%%%%%%%%%%%%%%%%%%%%%%%%%%
\begin{lemma}\label{lem:projWminmax}
Let $K,L \subseteq \bR^d$ be closed convex sets, and let $H$ be an affine subspace of $\bR^d$.
If $\Wmax{}(K) \subseteq \Wmin{}(L)$, then
\[\Wmax{}(H \cap K) \subseteq\Wmin{}(P_H(L)).
\]
\end{lemma}
\begin{proof}
By Lemma \ref{lem:affinemapWminmax}, we may assume that
\[
H = \{(x_1, \ldots, x_k, 0, \ldots, 0) : (x_1, \ldots, x_k) \in \bR^k\}.
\]
Assume that $X \in \Wmax{}(H \cap K)$.
Then $\cW_1(X) \subseteq H \cap K \subseteq K$, and necessarily $X$ has the form $X = (X_1, \ldots, X_k, 0, \ldots, 0)$.
By assumption, there is a dilation $X \prec Y = (Y_1, \ldots, Y_d)$ consisting of commuting self-adjoints such that $\sigma(Y) \subseteq L$.
But then
\[
Y':= (Y_1, \ldots, Y_k, 0, \ldots, 0) = P_H(Y)
\]
is a dilation of $X$, and $\sigma(Y) \subseteq P_H(L)$.
\end{proof}
\begin{remark}\label{remark:cuttingoffzeroes}
Note that nothing changes in the above result if we think of $X$ and $Y'$ as $k$-tuples, rather than $d$-tuples padded with $0$s.
\end{remark}

%%%%%%%%%%%%%%%%%%%%%%%%%%%%%%%%%%%
\subsection{Convex Geometry}\label{subsec:convex_geometry}

If $K$ and $L$ are closed convex sets in $\bR^d$, then there are two established notions of distance which compare $K$ and $L$. First,
\be\label{eq:banachmazursets}
\rho(K,L) := \inf\{C > 0: \exists T \in GL_d(\bR),  K \subseteq T(L) \subseteq C \cdot K\}.
\ee
Note that $\rho(K,L) = \rho(L, K)$.
Even if $K$ and $L$ have interior, it is still possible that $\rho(K, L) = \infty$ (which is understood to mean that there are no $C, T$ such that $K \subseteq T(L) \subseteq C \cdot K$).
However, if $K$ and $L$ are compact with $0$ in the interior of both sets, then $\rho(K, L)$ is finite, and if $K$ and $L$ are also symmetric with respect to the origin, then $\rho(K,L)$ is just the Banach-Mazur distance of the corresponding norms on $\bR^d$. Second,
\be
\mathring{\rho}(K, L) := \inf\{C > 0: \exists x, y \in \bR^d,  T \in GL_d(\bR), y + K \subseteq T(x + L) \subseteq C \cdot (y + K) \}.
\ee
Note that $\mathring{\rho}(K, L) = \mathring{\rho}(L, K)$.

Recall that a \textit{convex body} is a compact convex subset of $\bR^d$ with nonempty interior, and that every point in a convex body may be approximated by interior points. For any convex bodies $K$ and $L$, $\mathring{\rho}(K,L)$ is finite, as we can shift the interiors of $K$ and $L$ to include the origin. Moreover, if $\ol{\bB}_{2,d}$ denotes the closed unit ball of real $\ell^2_d$ space, then it is known that
\be\label{eq:distbounded}
K \subseteq \bR^d \textrm{ is a convex body } \implies \mathring{\rho}(K, \ol{\bB}_{2,d}) \leq d,
\ee
and in \cite{Lei59} Leichtweiss showed that the simplex is the unique maximizer (see also \cite{Palmon}):
\be\label{eq:distoptimizer}
K \subseteq \bR^d \textrm{ is a convex body with } \mathring{\rho}(K, \ol{\bB}_{2,d}) = d \hspace{.05 cm} \iff \hspace{.05 cm} K \textrm{ is a } d\textrm{-simplex}.
\ee
Here and below, a {\em $k$-simplex} is the convex hull of $k+1$ affinely independent points.

The distances $\rho$ and $\mathring{\rho}$ may also be tied to (not necessarily symmetric!) scaling constants that arise naturally in dilation theory:
\be\label{eq:thetadef1}
\theta(K, L) := \inf\{C > 0: \Wmax{}(K) \subseteq C \cdot \Wmin{}(L)\} ,
\ee
\be\label{eq:thetadef2}
\theta(K) := \theta(K, K) ,
\ee
\be\label{eq:mathringthetadef}
\mathring{\theta}(K) := \inf\{\theta(T(x + K)): x \in \bR^d, T \in GL_d(\bR)\} .
\ee
One can use Proposition \ref{prop:dilation_finite_infinite} to show that if $K$ and $L$ are compact and convex, then these infima are actually obtained.
It was also shown in Lemma \ref{lem:affinemapWminmax} and the preceding discussion that an invertible affine transformation factors through $\Wmax{}$ and $\Wmin{}$.
\bes
T \in GL_d(\bR), x \in \bR^d \implies \begin{array}{c} \Wmax{}(x + T(K)) = x + T(\Wmax{}(K))  \\  \\ \Wmin{}(x + T(K)) = x + T(\Wmin{}(K)) \end{array}
\ees
It follows that $\theta$ is invariant under invertible \textit{linear} transformations.
\bes
T \in GL_d(\bR) \implies \begin{array}{c} \theta(K, L) =  \theta(T(K), T(L)) \\ \\ \theta(K) = \theta(T(K)) \end{array}
\ees
Note, however, there is no reason to believe that $\theta$ is invariant under invertible \textit{affine} transformations, as translations and linear maps do not generally commute. In fact, there exist numerous counterexamples to affine invariance of $\theta$ in \cite{DDSS} and this manuscript. Applying reductions to $\mathring{\theta}$ yields
\bes
\mathring{\theta}(K) = \inf\{\theta(a + K): a \in \bR^d\}.
\ees
Similarly, since
\bes
\Wmax{}(a + K) \subseteq C \cdot \Wmin{}(a + K) \iff \Wmax{}(K) \subseteq \Wmin{}((C-1)a + C\cdot K),
\ees
it also holds that
\be\label{eq:simpledist}
\mathring{\theta}(K) = \inf\{C > 0: \exists b \in \bR^d, \Wmax{}(K) \subseteq \Wmin{}(b + C\cdot K)\}.
\ee

The following proposition gives a way of estimating the constants $\theta(K)$ and $\mathring{\theta}(K)$, when $K$ can be approximated by a set $L$ for which we know the dilation constants.
\begin{proposition}\label{prop:scaletoscale}
Let $K$ and $L$ be closed convex subsets of $\bR^d$. Then
\bes
\theta(K) \leq \rho(K, L) \, \theta(L)
\ees
and
\bes
\mathring{\theta}(K) \leq \mathring{\rho}(K, L) \, \mathring{\theta}(L).
\ees
\end{proposition}
\begin{proof}
We may certainly assume that $\theta(L), \mathring{\theta}(L), \rho(K,L),$ and $\mathring{\rho}(K, L)$  are finite.

Let $C_1 > \theta(L)$ and let $C_2 > 0, T \in GL_d(\bR)$ be such that $ K \subseteq T(L) \subseteq C_2 \cdot K$. Since $\theta(T(L)) = \theta(L)$, it follows that
\bes
\Wmax{}(K) \subseteq \Wmax{}(T(L)) \subseteq   C_1 \cdot \Wmin{}(T(L)) \subseteq C_1  C_2 \cdot \Wmin{}(K),
\ees
which shows that $\theta(K) \leq C_1  C_2$. Taking an infimum over all applicable $C_1$ and $C_2$ yields $\theta(K) \leq \rho(K,L) \, \theta(L)$.

Similarly, following the reduction in (\ref{eq:simpledist}), let $C_1 > 0$, $b \in \bR^d$ be such that
\bes
\Wmax{}(L) \subseteq \Wmin{}(b + C_1 \cdot L) = b + C_1 \cdot \Wmin{}(L)
\ees
and choose $x, y \in \bR^d, C_2 > 0$ with
\bes
y + L \subseteq T(x + K) \subseteq C_2 \cdot (y + L).
\ees
Then it follows that
\begin{align*}
\Wmax{}(T(x + K))
&\subseteq C_2 \cdot \Wmax{}(y + L) \\
&= C_2 y + C_2 \cdot \Wmax{}(L) \\
&\subseteq C_2 y + C_2 \cdot (b + C_1 \cdot \Wmin{}(L)) \\
& = C_2 y + C_2 b - C_2 C_1 y + C_2 C_1 \cdot \Wmin{}(y + L) \\
&\subseteq C_2 y + C_2 b - C_2 C_1 y + C_2 C_1 \cdot \Wmin{}(T(x+ K))\\
&= C_2 y + C_2 b - C_2 C_1 y + C_2 C_1 T(x) + C_2 C_1 \cdot\Wmin{}(T(K)).
\end{align*}

Setting $z := C_2 y + C_2 b - C_2 C_1 y + C_2 C_1 T(x)$ and rearranging the given containment $\Wmax{}(T(x + K)) \subseteq z + C_2 C_1 \cdot \Wmin{}(T(K))$ shows that
\bes
\Wmax{}(K) \subseteq x + T^{-1}(z) + C_2 C_1 \cdot \Wmin{}(K),
\ees
so by (\ref{eq:simpledist}), $\mathring{\theta}(K) \leq C_2  C_1$.
Taking an infimum over $C_2$ and $C_1$ gives $\mathring{\theta}(K) \leq \mathring{\rho}(K, L) \, \mathring{\theta}(L)$.
\end{proof}

Letting
\bes
\ol{\bB}_{p,d} = \{ {x} \in \bR^d: || {x}||_p \leq 1 \}
\ees
denote the closed unit ball of real $\ell^p_d$ space, we know from \cite[Example 7.24]{DDSS} that $\theta(\ol{\bB}_{2,d}) = d$. Since the Banach-Mazur distance between $\ell^p$ norms is governed by $\rho(\ol{\bB}_{p,d}, \ol{\bB}_{q,d}) \leq d^{|1/p-1/q|}$, the estimate
\bes
\theta(\ol{\bB}_{p,d}) \geq \cfrac{d}{d^{|1/p - 1/2|}} = d^{1-|1/p - 1/2|}
\ees
holds from setting $q = 2$ and applying Proposition \ref{prop:scaletoscale}. Note that the inequality $\rho(\ol{\bB}_{p,d}, \ol{\bB}_{q,d}) \leq d^{|1/p-1/q|}$ follows from examining the identity map from real $\ell^p_d$ space to real $\ell^q_d$ space, and it is generally not an equality, as seen in Chapter 1, Section 8 of \cite{JohnsonBook}. However, such complications only arise when $p < 2 < q$ or $q < 2 < p$, and we will not need to use this case to estimate dilation constants. Theorem \ref{thm:balldilationgeneral} shows that $\theta(\ol{\bB}_{p,d})$ is exactly $d^{1-|1/p - 1/2|}$.

The distance $\rho$ of (\ref{eq:banachmazursets}) is calculated using many different invertible linear transformations $T$. Alternatively, one may use only scaling maps, and from calculations similar to those in Proposition \ref{prop:scaletoscale}, it follows that
\be\label{eq:singlescale}
\alpha \cdot K_1 \subseteq K_2, \hspace{.5 cm} L_2 \subseteq \beta \cdot L_1 \hspace{.5 cm} \implies \hspace{.5 cm} \theta(K_1, L_1) \leq \cfrac{\beta}{\alpha} \cdot \theta(K_2, L_2).
\ee

%%%%%%%%%%%%%%%%%%%%%%%%%%%%%
\begin{proposition}\label{prop:centered}
Let $K \subseteq \bR^d$ be a closed convex set with $K = -K$.
Then $\mathring{\theta}(K) = \theta(K)$.
\end{proposition}
\begin{proof}
The main idea is reminiscent of the proof of \cite[Theorem 5.8]{FNT}.
First, rewrite \eqref{eq:simpledist} as
\[
\mathring{\theta}(K) = \inf\{C > 0: \exists b \in \bR^d, \Wmax{}(K+b) \subseteq \Wmin{}(C\cdot K)\}.
\]
Now let $C > \mathring{\theta}(K)$ and $b \in \bR^d$ be such that $\Wmax{}(K+b) \subseteq \Wmin{}(C\cdot K)$.
Since $K = -K$, $\Wmax{}(K-b) \subseteq \Wmin{}(C\cdot K)$.

Next, let $X \in \Wmax{}(K)$.
Then $X \pm b \in \Wmax{}(K \pm b)$, so by the previous paragraph, $X \pm b \in \Wmin{}(C \cdot K)$.
By convexity of $\Wmin{}(C \cdot K)$, we find that $X = \frac{1}{2}(X+b) + \frac{1}{2}(X-b) \in \Wmin{}(C \cdot K)$. Thus $\theta(K) \leq C$, and taking an infimum gives $\theta(K) \leq \mathring{\theta}(K)$. The reverse inequality is trivial, so we are done.
\end{proof}

%%%%%%%%%%%%%%%%%%%%%%%%%%%%%%%%%%

\section{The simplex is the only convex body that determines a unique matrix convex set}

It is natural to ask, for what compact convex sets $K \subseteq \bR^d$ does it hold that
\be\label{eq:min=maxsets}
\Wmax{}(K) = \Wmin{}(K)?
\ee

We shall now prove that simplices are the unique compact convex sets such that (\ref{eq:min=maxsets}) holds. In Section \ref{sec:cones}, we see that this may be used to improve the analogous uniqueness statement \cite[Corollary 5.3]{FNT} by removing a polyhedral assumption.

\begin{theorem}\label{thm:simplex_unique}
Let $K \subseteq \bR^d$ be a compact convex set.
Then the following are equivalent.
\begin{enumerate}
\item $K$ is a simplex.
\item $\Wmax{}(K) = \Wmin{}(K)$.
\item $\Wmax{2^{d-1}}(K) = \Wmin{2^{d-1}}(K)$.
\end{enumerate}
\end{theorem}
\begin{proof}
If $K$ is contained in a proper affine subspace $H$, then $H \cap K = P_H(K)$.
By Lemma \ref{lem:projWminmax} and Remark \ref{remark:cuttingoffzeroes}, we may assume, without loss of generality, that $K$ is not contained in a proper subspace, and therefore it has non-empty interior.

To show that (1) implies (2), we may first apply Lemma \ref{lem:affinemapWminmax} and assume that $K$ is the standard simplex $\spn \{0,e_1, \ldots, e_d\}$. Now, every $X \in \Wmax{}(K)$ is just a tuple $X = (X_1, \ldots, X_d)$ with $X_i \geq 0$ and $\sum X_i \leq I$, and thus can be considered as a positive operator valued measure on $\{1, \ldots, d\}$.
By Naimark's theorem, $X$ can be dilated to a spectral measure $P$ on $\{1, \ldots, d\}$, which is just a tuple $P = (P_1, \ldots, P_d)$ of mutually orthogonal (hence commuting) projections dilating $X$ (see also Theorem \ref{thm:positivediamond} for a more explicit dilation).

Since (2) clearly implies (3), it remains to prove that $\Wmax{2^{d-1}}(K) = \Wmin{2^{d-1}}(K)$ implies that $K$ is a simplex.
Suppose that (3) holds, and that $K$ is a convex body (so it has non-empty  interior) in $\bR^d$ that is not a simplex.
From (\ref{eq:distoptimizer}), the fact that $K$ is not a simplex implies that $\mathring{\rho}(K, \ol{\bB}_{2,d}) < d$. This means that after applying an invertible affine transformation, we may assume that $\ol{\bB}_{2,d} \subseteq K \subseteq C \cdot \ol{\bB}_{2,d}$ for some $C \in (0, d)$.
By Lemma \ref{lem:affinemapWminmax}, the equation $\Wmax{2^{d-1}}(K) = \Wmin{2^{d-1}}(K)$ is not affected by an affine transformation.
Thus,
\[
\Wmax{2^{d-1}}(\ol{\bB}_{2,d}) \subseteq \Wmax{2^{d-1}}(K) = \Wmin{2^{d-1}}(K) \subseteq \Wmin{2^{d-1}}(C \cdot \ol{\bB}_{2,d}).
\]
By definition of $\Wmax{}$ and $\Wmin{}$, for every $d$-tuple of self-adjoint $2^{d-1} \times 2^{d-1}$ matrices $X$, if $\sum v_i X_i \leq I$ for all $v \in \ol{\bB}_{2,d}$, then $X$ has a normal dilation $N$ with spectrum $\sigma(N) \subseteq C \cdot \ol{\bB}_d$.
But this contradicts \cite[Example 7.24]{DDSS}, which shows that there is a $d$-tuple $X \in \Wmax{2^{d-1}}(\ol{\bB}_d)$ that has no such normal dilation (see also Theorem \ref{thm:everythingball} below).
This contradiction shows that $K$ must be a simplex.
\end{proof}

Theorem \ref{thm:simplex_unique} also implies the following.

\begin{corollary}\label{cor:maxKequalsminL}
Suppose that $K, L \subset \bR^d$ are compact convex sets with $\Wmax{n}(K) = \Wmin{n}(L)$ for at least one $n \geq 2^{d-1}$.
Then $K = L$ is a simplex.
\end{corollary}
\begin{proof}
Given $x \in L$, we know that $X := \bigoplus\limits_{i=1}^n x \in \Wmin{n}(L) = \Wmax{n}(K)$. Therefore $X$ satisfies the linear inequalities that characterize $K$, which proves $x \in K$ by restricting to the diagonal, so we may conclude $L \subseteq K$. Similarly, the claim that $\Wmax{n}(K) = \Wmin{n}(L) \subseteq \Wmax{n}(L)$ may be used to show that $K \subseteq L$, so we have that $K = L$.

Now, $\Wmax{n}(K) = \Wmin{n}(K)$ for some $n \geq 2^{d-1}$. Fix $T \in \Wmax{2^{d-1}}(K)$, so that for any $x \in K$, $S = T \oplus \bigoplus\limits_{i=1}^{n-2^{d-1}} x \in \Wmax{n}(K) = \Wmin{n}(K)$, meaning $S$ has a normal dilation with joint spectrum in $K$. Certainly this also applies to $T$, and $T \in \Wmin{2^{d-1}}(K)$. Since this proves $\Wmax{2^{d-1}}(K) = \Wmin{2^{d-1}}(K)$, we may conclude that $K$ is a simplex using Theorem \ref{thm:simplex_unique}.
\end{proof}

Of course, with any changes in the minimum matrix dimension $2^{d-1}$ used in the statement of Theorem \ref{thm:simplex_unique}, Corollary \ref{cor:maxKequalsminL} changes as well.

\begin{problem}\label{prob:whatdimensionmin=max}
Is there some $n < 2^{d-1}$ such that if $K \subset \bR^d$ is compact and convex with $\Wmax{n}(K) = \Wmin{n}(K)$, then $K$ is a simplex?
\end{problem}

In Section \ref{sec:minhulls} we will see that there is a large family $\mathcal{C}$ of convex bodies, which we call \lq\lq simplex-pointed sets,\rq\rq\hspace{0pt} for which it suffices to check up to level \textit{two}:
\bes
K \in \mathcal{C}, \hspace{.2 cm} \Wmax{2}(K) = \Wmin{2}(K) \hspace{.1 cm} \implies \hspace{.1 cm} K \textrm{ is a simplex}.
\ees

We also recover the following reformulations of Theorems 5.7  and 5.8 from \cite{FNT}. As in \cite{FNT}, a crucial part of the proof is the application of scaling results from convex geometry.
%%%%%%%%%%%%%%%%%%%%%%%%%%
\begin{corollary}\label{cor:dplustwo_symmetric}
If $K \subseteq \bR^d$ is a convex body, then $\mathring{\theta}(K) \leq d+2$.
If, in addition, $K$ is symmetric with respect to the origin (i.e., $K = -K$), then $\theta(K)\leq d$.
\end{corollary}
\begin{proof}
Let $\Delta$ be a nondegenerate simplex in $\bR^d$.
By \cite{Las11}, $\mathring{\rho}(K, \Delta) \leq d+2$.
Moreover, if $K = -K$, then by \cite[Corollary 5.8]{GLMP}, $\mathring{\rho}(K,\Delta) \leq d$.
The proof is completed by invoking Theorem \ref{thm:simplex_unique}, together with Propositions \ref{prop:scaletoscale} and \ref{prop:centered}.
\end{proof}

We also give an alternative proof of the fact that $\theta(K) \leq d$ for symmetric sets, using the techniques of \cite{DDSS}. In \cite{DDSS}, the dilation scale of $d$ was seen under conditions which did not appear to capture all symmetric sets, but which were also not limited to the symmetric setting.

\begin{theorem}\label{thm:dilate_sym_n}
If $K \subseteq \bR^d$ is a symmetric convex body, then $\theta(K) \leq d$. In fact, if $X \in \cB(H)_{sa}^d$ satifies $\cW_1(X) \subseteq K$, then there exist $k \leq \frac{d(d+1)}{2}+1$ points $x_1, \ldots, x_k \in K$ and a normal dilation $X \prec T$ acting on $H \otimes \bC^k$ such that $\sigma(T) \subseteq d \cdot \conv\{x_1, \ldots, x_k\}$.
\end{theorem}
\begin{proof}
We may assume that $K$ is in {\em John position}, meaning that the ellipsoid of maximal volume contained in $K$ is the unit ball (see \cite[Section 2.1]{AGM}).
Indeed, if $K$ is not in John position then it may be moved to John position by a linear transformation ($K$ is symmetric, so no translation is required).
By \cite[Theorem 2.1.10]{AGM}, there exist $k \leq \frac{d(d+1)}{2}+1$, $x_1, \ldots, x_k \in \partial K \cup \partial \bB_{d}$, and positive numbers $c_1, \ldots, c_m$ such that $\sum_m c_m = d$ and
\[
I_d = \sum_{m=1}^k c_m x_m \otimes x_m,
\]
where $x_m \otimes x_m$ denotes the rank one operator $x_m \otimes x_m (x) = \langle x, x_m \rangle x_m$. (This result goes back to F. John, and it appears in passing in \cite[Section 4]{John}).

Now set $\lambda^{(m)} = d x_m \otimes x_m$, so $I_d \in \conv \{\lambda^{(1)},\ldots,\lambda^{(k)}\}$.
Moreover, the fact that $\ol{\bB}_d\subseteq K$ implies that $x_m + \{x_m\}^\perp$ is a supporting hyperplane for $K$ for all $m$; in other words $\lambda^{(m)} K \subseteq d K$ for all $m=1, \ldots, k$.
In fact, $\lambda^{(m)} K \subseteq L$ for all $m$, where $L = d \cdot \conv\{x_1, \ldots, x_k\}$.
Applying Theorem \ref{general_dilation}, we obtain, for every $\cW_1(X) \subseteq K$, the normal dilation $T$ as claimed.
\end{proof}

%%%%%%%%%%%%%%%%%%%%%%%%%%
%\begin{remark}
%An observation of Milman (see \cite[Theorem 14.5]{TOM}) says that if $K \subseteq L$ are two convex bodies with $K$ in maximal volume position within $L$, then one may find $x_1, \ldots, x_k \in \partial K \cap \partial L$ and $y_1, \ldots, y_k \in \partial K' \cap \partial L'$ (with $k \leq d^2$) such that $\langle x_i, y_i \rangle = 1$ for all $i$, together with $c_1, \ldots, c_k > 0$, such that $\sum c_m = d$ and
%\[
%I = \sum_m c_m x_m \otimes y_m.
%\]
%It can be shown that the non-symmetric rank one operators $\lambda^{(m)} = d x_m \otimes y_m$ can also be used together with Theorem \ref{general_dilation} to construct a dilation which demonstrates that $\Wmax{}(K) \subseteq d \cdot \Wmin{}(\conv\{x_1, \ldots, x_k\})$, and in particular $\Wmax{}(K) \subseteq d \cdot \Wmin{}(K)$.
%However, it is not clear whether this generalized approach produces interesting improvements.
%By the way, Milman's observation also applies to convex bodies in complex spaces, but one does not get the constant $d$ in the complex setting using an analogous process, as the proof of Theorem \ref{general_dilation} does not go through (equation (7.4) in \cite{TOM} does not necessarily give rise to a normal operator).
%\end{remark}

From the dilation-theoretic characterization of $\Wmin{}$, we get the following dilation result.
%%%%%%%%%%%%%%%%%%%%%%%%%%
\begin{corollary}\label{cor:dilation_sym}
Let $K \subseteq \bR^d$ be a symmetric compact convex set and $H$ a Hilbert space.
If $A \in \cB(H)^d_{sa}$ has $\cW_1(A) \subseteq K$, then there exists a normal dilation $A \prec N$ with $\sigma(N) \subseteq d \cdot K$.
\end{corollary}

The constant $d$ cannot be improved in general, because the ball $\ol{\bB}_{2,d}$ is symmetric but satisfies $\theta(\ol{\bB}_{2,d})=d$, as seen in \cite{DDSS}. On the other hand, we will see in the remaining sections that for certain symmetric sets, the constant can be significantly improved.

%%%%%%%%%%%%%%%%%%%%%%%%%%
%%%%%%%%%%%%%%%%%%%%%%%%%%

\section{The Euclidean ball}\label{sec:ball}

In this section, we solve the dilation problem for the ball: if $K$ and $L$ are $\ell^2$-balls in $\bR^d$ (perhaps with different centers and radii), then $\Wmax{}(K) \subseteq \Wmin{}(L)$ if and only if there is a $d$-simplex $\Pi$ with $K \subseteq \Pi \subseteq L$.
This information is catalogued, along with a numerical estimate, in Theorem \ref{thm:everythingball}.
Crucial to this pursuit is a modification of \cite[Lemma 7.23]{DDSS}, which considers the self-adjoint unitary $2^{d-1} \times 2^{d-1}$ matrices $F_1^{[d]}, \ldots, F_d^{[d]}$, defined as follows.
\be\label{eq:recursion} \begin{aligned}
F_1^{[1]} &= 1 \\
F_j^{[d]} &= \begin{pmatrix} 0 & 1 \\ 1 & 0 \end{pmatrix} \otimes F_j^{[d-1]}, \hspace{.5 cm} 1 \leq j \leq d - 1, \\
F_d^{[d]} &= \begin{pmatrix} 1 & 0 \\ 0 & -1 \end{pmatrix} \otimes I_{2^{d-2}}.
\end{aligned} \ee
When the number of matrices $d$ is not needed as a superscript (i.e., whenever induction in $d$ is not needed), we shall denote the matrices  as $F_1, \ldots, F_d$.

It was shown in \cite[Lemma 7.23]{DDSS} that the tuple $(F_1, \ldots, F_d)$ is in $\Wmax{}(\ol{\bB}_{2,d})$, and also that $F_1 \otimes F_1 + \ldots + F_d \otimes F_d$ has $d$ as an eigenvalue. From the construction, it is evident that the self-adjoint matrices $F_1, \ldots, F_d$ pairwise anti-commute and are unitary, though these facts were not emphasized. Consider also the \textit{universal} unital $C^*$-algebra
\bes
\mathcal{A} := C^*( x_1, \ldots, x_d \hspace{.2 cm} | \hspace{.2 cm} x_j = x_j^*, \hspace{.07 cm} x_j^2 = 1, \hspace{.07 cm} x_i x_j = - x_j x_i \textrm{ for } i \not= j)
\ees
generated by pairwise anti-commuting self-adjoint unitaries. The relations on the generators of $\mathcal{A}$ impose that $\mathcal{A}$ is spanned by $1$ and by monomial terms $x_{i_1} x_{i_2} \ldots x_{i_j}$ where $i_1 < i_2 < \ldots < i_j$. Thus the vector space dimension of $\mathcal{A}$ is at most $\sum\limits_{k=0}^d \left( \begin{array}{c} d \\ k \end{array} \right) = 2^{d}$. By definition, there is a unital $C^*$-homomorphism $\phi: \mathcal{A} \to C^*(F_1, \ldots, F_d)$ defined by $\phi: x_j \mapsto F_j$, but from direct examination (and inductive proof) we see that for $d \geq 2$, the vector space dimension of $C^*(F_1, \ldots, F_d)$ is also $2^{d}$. This means the kernel of $\phi$ is trivial, so
\be\label{eq:universality}
C^*(F_1, \ldots, F_d) \cong C^*( x_1, \ldots, x_d \hspace{.2 cm} | \hspace{.2 cm} x_j = x_j^*, \hspace{.07 cm} x_j^2 = 1, \hspace{.07 cm} x_i x_j = - x_j x_i \textrm{ for } i \not= j) \textrm{ for } d \geq 2.
\ee

Below we modify the inductive step of \cite[Lemma 7.23]{DDSS} to include a parameter $y \in \bR^d$, which will help us to strengthen the dilation estimates. Note, however, that much like (\ref{eq:universality}), this lemma must start at $d = 2$.

\begin{lemma}\label{lem:themagicmatrices}
Fix $d \geq 2$. Then the tuple $(F_1, \ldots, F_d) = (F_1^{[d]}, \ldots, F_d^{[d]})$ of pairwise anti-commuting, self-adjoint, unitary matrices of dimension $2^{d-1} \times 2^{d-1}$ defined by (\ref{eq:recursion}) has the property that for any $y \in \bR^d$, $\left|\left| \sum\limits_{j=1}^{d} (F_j - y_j) \otimes F_j \right|\right| \geq \sqrt{||y||^2 + (d-1)^2} + 1$.
\end{lemma}
\begin{proof}
First, we prove the case $(y_1, \ldots, y_d) = (0, \ldots, 0, a)$ inductively in $d$. Now, \cite[Lemma 7.23]{DDSS} shows that $T := \sum\limits_{j=1}^{d-1} F_j^{[d-1]} \otimes F_j^{[d-1]}$ has $(d-1)$ as an eigenvalue, with eigenvector denoted $x^t$. (Note that the case $d - 1 = 1$ is included here.) Since $F_j^{[d]} = \begin{pmatrix} 0 & 1 \\ 1 & 0 \end{pmatrix} \otimes F_j^{[d-1]}$ for $1 \leq j \leq d - 1$ and $F_d^{[d]} = \begin{pmatrix} 1 & 0 \\ 0 & -1 \end{pmatrix} \otimes I_{2^{d-2}}$, it follows that
\bes
\sum_{j=1}^{d-1} F_j^{[d]} \otimes F_j^{[d]} + (F_{d}^{[d]} - a I) \otimes F_d^{[d]} = \begin{pmatrix} (1 - a)I& 0 & 0 & T \\ 0 & (-1 - a)I & T & 0 \\ 0 & T & (-1 + a)I & 0 \\ T & 0 & 0 & (1 + a)I \end{pmatrix}
\ees
has an eigenvalue  $\sqrt{a^2 + (d-1)^2} + 1$ with eigenvector $\left( \cfrac{\sqrt{a^2+(d-1)^2} - a}{d-1} \cdot x, 0, 0, x \right)^t$. Therefore the theorem holds for $y$ of the form $(0, \ldots, 0, a)$.

Next, consider any $y \in \bR^d \setminus \{0\}$, $d \geq 2$. There is an orthogonal matrix $U \in O_d(\bR)$ whose final row is $\frac{1}{||y||} y$, and we may apply the transformation $U$ to $(F_1, \ldots, F_d)$ to produce a tuple $(H_1, \ldots, H_d)$. A simple calculation shows that because $U$ is orthogonal, the matrices $H_1, \ldots, H_d$ pairwise anti-commute and are self-adjoint unitaries. Moreover, the vector space dimension of $C^*(H_1, \ldots, H_d)$ is still $2^d$, so just as in (\ref{eq:universality}), $H_1, \ldots, H_d$ is universal for these relations, and there is a unital $C^*$-homomorphism $\phi: H_j \mapsto F_j$. It follows that $\phi \otimes \phi$ is a contraction, so
\be\label{eq:rotationhappens}
\begin{aligned}
\left|\left|\sum\limits_{j=1}^{d-1} H_j \otimes H_j  + (H_d - ||y||) \otimes H_d \right|\right|  &\geq \left|\left|\sum\limits_{j=1}^{d-1} F_j \otimes F_j + (F_d - ||y||) \otimes F_d \right|\right| \\ &\geq \sqrt{||y||^2 + (d-1)^2} + 1.
\end{aligned}
\ee

Finally, because the last row of $U$ is $\cfrac{1}{||y||} \hspace{2pt} y$, we know that by definition,
\bes
H_d = \sum\limits_{j=1}^d \cfrac{1}{||y||} \hspace{2 pt} y_j \cdot F_j,
\ees and because $U$ is orthogonal, we also know that
\bes
\sum\limits_{j=1}^d H_j \otimes H_j = \sum\limits_{j=1}^d F_j \otimes F_j.
\ees
Therefore, it follows that
\bes
\sum\limits_{j=1}^{d-1} H_j \otimes H_j  + (H_d - ||y||) \otimes H_d = \sum\limits_{j = 1}^d (F_j - y_j) \otimes F_j,
\ees
so by (\ref{eq:rotationhappens}), we have reached the estimate $\left|\left| \sum\limits_{j = 1}^d (F_j - y_j) \otimes F_j \right| \right| \geq \sqrt{||y||^2 + (d-1)^2} + 1$.
\end{proof}
Since $F_1, \ldots, F_d$ are anti-commuting self-adjoint unitaries such that $F_1 \otimes F_1 + \ldots + F_d \otimes F_d$ has an eigenvalue $d$, it follows from the diagonalization of the normal tuple $N = (F_1 \otimes F_1, \ldots, F_d \otimes F_d)$ and the triangle inequality that $(1, \ldots, 1)$ is in the joint spectrum of $N$. This claim may certainly be strengthened to include other tuples with $\pm 1$ entries by examining conjugations, but we will not need it. From
\bes
(1, \ldots, 1) \in \sigma(F_1 \otimes F_1, \ldots, F_d \otimes F_d)
\ees
it follows that for $t_1, \ldots, t_d > 0$,
\be\label{eq:antitocomm}
||t_1 F_1 \otimes F_1 + \ldots + t_d F_d \otimes F_d|| = t_1 + \ldots + t_d.
\ee

Moreover, the claim in \cite{DDSS} that
\be\label{eq:themagicmatrices}
(F_1, \ldots, F_d) \in \Wmax{}(\ol{\bB}_{2,d})
\ee
may be viewed as a direct result of pairwise anti-commutation.

\begin{lemma}\label{lem:antinorm}
If $x_1, \ldots, x_n$ are self-adjoint, pairwise anti-commuting elements of a $C^*$-algebra, then it follows that
\[
||x_1 + \ldots + x_n|| = \sqrt{||x_1^2 + \ldots + x_n^2||} \leq \sqrt{||x_1||^2 + \ldots + ||x_n||^2}.
\]
\end{lemma}
\begin{proof}
Apply the $C^*$-norm identity $||A|| = \sqrt{||AA^*||}$ and cancel $x_ix_j + x_jx_i = 0$ for $i \not= j$:
\bes
||x_1 + \ldots + x_n|| = \sqrt{||(x_1 + \ldots + x_n)^2||} = \sqrt{||x_1^2 + \ldots + x_n^2||} \leq \sqrt{||x_1||^2 + \ldots + ||x_n||^2}.
\ees
\end{proof}

Indeed, $(F_1, \ldots, F_d) \in \Wmax{}(\ol{\bB}_{2,d})$ follows easily, as whenever $(\lambda_1, \ldots, \lambda_d) \in \ol{\bB}_{2,d}$, the inequality  $-1 \leq \lambda_1 F_1 + \ldots + \lambda_d F_d \leq 1$ holds due to anti-commutation and Lemma \ref{lem:antinorm}.
The properties of the tuple $(F_1, \ldots, F_d)$ are sufficient to characterize exactly when the maximal matrix convex set over a ball is contained in the minimal matrix convex set over another ball. First, we obtain a numerical estimate.

\begin{lemma}\label{lem:totallyrad}
Let $C_1, C_2 > 0$, $x, y \in \bR^d$ be such that $x + C_1 \cdot (F_1, \ldots, F_d) \in \Wmin{2^{d-1}}(y + C_2 \cdot \ol{\bB}_{2,d})$, where $F_1, \ldots, F_d$ are as in Lemma \ref{lem:themagicmatrices}. Then $C_2 \geq \sqrt{||y-x||^2 + C_1^2(d-1)^2} + C_1$.
\end{lemma}
\begin{proof}
After translation and scaling, it suffices to consider $x = 0$, $C_1 = 1$, and $(F_1, \ldots, F_d) \in \Wmin{2^{d-1}}(y + C_2 \cdot \ol{\bB}_{2,d})$. We will show that $C_2 \geq \sqrt{||y||^2 + (d-1)^2} + 1$.

If there is a normal dilation $(N_1, \ldots, N_d)$ of $(F_1, \ldots, F_d)$ with joint spectrum in $y + C_2 \cdot \ol{\bB}_{2,d}$, then by Lemma \ref{lem:antinorm},
\bes
|| ( N_1 - y_1) \otimes F_1 + \ldots + (N_d - y_d) \otimes F_d|| \leq C_2.
\ees
On the other hand, by Lemma \ref{lem:themagicmatrices},
\bes
|| (F_1 - y_1) \otimes F_1 + \ldots +  (F_d - y_d) \otimes F_d|| \geq \sqrt{||y||^2 + (d-1)^2} + 1.
\ees
Comparing these norms shows that $C_2 \geq \sqrt{||y||^2 + (d-1)^2} + 1$.
\end{proof}

We also show that if $C_2 \geq \sqrt{||y-x||^2 + C_1^2(d-1)^2} + C_1$, then there is a simplex in between $x + C_1 \cdot \ol{\bB}_{2,d}$ and $y + C_2 \cdot \ol{\bB}_{2,d}$. The following lemma provides an estimate that will help prove this claim.

\begin{lemma}\label{lem:optimalsimplexplacement}
Suppose $b > 1$, $K \subseteq \bR^{d-1}$ is compact and convex, and $\Pi \subseteq \bR^d$ is the convex hull of $K \times \{-1\}$ and $(0, \ldots, 0, b)$.
Then
\bes
\ol{\bB}_{2,d} \subseteq \Pi \hspace{.1 cm} \iff \sqrt{\cfrac{b+1}{b-1}} \cdot \ol{\bB}_{2,d-1} \subseteq K
\ees
\end{lemma}
\begin{proof}
Parametrize $\Pi$ by its final coordinate $x_d \in [-1, b]$, $\Pi(x_d) = \{y \in \bR^{d-1}: (y, x_d) \in \Pi\}$. Then $\ol{\bB}_{2,d} \subseteq \Pi$ holds if and only if for each $x_d \in [-1, 1]$, $\sqrt{1 - x_d^2} \cdot \ol{\bB}_{2,d-1} \subseteq \Pi(x_d)$. If
\bes
C(x_d) = \max\{C \geq 0: C \cdot \ol{\bB}_{2,d-1} \subseteq \Pi(x_d)\},
\ees
then $C(x_d)$ is linear, with $C(b) = 0$. Setting $\alpha = C(-1)$, we have $C(x_d) = \cfrac{\alpha(b - x_d)}{b+1}$.
Therefore, $\ol{\bB}_{2,d} \subseteq \Pi$ if and only if for each $x_d \in [-1, 1]$, $\cfrac{\alpha (b-x_d)}{b + 1} \geq \sqrt{1 - x_d^2}$, or rather,
\bes
\alpha \geq \max\limits_{x_d \in [-1,1]} \left\{ \cfrac{b + 1}{b - x_d} \sqrt{1 - x_d^2} \right\}.
\ees
The maximum occurs at the critical point $x_d = 1/b$, so the inequality is equivalent to $\alpha \geq \cfrac{b + 1}{b - \frac{1}{b}}\sqrt{1 - \frac{1}{b^2}} = \sqrt{\cfrac{b + 1}{b - 1}}$.
Since $\alpha = C(-1)$ is exactly the maximum scale of $\ol{\bB}_{2,d-1}$ sitting inside $K$, we are done.
\end{proof}

\begin{lemma}\label{lem:simplexballcontainment}
Let $x, y \in \bR^d$, $C_1 > 0$, and $C_2 = \sqrt{||y-x||^2 + C_1^2(d-1)^2} + C_1$.
Then there is a simplex $\Pi$ with $x + C_1 \cdot \ol{\bB}_{2,d} \subset \Pi \subset y + C_2 \cdot \ol{\bB}_{2,d}$.
\end{lemma}
\begin{proof}
The result is trivial for $d = 1$, so we assume $d \geq 2$. After a rotation, translation, and scaling, we may suppose that $x = (0, \ldots, 0)$, $y = (0, \ldots, 0, a)$ for $a > 0$, $C_1 = 1$, and $C := C_2 = \sqrt{||y-x||^2 + (d-1)^2} + 1$. We will produce a simplex $\Pi$ with $\ol{\bB}_{2,d} \subset \Pi \subset (0, \ldots, 0, a) + C \cdot \ol{\bB}_{2,d}$.

For any $r > d-1$, let $\Delta = \Delta(r) \subseteq \bR^{d-1}$ be a $(d-1)$-simplex with $\frac{r}{d-1} \cdot \ol{\bB}_{2,d-1} \subset \Delta \subset r \cdot \ol{\bB}_{2,d-1}$, as an application of (\ref{eq:distoptimizer}) in dimension $d - 1$.
We wish to find $r$ so that the convex hull
\bes
\Pi = \Pi(r) = \mathrm{Conv}(\Delta \times \{-1\}, \, (0, \ldots, 0, a + C)),
\ees
which is itself a $d$-simplex, satisfies
\bes
\ol{\bB}_{2,d} \subseteq \Pi \subseteq (0, \ldots, 0, a) + C \cdot \ol{\bB}_{2,d}.
\ees
For the claim $\Pi \subseteq (0, \ldots, 0, a) + C \cdot \ol{\bB}_{2,d}$, we must have
\bes
(a + 1)^2 + r^2 \leq C^2.
\ees
On the other hand, for the claim $\ol{\bB}_{2,d} \subseteq \Pi$, Lemma \ref{lem:optimalsimplexplacement} demands
\bes
\cfrac{r}{d-1} \geq \sqrt{\cfrac{(a + C) + 1}{(a + C) - 1}}.
\ees
The two inequalities can be satisfied simultaneously if and only if
\bes
(a+1)^2 + (d-1)^2 \cdot \cfrac{a + C + 1}{a + C - 1} \leq C^2,
\ees
which is equivalent to
\bes
(d-1)^2 \cdot \cfrac{a + C + 1}{a + C - 1} \leq C^2 - (a + 1)^2.
\ees
Canceling linear terms (which are positive for $C > 1$) gives the equivalent form
\bes
(d-1)^2 \leq (C - (a + 1))(a + C - 1) = (C-1)^2 - a^2,
\ees
or
\bes
(C-1)^2 \geq a^2 + (d-1)^2.
\ees
Since this inequality is satisfied by definition, we know $\ol{\bB}_{2,d} \subset \Pi \subset (0, \ldots, 0, a) + C \cdot \ol{\bB}_{2,d}$.
\end{proof}

Finally, we have reached a solution of the dilation problem for $\ell^2$-balls.

\begin{theorem}\label{thm:everythingball}
Fix $C_1, C_2 > 0$, $x, y \in \bR^d$, and let $F_1, \ldots, F_d$ be as in Lemma \ref{lem:themagicmatrices}. Then the following are equivalent.

\begin{enumerate}
\item\label{item:dilate} $\Wmax{}(x +  C_1 \cdot \ol{\bB}_{2,d}) \subseteq \Wmin{}(y + C_2 \cdot \ol{\bB}_{2,d})$

\item\label{item:dilatedimension} $\Wmax{2^{d-1}}(x +  C_1 \cdot \ol{\bB}_{2,d}) \subseteq \Wmin{2^{d-1}}(y + C_2 \cdot \ol{\bB}_{2,d})$

\item\label{item:specifics} $x + C_1 \cdot (F_1, \ldots, F_d) \in \Wmin{2^{d-1}}(y + C_2 \cdot \ol{\bB}_{2,d})$

\item\label{item:numericalestimate} $C_2 \geq \sqrt{||y - x||^2 + C_1^2 (d-1)^2} + C_1$

\item\label{item:simplexcontainment} There is a simplex $\Pi$ with $x + C_1 \cdot \ol{\bB}_{2,d} \subset \Pi \subset y + C_2 \cdot \ol{\bB}_{2,d}$
\end{enumerate}
\end{theorem}
\begin{proof}
First, $(\ref{item:dilate}) \implies (\ref{item:dilatedimension})$ and $(\ref{item:dilatedimension}) \implies (\ref{item:specifics})$ are trivial, $(\ref{item:specifics}) \implies (\ref{item:numericalestimate})$ is exactly Lemma \ref{lem:totallyrad}, and $(\ref{item:numericalestimate}) \implies (\ref{item:simplexcontainment})$ is exactly Lemma \ref{lem:simplexballcontainment}. Finally, $(\ref{item:simplexcontainment}) \implies (\ref{item:dilate})$ holds because any simplex $\Pi$ has $\Wmax{}(\Pi) = \Wmin{}(\Pi)$, as in Theorem \ref{thm:simplex_unique}.

\end{proof}

This allows us to compute the dilation constant $\theta(\cdot)$ for any $\ell^2$-ball, generalizing \cite[Examples 7.22 and 7.24]{DDSS}.

\begin{corollary}\label{cor:allshiftedballscales}
Fix $d \geq 2$, $A > 0, x \in \bR^d$, and consider the ball $x + A \cdot \ol{\bB}_{2,d}$. Then
\bes
\theta(x + A \cdot \ol{\bB}_{2,d}) = \left\{\begin{array}{ccc} \infty, \hfill & & ||x|| \geq A \\ \cfrac{A(d-1)}{\sqrt{A^2 - ||x||^2}} + 1, & & ||x|| < A \end{array}  \right.
\ees
and consequently $\mathring{\theta}(\ol{\bB}_{2,d}) = \theta(\ol{\bB}_{2,d}) = d$.
\end{corollary}
\begin{proof}
Consider $C \geq 1$. By definition, the claim $\theta(x + A \cdot \ol{\bB}_{2,d}) \leq C$ holds if and only if $ \Wmax{}\left(x + A \cdot \ol{\bB}_{2,d}\right) \subseteq \Wmin{}\left(Cx + C A \cdot \ol{\bB}_{2,d} \right)$. From Theorem \ref{thm:everythingball}, $\Wmax{}\left(x + A \cdot \ol{\bB}_{2,d}\right) \subseteq \Wmin{}\left(Cx + C A \cdot \ol{\bB}_{2,d}\right)$ if and only if
\bes
C A \geq \sqrt{||Cx - x||^2 + A^2(d-1)^2}  + A,
\ees
that is,
\be\label{eq:anequationhere}
C - 1 \geq \sqrt{(C-1)^2 \cfrac{||x||^2}{A^2} + (d-1)^2}.
\ee
Since $C \geq 1$, squaring and solving (\ref{eq:anequationhere}) yields the equivalent form
\bes
(C-1)^2\left(1 - \cfrac{||x||^2}{A^2}\right) \geq (d-1)^2.
\ees
For finite $C$, this is impossible unless $||x|| < A$, in which case we reach the desired estimate $C \geq \cfrac{A(d-1)}{\sqrt{A^2 - ||x||^2}} + 1$. 
\end{proof}

\begin{remark}
If $K$ is a compact convex set and $0 \notin K$, then $\theta(K) = \infty$ unless $K$ is a simplex.
Indeed, if $K$ is not a simplex, then $\theta(K) > 1$ by Theorem \ref{thm:simplex_unique}, but the assumption $0 \notin K$ produces an issue at the ground level: $K \nsubseteq C \cdot K$ for $C \in (1, \infty)$, and thus $\Wmax{}(K) \nsubseteq C \cdot \Wmin{}(K)$ for $C \in (1, \infty)$.
This gives $\theta(x + A \cdot \ol{\bB}_{2,d}) = \infty$ for $\|x\|>A$, as in the previous corollary.
For $\|x\| = A$, there is also a \lq\lq positive\rq\rq\hspace{0pt} interpretation for the value $\theta(x + A \cdot \ol{\bB}_{2,d}) = \infty$.
Consider the closed half-space
\[
\mathbb{H} = \{y \in \bR^d : \langle x,y \rangle \geq 0\} = \overline{\cup_{C \geq 1} C\cdot (x + A \cdot \ol{\bB}_{2,d})}
\]
as the infinite inflation of $x + A \cdot \ol{\bB}_{2,d}$, i.e. \lq\lq $\mathbb{H} = \infty \cdot \Wmin{}(x + A \cdot \ol{\bB}_{2,d})$.\rq\rq\hspace{0pt}
The positive interpretation of $\theta(x + A \cdot \ol{\bB}_{2,d}) = \infty$ is then
\[
\Wmax{}(x + A \cdot \ol{\bB}_{2,d}) \subseteq \Wmin{}(\mathbb{H}).
\]
Indeed, one can always find a simplex $\Pi$ with $x + A \cdot \ol{\bB}_{2,d} \subset \Pi \subset \mathbb{H}$.
\end{remark}

As in \cite{DDSS}, there are also consequences for the existence of UCP maps.

\begin{corollary}\label{cor:UCP_and_UP_maps}
Fix $d \geq 2$, let $F_1, \ldots, F_d \in M_{2^{d-1}}(\bC)$ be as in Lemma \ref{lem:themagicmatrices}, and fix $C > 0$, $y \in \bR^d$. Equip $\bS^{d-1} = \partial \ol{\bB}_{2,d}$ with any positive measure of full support, and let $M_{x_1}, \ldots, M_{x_d} \in \cB(L^2(\bS^{d-1}))$ be the multiplication operators for the coordinate functions $x_j$.

\vspace{.05 cm}

\begin{itemize}
\item There is a UCP map $\phi: \cB(L^2(\bS^{d-1})) \to M_{2^{d-1}}(\bC)$ with
\bes
\phi: M_{x_j} \mapsto \cfrac{1}{C}  \, F_j - \cfrac{y_j}{C} \, I
\ees
if and only if $C \geq \sqrt{||y||^2 + (d-1)^2} + 1$.
\item There is a unital positive (not necessarily UCP) map $\psi: \operatorname{span} \{1, M_{x_1}, \ldots, M_{x_d}\} \to M_{2^{d-1}}(\bC)$ with
\bes
\psi: M_{x_j} \mapsto \cfrac{1}{C} \, F_j - \cfrac{y_j}{C} \, I
\ees
if and only if $C \geq ||y|| + 1$.
\end{itemize}
\end{corollary}
\begin{proof}
Because the chosen measure on $\bS^{d-1}$ has full support, we know that $(M_{x_1}, \ldots, M_{x_d})$ has joint spectrum equal to $\bS^{d-1}$, and since
\bes
M_{x_j} \mapsto \cfrac{1}{C} \,  F_j - \cfrac{y_j}{C} \, I
\ees
holds for a linear map if and only if
\be\label{eq:rearrangedUCP}
CM_{x_j} + y_j \, I \mapsto F_j,
\ee
we will use the second formulation. The joint spectrum of $(CM_{x_1} + y_1 I, \ldots, CM_{x_d} + y_d I)$ is $y + C \cdot \bS^{d-1}$, so, as for any normal tuple, the matrix range of this normal tuple is the minimal matrix convex set containing the joint spectrum, i.e. $\Wmin{}(y + C \cdot \ol{\bB}_{2,d})$.

If $C \geq \sqrt{||y||^2 + (d-1)^2} + 1$, then by Theorem \ref{thm:everythingball} we know that $\Wmax{}(\ol{\bB}_{2,d}) \subseteq \Wmin{}(y + C \cdot \ol{\bB}_{2,d})$, and therefore $\mathcal{W}(F_1, \ldots, F_d) \subseteq \Wmin{}(y + C \cdot \ol{\bB}_{2,d}) = \mathcal{W}(CM_{x_1} + y_1 I, \ldots, CM_{x_d} + y_d I)$.
By \cite[Theorem 5.1]{DDSS}, the UCP map (\ref{eq:rearrangedUCP}) exists.

On the other hand, suppose the UCP map (\ref{eq:rearrangedUCP}) exists.
Using \cite[Theorem 5.1]{DDSS} in the other direction, we obtain that $\mathcal{W}(F_1, \ldots, F_d) \subseteq \Wmin{}(y + C \cdot \ol{\bB}_{2,d})$.
In particular, $(F_1, \ldots, F_d) \in \Wmin{2^{d-1}}(y + C \cdot \ol{\bB}_{2,d})$, so by Theorem \ref{thm:everythingball}, $C \geq \sqrt{||y||^2 + (d-1)^2} + 1$, as desired.

Finally, if the map $CM_{x_j} + y_j \, I \mapsto F_j$ need only be unital and positive, on the domain $\operatorname{span} \{1, M_{x_1}, \ldots, M_{x_d}\}$, we may repeat the above argument with the matrix range $\mathcal{W}$ replaced by the numerical range $\mathcal{W}_1$. Since $\mathcal{W}_1(CM_{x_1} + y_1 \, I, \ldots, CM_{x_d} + y_d \, I) = y + C \cdot \ol{\bB}_{2,d}$ and $\mathcal{W}_1(F_1, \ldots, F_d) = \ol{\bB}_{2,d}$, the positive unital map exists if and only if $\ol{\bB}_{2,d} \subseteq y + C \cdot \ol{\bB}_{2,d}$.
That is, $C \geq ||y|| + 1$.
\end{proof}

\begin{remark}
Note that the maps $\phi$ and $\psi$ in Corollary \ref{cor:UCP_and_UP_maps} are defined on different domains. By Arveson's extension theorem, any UCP map $\phi: \operatorname{span}\{1, M_{x_1}, \ldots, M_{x_d}\} \to M_{2^{d-1}}(\bC)$ extends to a UCP map $\phi: \cB(L^2(\bS^{d-1})) \to M_{2^{d-1}}(\bC)$. However, there is no corresponding extension theorem for unital positive maps. If $||y|| + 1 \leq C < \sqrt{||y||^2 + (d-1)^2} + 1$, then it is actually guaranteed that the unital positive map $\psi$ in Corollary \ref{cor:UCP_and_UP_maps} cannot extend to a unital positive map $\psi: C^*(M_{x_1}, \ldots, M_{x_d}) \to M_{2^{d-1}}(\bC)$. In particular, if the extension exists, then it is a unital positive map whose domain is a commutative C*-algebra, so the extension is actually UCP. Applying Arveson extension to the domain $\cB(L^2(\bS^{d-1}))$ would then contradict the first part of Corollary \ref{cor:UCP_and_UP_maps}.
\end{remark}

Given a convex body $K$, if $0$ is an interior point of $K$, then we have an estimate for $\theta(K)$. Namely, since it is impossible to dilate an ellipse $E$ centered at the origin with smaller scale than $\theta(E) = \theta(\ol{\bB}_{2,d}) = d$, we know from Proposition \ref{prop:scaletoscale} that $d = \theta(\ol{\bB}_{2,d}) \leq \rho(\ol{\bB}_{2,d}, K) \, \theta(K)$.
That is, $\theta(K) \geq \cfrac{d}{\rho(\ol{\bB}_{2,d}, K)}$. If equality $\theta(K) = \cfrac{d}{\rho(\ol{\bB}_{2,d}, K)}$ occurs, we may claim that the optimal dilation scale among all translations of $K$ is found when the translation is zero.

\begin{corollary}\label{cor:bestshiftatzero?}
Let $K \subset \bR^d$ be a convex body with $0$ in the interior. Then it follows that $\mathring{\theta}(K) \geq \cfrac{d}{\rho(\ol{\bB}_{2,d},K)}$. In particular, if $\theta(K) = \cfrac{d}{\rho(\ol{\bB}_{2,d},K)}$, then $\mathring{\theta}(K) = \theta(K)$.
\end{corollary}
\begin{proof}

By Proposition \ref{prop:scaletoscale} and Corollary \ref{cor:allshiftedballscales},
\[
\mathring{\rho}(\ol{\bB}_{2,d}, K) \, \mathring{\theta}(K) \geq \mathring{\theta}(\ol{\bB}_{2,d}) = \theta(\ol{\bB}_{2,d}) = d,
\]
that is,
\bes
\mathring{\theta}(K) \geq \cfrac{d}{\mathring{\rho}(\ol{\bB}_{2,d}, K)} \geq \cfrac{d}{{\rho}(\ol{\bB}_{2,d}, K)}.
\ees
\end{proof}

The computation of Banach-Mazur distance
\be\label{eq:banachmazurballonly}
\rho (\ol{\bB}_{2,d}, K) = \inf \{ C > 0: \exists T \in GL_d(\bR), K \subseteq T(\ol{\bB}_{2,d}) \subseteq C \cdot K\}
\ee
includes many possible transformations $T$. In certain circumstances, the infimum may be attained for a scaling map $T = aI$, so the ellipse that most closely resembles $K$ is a ball. Since the ball is orthogonally invariant, we may then extend Corollary \ref{cor:bestshiftatzero?} to include orthogonal images of $K$ in addition to translations.

\begin{corollary}\label{cor:bestaffineatzero?}
Let $K \subseteq \bR^d$ be a convex body with $0$ in the interior, such that the infimum of (\ref{eq:banachmazurballonly}) is attained for a scaling map $T = aI$. Then for any $U \in O_d(\bR)$ and $y \in \bR^d$, $\theta(K, y + U(K)) \geq \cfrac{d}{\rho(\ol{\bB}_{2,d},K)}$.
\end{corollary}
\begin{proof}
Let $\alpha \cdot \ol{\bB}_{2,d} \subseteq K \subseteq \beta \cdot \ol{\bB}_{2,d}$, where $\alpha, \beta > 0$ are such that $\rho(\ol{\bB}_{2,d}, K) = \cfrac{\beta}{\alpha}$. Such a choice exists due to the assumption that the infimum of (\ref{eq:banachmazurballonly}) is attained at a scaling map. Using orthogonal invariance of the ball, we reach the two separate containments $\alpha \cdot \ol{\bB}_{2,d} \subseteq K$ and $y + U(K) \subseteq \beta \cdot \left( \cfrac{1}{\beta} \cdot y + \ol{\bB}_{2,d} \right) $. From (\ref{eq:singlescale}) and Theorem \ref{thm:everythingball}, we find that
\bes
d \leq \theta \left(\ol{\bB}_{2,d}, \frac{1}{\beta} \cdot y + \ol{\bB}_{2,d} \right) \leq \cfrac{\beta}{\alpha} \cdot \theta(K, y + U(K)) = \rho(\ol{\bB}_{2,d}, K) \cdot \theta(K, y + U(K)),
\ees
as desired.
\end{proof}

As above, if the dilation constant $\theta(K)$ is \textit{equal} to the estimate $\cfrac{d}{\rho(\ol{\bB}_{2,d}, K)}$, then under the assumptions of Corollary \ref{cor:bestaffineatzero?} (i.e. that the Banach-Mazur distance between $K$ and the ball is attained using a scaling map), we may conclude that the best dilation of $K$ to a translated, scaled, orthogonal image of $K$ occurs when the translation is zero. We will see in Section \ref{sec:pballs} that if $K$ is the unit ball of $\ell^p$-space on $\bR^d$ or on $\bC^d \cong \bR^{2d}$, then $\theta(K)$ is indeed equal to the estimate $\cfrac{d}{\rho(\ol{\bB}_{2,d}, K)}$ or $\cfrac{2d}{\rho(\ol{\bB}_{2,2d}, K)}$, respectively.

\begin{problem}
Let $K \subseteq \bR^d $ be a convex body with $0$ in the interior. Under what circumstances must $\theta(K) = \cfrac{d}{\rho(\ol{\bB}_{2,d},K)}$? Is it sufficient to assume that $K = -K$?
\end{problem}

%%%%%%%%%%%%%%%%%%%%%%%%%

\section{Sharp constants for $\ell^p$ balls and their positive sections}\label{sec:pballs}

Let
\bes
\ol{\bB}_{p,d}^+ := \ol{\bB}_{p,d} \cap [0,1]^d
\ees
denote the positive section of the closed unit ball in real $\ell^p_d$ space. The methods of \cite{DDSS} (e.g., Proposition 8.1 there) show that the balls and positive sections
admit dilations governed by
\bes
\Wmax{}(\ol{\bB}_{p,d}) \subseteq d \cdot \Wmin{}(\ol{\bB}_{p,d})
\ees
and
\bes
\Wmax{}(\ol{\bB}_{p,d}^+) \subseteq d \cdot \Wmin{}(\ol{\bB}_{p,d}^+),
\ees
so $\theta(\ol{\bB}_{p,d}) \leq d$ and $\theta(\ol{\bB}_{p,d}^+) \leq d$.
In this section we find sharp constants for all $p$. Note that since $\ol{\bB}_{1,d}^+$ is a simplex,
\be\label{eq:positivediamond}
\Wmax{}(\ol{\bB}_{1,d}^+) =  \Wmin{}(\ol{\bB}_{1,d}^+),
\ee
and hence $\theta(\ol{\bB}_{1,d}^+) = 1$ (see Theorem \ref{thm:simplex_unique}).

We begin with an explicit dilation for (\ref{eq:positivediamond}), which may be found by following the first part of the proof of \cite[Theorem 2.1]{Popescu}.
We record this with details for completeness.
\begin{theorem}\label{thm:positivediamond}
Suppose that $(X_1, \ldots, X_d) \in \cB(H)^d_{sa}$ is a tuple of positive operators with $\sum_{i=1}^d X_i \leq I$. Then there exists a normal dilation $(P_1, \ldots, P_d) \in \cB(H \otimes \bC^{d+1})^d_{sa}$ such that the $P_i$ are mutually orthogonal projections. Moreover, the $P_i$ may be chosen with block entries in the real unital $C^*$-algebra generated by $X_1, \ldots, X_d$.
\end{theorem}
\begin{proof}
Consider the row operator
\bes
T := \left[ \sqrt{X_1} \hspace{.3 cm}\sqrt{X_2} \hspace{.3 cm} \ldots \hspace{.3 cm} \sqrt{X_d} \right],
\ees
which satisfies $TT^* \leq I_{H}$ and consequently $T^*T \leq I_{H \otimes \bC^d}$. Let
\bes
\Delta := \sqrt{I_{H \otimes \bC^d} - T^*T},
\ees
so that when $\Delta^2$ is viewed as a $d \times d$ block matrix with entries in $\cB(H)$,
\bes
(\Delta^2)_{ii} = I_H - X_i \hspace{.5 cm} \textrm{and} \hspace{.5 cm} i \not= j \implies (\Delta^2)_{ij} = -\sqrt{X_i}\sqrt{X_j}.
\ees
Define another row operator
\bes
V_i := \left[ \sqrt{X_i} \hspace{.3 cm} \Delta_{i1} \hspace{.3 cm} \Delta_{i2}\hspace{.3 cm} \ldots \hspace{.3 cm} \Delta_{id}  \right],
\ees
so that
\bes
V_iV_i^* = X_i + \sum_{k} \Delta_{ik} \Delta_{ik}^* = X_i + \sum_{k} \Delta_{ik} \Delta_{ki} = X_i + (\Delta^2)_{ii} = I_H.
\ees
It follows that $P_i := V_i^* V_i$ is a projection in $\cB(H \otimes \bC^{d+1})$, and examination of the top left block entry shows that $P_i$ dilates $X_i$. For $i \not= j$, the identity
\bes
V_i V_j^* = \sqrt{X_i} \sqrt{X_j} + \sum_k \Delta_{ik} \Delta_{jk}^* = \sqrt{X_i} \sqrt{X_j} + \sum_k \Delta_{ik} \Delta_{kj} = \sqrt{X_i} \sqrt{X_j} + (\Delta^2)_{ij} = 0
\ees
implies that the projections are mutually orthogonal.
\end{proof}

\begin{remark} Recall that it is possible to dilate any contraction $S$ to a unitary
\bes
U := \begin{pmatrix} S & \sqrt{I - SS^*} \\ \sqrt{I - S^*S} & -S^* \end{pmatrix},
\ees
as in \cite{Halmos50}. In fact, $S$ is not required to be a square operator. If this dilation is applied to $S = T^* = \begin{pmatrix} \sqrt{X_1} \\ \vdots \\ \sqrt{X_d} \end{pmatrix}$ for $X_i$ as in Theorem \ref{thm:positivediamond}, then the first $d$ block rows of $U$ are the same as the row operators $V_i$ obtained in the proof.
\end{remark}

Theorem \ref{thm:positivediamond} demonstrates that (\ref{eq:positivediamond}) may be realized with dilations that annihilate each other, which in general is much stronger than commutation.
This is not surprising, as distinct extreme points of $\ol{\bB}_{1,d}^+$ never have matching nonzero entries (recall that by Proposition \ref{prop:dilation_finite_infinite}, the spectrum of a normal dilation can be pushed to the extreme points).
On the other hand, the following lemma shows that we may reach annihilating dilations of certain tuples $(X_1, \ldots, X_d)$ of positive operators in a very specific way, which will aid us in later computations.

\begin{lemma}\label{lem:poscommgivesann}
Let $(X_1, \ldots, X_d) \in \cB(H)^d_{sa}$ be a tuple of positive operators such that if $i \not= j$, then the only operator $T$ with $0 \leq T \leq X_i$ and $0 \leq T \leq X_j$ is $T = 0$. Suppose that $(Y_1, \ldots, Y_d) \in \cB(K)^d_{sa}$ is a normal dilation of $(X_1, \ldots, X_d)$ such that $Y_i \geq 0$ for each $i$. Then there is a normal dilation $(Z_1, \ldots, Z_d) \in \cB(K)^d_{sa}$ of $(X_1, \ldots, X_d)$ which satisfies $0 \leq Z_i \leq Y_i$, $Z_i Z_j = 0$ for $i \not= j$, and $\sigma(Z_1,\ldots,Z_d) \subseteq \sigma(Y_1,\ldots,Y_d) \cup \{0\}$.
\end{lemma}
\begin{proof}
Since $Y_1, \ldots, Y_d$ are positive, bounded, and commute, by the spectral theorem there is a positive, bounded spectral measure $F$ and a collection of nonnegative measurable functions $g_i$ such that
\bes
Y_i = \int g_i \,dF.
\ees
Let $V: H \to K$ be an isometric embedding satisfying $X_i = V^* Y_i V$. The operator-valued measure $E(\cdot) = V^*F(\cdot)V$, which is positive and bounded but not necessarily spectral, satisfies
\bes
X_i = \int g_i \,dE.
\ees
Fix $n \geq 1$, $i \not= j$, and let $S_{i,j,n} = \{t: g_i(t) > 1/n \textrm{ and } g_j(t) > 1/n \}$, so that $\frac{1}{n}E(S_{i,j,n}) = \int \frac{1}{n}I_{S_{i,j,n}} dE$ is a positive operator that is bounded above by both $X_i$ and $X_j$. Therefore, by the assumptions about $X_i$ and $X_j$, $E(S_{i,j,n}) = 0$, and from a countable union we see that $S := \{t: \exists i\not=j \textrm{ with } g_i(t) > 0 \textrm{ and } g_j(t) > 0 \}$ also has $E(S) = 0$. We are then free to change the value of $g_i$ on $S$:
\bes
X_i = \int g_i \cdot I_{S^c} \,dE.
\ees
It follows that
\bes
Z_i := \int g_i \cdot I_{S^c} \,dF
\ees
are the desired positive dilations of $X_i$. The $Z_i$ annihilate each other because the functions $g_i \cdot I_{S^c}$ are never simultaneously nonzero, and the $Z_i$ are bounded above by $Y_i$ since $g_i \cdot I_{S^c} \leq g_i$. Finally,
\bes
\sigma(Z_1,\ldots,Z_d) = \mathrm{EssRan}(g_1 \cdot I_{S^c}, \ldots, g_d \cdot I_{S^c}) \subseteq \mathrm{EssRan}(g_1,\ldots,g_d) \cup \{0\} = \sigma(Y_1,\ldots,Y_d) \cup \{0\}.
\ees
\end{proof}

Any $d$ self-adjoint operators $Z_1, \ldots, Z_d$ with $||Z_i|| \leq M$ and $Z_i Z_j = 0$ for all $i$ and $j \neq i$ automatically satisfy $Z_1 + \ldots + Z_d \leq M$, so Lemma \ref{lem:poscommgivesann} may be used to relate dilation of positive sets to the positive diamond.

\begin{theorem}\label{thm:BplusWminmax}
For all $p \in [1, \infty]$ and all $d \in \bZ^+$, the positive $\ell^p_d$ ball $\ol{\bB}_{p,d}^+$ satisfies
\be\label{eq:pospball}
\Wmax{}(\ol{\bB}_{p,d}^+) \subseteq d^{1-1/p}\cdot \Wmin{}(\ol{\bB}_{1,d}^+) \subseteq d^{1-1/p}\cdot \Wmin{}(\ol{\bB}_{p,d}^+).
\ee
The scale $d^{1-1/p}$ is optimal: $\theta(\ol{\bB}_{p,d}^+) = \theta(\ol{\bB}_{p,d}^+, \ol{\bB}_{1,d}^+) = d^{1-1/p}$. Moreover, given a tuple $(X_1, \ldots, X_d) \in \cB(H)^d_{sa}$ of positive operators with $\mathcal{W}_1(X) \subseteq \ol{\bB}_{p,d}^+$, there is a normal dilation $(N_1, \ldots, N_d) \in \cB(H \otimes \bC^{d+1})^d_{sa}$ with $N_i = d^{1-1/p} P_i$, where $P_1, \ldots, P_d$ are mutually orthogonal projections with block entries in the real unital $C^*$-algebra generated by $X_1, \ldots, X_d$.
\end{theorem}
\begin{proof}
Suppose $(X_1, \ldots, X_d) \in \cB(H)^d_{sa}$ has $\mathcal{W}_1(X) \subseteq \ol{\bB}_{p,d}^+$. Since this set is contained in $d^{1-1/p} \cdot \ol{\bB}_{1,d}^+$, it follows that the positive operators $X_i$ satisfy $X_1 + \ldots + X_d \leq d^{1-1/p} I$. From Theorem \ref{thm:positivediamond}, there is a normal dilation $(N_1, \ldots, N_d)$ of the form $N_i = d^{1-1/p} P_i$, where $P_1, \ldots, P_d$ are mutually orthogonal projections in $\cB(H \otimes \bC^{d+1})$ with block entries in the real unital $C^*$-algebra generated by $X_1, \ldots, X_d$. For tuples of matrices, the above dilation procedure is the explicit witness to the computation
\bes
\Wmax{}(\ol{\bB}_{p,d}^+) \subseteq d^{1-1/p} \cdot \Wmax{}(\ol{\bB}_{1,d}^+) = d^{1-1/p} \cdot \Wmin{}(\ol{\bB}_{1,d}^+) \subseteq d^{1-1/p} \cdot \Wmin{}(\ol{\bB}_{p, d}^+),
\ees
which establishes both desired containments.

For optimality of the constant $d^{1-1/p}$, suppose that $\Wmax{}(\ol{\bB}_{p,d}^+) \subseteq C \cdot \Wmin{}(\ol{\bB}_{p,d}^+)$. Fix $0 < \varepsilon < 1/d$ and let $P_j \in M_2(\bC)$ be the projection onto the span of $(\sqrt{1-j^2\varepsilon^2}, j\varepsilon)$, so that the top left entry of $P_j$ is $1 - j^2\varepsilon^2$ and the $P_j$ have ranges with pairwise trivial intersection. It follows that the only positive semidefinite matrix bounded above by two distinct $P_i$ and $P_j$ is zero. Define $X_j = d^{-1/p} P_j$, so that $0 \leq X_j \leq d^{-1/p} I_2$. Since $(d^{-1/p}, \ldots, d^{-1/p})$ has $\ell^p$ norm $1$, it follows that $(X_1, \ldots, X_d) \in \Wmax{}(\ol{\bB}_{p,d}^+)$. From $\Wmax{}(\ol{\bB}_{p,d}^+) \subseteq C \cdot \Wmin{}(\ol{\bB}_{p,d}^+)$, there is then an isometry $V: \bC^2 \hookrightarrow H$ and a normal dilation $(Y_1, \ldots, Y_d) \in \cB(H)^d_{sa}$ of positive operators with $||Y_j|| \leq C$. By Lemma \ref{lem:poscommgivesann}, there is a different normal dilation $(Z_1, \ldots, Z_d) \in \cB(H)^d_{sa}$ of positive operators which annihilate each other and have $||Z_j|| \leq C$. The joint spectrum of $(Z_1, \ldots, Z_d)$ is then contained in $C \cdot \ol{\bB}_{1,d}^+$, and
\bes
Z_1 + \ldots + Z_d \leq C.
\ees
From applying $V^*(\cdot)V$ we see that
\bes
X_1 + \ldots + X_d \leq C,
\ees
and from the top-left entry of $X_j$ we reach
\bes
\sum_{j=1}^d d^{-1/p} (1 - j^2 \varepsilon^2) \leq C.
\ees
Taking $\varepsilon \to 0^+$ gives $d \cdot d^{-1/p} = d^{1-1/p} \leq C$.
This shows that $\theta(\ol{\bB}_{p,d}^+) \geq d^{1-1/p}$, and since we have $d^{1-1/p} \geq \theta(\ol{\bB}_{p,d}^+,\ol{\bB}_{1,d}^+) \geq \theta(\ol{\bB}_{p,d}^+) \geq d^{1-1/p}$, the proof is complete.
\end{proof}

That is, the dilation constant for $\ol{\bB}_{p,d}^+$ is the highest $\ell^1$ norm attained in that set. This computation will be refined and generalized in Section \ref{sec:minhulls}.
We may also use the explicit dilation of the positive diamond $\ol{\bB}_{1,d}^+$ found in Theorem \ref{thm:positivediamond} to form an explicit dilation of the full diamond, improving the estimate obtained in \cite{DDSS}. Moreover, these estimates may also be phrased in terms of the self-adjoint matrix ball
\[
\mathcal{B}^{(d)} = \{ X \in \bM^d_{sa}: X_1^2 + \ldots + X_d^2 \leq I\}
\]
after a quick computation.

\begin{lemma}\label{lem:diamondinmatrixball}
If $Y = (Y_1, \ldots, Y_d) \in \cB(H)^d_{sa}$ has $\mathcal{W}_1(Y) \subseteq \ol{\bB}_{1,d}$, then $\sum Y_i^2 \leq I$. In particular, it holds that $\Wmax{}(\ol{\bB}_{1,d}) \subseteq \mathcal{B}^{(d)}$.
\end{lemma}
\begin{proof}
Assume $\mathcal{W}_1(Y) \subseteq \ol{\bB}_{1,d}$. Then for all $\sigma \in \{0,1\}^d$,
\bes
-I \leq (-1)^{\sigma_1} Y_1 + (-1)^{\sigma_2} Y_2 + \ldots + (-1)^{\sigma_d} Y_d \leq I.
\ees
Squaring this inequality (which is valid due to the $C^*$-norm identity) for a fixed $\sigma$ leads to
\bes
0 \leq Y_1^2 + \ldots + Y_d^2 + \sum_{j \not= k} (-1)^{\sigma_j + \sigma_k} Y_j Y_k \leq I,
\ees
and averaging over all choices of $\sigma$ cancels out all cross terms, so
\bes
0 \leq Y_1^2 + \ldots + Y_d^2 \leq I.
\ees
Consequently, $\Wmax{}(\ol{\bB}_{1,d}) \subseteq \mathcal{B}^{(d)}$.
\end{proof}

In \cite{DDSS}, shortly before Corollary 9.9, it is shown that $\mathcal{B}^{(d)} \subseteq \sqrt{d} \cdot \Wmin{}(\ol{\bB}_{2,d})$. The following theorem improves upon this result (in particular, by using the minimal matrix convex set over a diamond instead of a ball) and contains the optimal dilation scale for $\ol{\bB}_{1,d}$ as well.

\begin{theorem}\label{thm:diamonddilationgeneral}
For a list of positive numbers $a_1, \ldots, a_d$, let $D(a_1, \ldots, a_d)$ be the convex hull of $(\pm a_1, 0, \ldots, 0), \ldots, (0, \ldots, 0, \pm a_d)$. Then
\be\label{eq:diamonddogs}
\Wmax{}(\ol{\bB}_{1,d}) \subseteq \Wmin{}(D(a_1, \ldots, a_d))
\ee
holds if and only if $\sum \cfrac{1}{a_i^2} \leq 1$. In fact, if  $\sum \cfrac{1}{a_i^2} \leq 1$, then the stronger claim
\be\label{eq:matrixballinsquishydiamond}
\mathcal{B}^{(d)} \subseteq \Wmin{}(D(a_1, \ldots, a_d))
\ee
holds, and if $(X_1, \ldots, X_d) \in \cB(H)^d_{sa}$ satisfies $\sum X_i^2 \leq I$, then there is a normal dilation $(N_1, \ldots, N_d) \in \cB(H \otimes \bC^{2d+1})^d_{sa}$ with joint spectrum in $D(a_1, \ldots, a_d)$. Moreover, we may insist that the block entries of $N_j$ are in the real unital $C^*$-algebra generated by $X_1, \ldots, X_d$ and that $N_iN_j = 0$ for $i \not= j$.
\end{theorem}
\begin{proof}
Suppose $\Wmax{}(\ol{\bB}_{1,d}) \subseteq \Wmin{}(D(a_1, \ldots, a_d))$, and let $F_1, \ldots, F_d$ be as in Lemma \ref{lem:themagicmatrices}. If a normal tuple $(N_1, \ldots, N_d)$ has joint spectrum in $D(a_1, \ldots, a_d)$, then
\bes
\left|\left|\cfrac{1}{a_1^2} \cdot N_1^2 + \ldots + \cfrac{1}{a_d^2} \cdot N_d^2 \right|\right| \leq 1
\ees
follows because $\left(\cfrac{1}{a_1} \,N_1, \ldots, \cfrac{1}{a_d} \,N_d \right)$ is a normal tuple with joint spectrum in $\ol{\bB}_{1,d} \subseteq \ol{\bB}_{2,d}$.
Now, since $F_1, \ldots, F_d$ pairwise anti-commute, we also see from Lemma \ref{lem:antinorm} that
\be\label{eq:somanyequationsohmygod}
\begin{aligned}
\left|\left| \cfrac{1}{a_1} \cdot N_1 \otimes F_1 + \ldots + \cfrac{1}{a_d} \cdot N_d \otimes F_d \right|\right| = \sqrt{\left|\left| \left( \cfrac{1}{a_1} \cdot N_1 \otimes F_1 + \ldots + \cfrac{1}{a_d} \cdot N_d \otimes F_d \right)^2\right|\right|} = \\
\sqrt{\left|\left| \cfrac{1}{a_1^2} \cdot N_1^2 \otimes I + \ldots + \cfrac{1}{a_j^2}  \cdot N_d^2 \otimes I \right|\right|} = \sqrt{\left|\left|\cfrac{1}{a_1^2}  \cdot N_1^2 + \ldots + \cfrac{1}{a_d^2}  \cdot N_d^2 \right| \right|} \leq 1. \hspace{1.5 cm}
\end{aligned}
\ee
Applying Lemma \ref{lem:antinorm} to the tuple $\left(\cfrac{1}{a_1} \, F_1, \ldots, \cfrac{1}{a_d} \,F_d \right)$ shows that
\bes
\left (\cfrac{1}{a_1} \, F_1, \ldots,  \cfrac{1}{a_d} \, F_d \right) \in   \sqrt{\sum_{j=1}^d \cfrac{1}{a_j^2}} \cdot \Wmax{}(\ol{\bB}_{1,d}) \subseteq \sqrt{\sum_{j=1}^d \cfrac{1}{a_j^2}} \cdot \Wmin{}(D(a_1, \ldots, a_d)),
\ees
so $\left(\cfrac{1}{a_1} \, F_1, \ldots,  \cfrac{1}{a_d} \, F_d\right)$ has a normal dilation $M = \sqrt{\sum\limits_{j=1}^d \cfrac{1}{a_j^2}} \, \cdot N$, $\sigma(N) \subseteq D(a_1, \ldots, a_d)$.
Recalling \eqref{eq:antitocomm} and \eqref{eq:somanyequationsohmygod}, it follows that
\bes
\begin{aligned}
\sum_{j=1}^d \cfrac{1}{a_j^2} = \left|\left| \cfrac{1}{a_1}  \left (\cfrac{1}{a_1} \,  F_1 \right) \otimes F_1 + \ldots + \cfrac{1}{a_d} \left (\cfrac{1}{a_d} \,  F_d \right) \otimes F_d \right|\right| \leq \left|\left| \cfrac{1}{a_1} \, M_1 \otimes F_1 + \ldots + \cfrac{1}{a_d} \, M_d \otimes F_d \right|\right| = \\
\sqrt{\sum_{j=1}^d \cfrac{1}{a_j^2}} \cdot \left|\left| \cfrac{1}{a_1} \, N_1 \otimes F_1 + \ldots + \cfrac{1}{a_d} \, N_d \otimes F_d \right|\right| \leq \sqrt{\sum_{j=1}^d \cfrac{1}{a_j^2}} \,. \hspace{3.5 cm}
\end{aligned}
\ees
That is, $\cfrac{1}{a_1^2} + \ldots + \cfrac{1}{a_d^2}  \leq 1$.

For the other direction, consider any tuple $(X_1, \ldots, X_d) \in \cB(H)^d_{sa}$ which satisfies
\bes
0 \leq X_1^2 + \ldots + X_d^2 = |X_1|^2 + \ldots + |X_d|^2 \leq I.
\ees
First, applying the positive linear functional $\langle (\cdot)  {v},  {v} \rangle$ for arbitrary $ {v} \in H$ shows that
\be\label{eq:thesquaredform}
0 \leq ||\hspace{2 pt} |X_1| \hspace{2 pt}  {v} ||^2 + \ldots + || \hspace{2 pt} |X_d| \hspace{2 pt}  {v} ||^2 \leq ||  {v} ||^2.
\ee
Given any list $a_1, \ldots, a_d$ of positive numbers, the Cauchy-Schwarz Inequality and (\ref{eq:thesquaredform}) give that
\be\label{eq:abssum}
0 \leq \cfrac{1}{a_1} \cdot ||\hspace{2 pt} |X_1| \hspace{2 pt}  {v} || + \ldots + \cfrac{1}{a_d} \cdot || \hspace{2 pt} |X_d| \hspace{2 pt}  {v} || \leq \sqrt{\cfrac{1}{a_1^2} + \ldots + \cfrac{1}{a_d^2}} \cdot || {v}||.
\ee
From the Cauchy-Schwarz Inequality once more we also know that
\bes
0 \leq \left \langle |X_j|  {v},  {v} \right \rangle \leq  || \hspace{2 pt} |X_j| \hspace{2 pt}  {v} || \cdot || {v}||,
\ees
which combined with (\ref{eq:abssum}) gives that
\bes
0 \leq \left \langle \cfrac{1}{a_1} \cdot |X_1|  {v},  {v} \right\rangle + \ldots + \left\langle \cfrac{1}{a_d} \cdot |X_d|  {v},  {v} \right\rangle \leq \sqrt{\cfrac{1}{a_1^2} + \ldots + \cfrac{1}{a_d^2}} \cdot || {v}||^2.
\ees
Since $v$ is arbitrary, this implies
\be\label{eq:ell1sum}
0 \leq \cfrac{1}{a_1} \cdot |X_1| + \ldots + \cfrac{1}{a_d} \cdot |X_d| \leq \sqrt{\cfrac{1}{a_1^2} + \ldots + \cfrac{1}{a_d^2}} \cdot I.
\ee
If $\sum \cfrac{1}{a_j^2} \leq 1$, then we have reached $\sum \cfrac{1}{a_j} \, |X_j| \leq I$. In this case, write $X_j = A_j - B_j$, $|X_j| = A_j + B_j$, where $A_j$ and $B_j$ are positive and determined from the functional calculus (using real coefficients). We then know that
\bes
0 \leq \cfrac{1}{a_1} \cdot A_1 + \cfrac{1}{a_1} \cdot B_1 + \ldots +  \cfrac{1}{a_d} \cdot A_d + \cfrac{1}{a_d} \cdot B_d \leq I.
\ees
From Theorem \ref{thm:positivediamond}, the tuple $\left( \cfrac{1}{a_1} \cdot A_1, \cfrac{1}{a_1} \cdot B_1, \ldots, \cfrac{1}{a_d} \cdot A_d,\cfrac{1}{a_d} \cdot B_d \right) \in \cB(H)^{2d}_{sa}$ may be dilated to a tuple $(K_1, L_1, \ldots, K_d, L_d) \in \cB(H \otimes \bC^{2d+1})^{2d}_{sa}$ of positive contractions which pairwise annihilate each other. It follows that $N_j = a_j \cdot (K_j - L_j)$ is a self-adjoint dilation of $X_j$ such that $N_i N_j = 0$ for $i \not= j$ and $||N_j|| \leq a_j$. The joint spectrum of $(N_1, \ldots, N_d)$ is then in $D(a_1, \ldots, a_d)$. Finally, tracing the steps used shows that the block entries of the $N_j$ are in the real unital $C^*$-algebra generated by $X_1, \ldots, X_d$. With Lemma \ref{lem:diamondinmatrixball}, this explicit dilation procedure shows that if $\sum \cfrac{1}{a_j^2} \leq 1$, then $\Wmax{}(\ol{\bB}_{1,d}) \subseteq \mathcal{B}^{(d)} \subseteq \Wmin{}(D(a_1, \ldots, a_d))$.
\end{proof}

From the constant sequence $a_1 = \ldots = a_d = \sqrt{d}$, we obtain
\bes
\Wmax{}(\ol{\bB}_{1,d}) \subseteq \sqrt{d} \cdot \Wmin{}(\ol{\bB}_{1,d}),
\ees
and this constant is optimal. That is, $\theta(\ol{\bB}_{1,d}) = \sqrt{d}$. It follows that the dilation scale $\sqrt{d}$ is also optimal in the containment
\bes
\mathcal{B}^{(d)} \subseteq \sqrt{d} \cdot \Wmin{}(\ol{\bB}_{1,d}).
\ees
This result improves the claim that $\mathcal{B}^{(d)} \subseteq \sqrt{d} \cdot \Wmin{}(\ol{\bB}_{2,d})$ from \cite{DDSS}.

Consider now the dual problem of dilating the cube $\ol{\bB}_{\infty,d} = [-1,1]^d$.
First, using duality (Lemma \ref{lem:duality_inclusion}), we immediately obtain
\bes
\Wmax{}([-1,1]^d) \subseteq \Wmin{}\left([-a_j,a_j]^d \right) \iff \cfrac{1}{a_1^2} + \ldots + \cfrac{1}{a_d^2} \leq 1,
\ees
and in particular
\bes
\Wmax{}([-1,1]^d) \subseteq \sqrt{d} \cdot \Wmin{}([-1, 1]^d),
\ees
with $\theta([-1,1]^d) = \sqrt{d}$ optimal.
So, we know that given a $d$-tuple $X$ of self-adjoint contractions, we should be able to find a $d$-tuple  $N$ of commuting self-adjoints each of norm at most $\sqrt{d}$, such that $X \prec N$. Our next goal is to find such a dilation explicitly; this will be achieved in Theorem \ref{thm:cubedilationgeneral} below. A crucial step in this process is the dilation technique of Halmos in \cite{Halmos50}; if a self-adjoint operator $X$ has $||X|| \leq 1$, then
\be\label{eq:scaledunitarydilation}
Y :=  \begin{pmatrix} X & \sqrt{I - X^2} \\ \sqrt{I - X^2} & -X \end{pmatrix}
\ee
is a self-adjoint dilation of $X$ which has $Y^2 = I$.
That is, $Y$ is self-adjoint and unitary.
The block entries of $Y$ also belong to the real unital $C^*$-algebra generated by $X$.

\begin{theorem}\label{thm:cubedilationgeneral}
Let $a_1, \ldots, a_d$ be a list of positive numbers. Then
\be\label{eq:wmaxpropertyp}
\Wmax{}([-1,1]^d) \subseteq \Wmin{}\left(\prod_{j=1}^d [-a_j, a_j]\right)
\ee
holds if and only if $\sum \cfrac{1}{a_j^2} \leq 1$. In this case, if $(X_1, \ldots, X_d) \in \cB(H)^d_{sa}$ is such that $||X_j|| \leq 1$ for each $j$, then there is a normal dilation $(N_1, \ldots, N_d) \in \cB(H \otimes \bC^{4^{d-1}})^d_{sa}$ with $||N_j|| \leq a_j$ for each $j$. The $N_j$ may be chosen with block entries in the real unital $C^*$-algebra generated by $X_1, \ldots, X_d$.
\end{theorem}
\begin{proof}
First, suppose that $\Wmax{}([-1,1]^d) \subseteq \Wmin{}\left(\prod_{j=1}^d [-a_j, a_j]\right)$ holds. From Lemma \ref{lem:themagicmatrices}, (\ref{eq:antitocomm}), and (\ref{eq:themagicmatrices}), the tuple $(F_1, \ldots, F_d) \in \Wmax{}([-1,1]^d)$ is such that for all $t_1, \ldots, t_d > 0$, $||t_1 F_1 \otimes F_1 + \ldots + t_d F_d \otimes F_d|| = t_1 + \ldots + t_d$. On the other hand, if $(M_1, \ldots, M_d)$ is a normal tuple of self-adjoints with $||M_j|| \leq a_j$, then by Lemma \ref{lem:antinorm}, $|| t_1 M_1 \otimes F_1 + \ldots + t_d M_d \otimes F_d|| \leq \sqrt{t_1^2 a_1^2 + \ldots + t_d^2 a_d^2}$. Therefore we must have
\bes
\sqrt{t_1^2 a_1^2 + \ldots + t_d^2 a_d^2} \geq t_1 + \ldots + t_d.
\ees
Plugging in $t_j = \cfrac{1}{a_j^2}$ yields
\bes
\sqrt{\cfrac{1}{a_1^2} + \ldots + \cfrac{1}{a_d^2}} \geq \cfrac{1}{a_1^2} + \ldots + \cfrac{1}{a_d^2},
\ees
that is, $\cfrac{1}{a_1^2} + \ldots + \cfrac{1}{a_d^2} \leq 1$.

We will show the converse dilation procedure inductively, noting that $d = 1$ is trivial. Therefore, we may suppose the theorem holds for all lists $b_1, \ldots, b_d$ of positive numbers such that $\sum \limits_{j=1}^d \cfrac{1}{b_j^2} \leq 1$. Suppose $a_1, \ldots, a_{d+1}$ is a list of positive numbers with $\sum \limits_{j=1}^{d+1} \cfrac{1}{a_j^2} \leq 1$. Since $a_{d+1} > 1$, write $a_{d+1} = \sqrt{1 + 1/s^2}$ for $s > 0$, and define $b_j = \frac{a_j}{\sqrt{1+s^2}}$. Then
\bes
\cfrac{1}{b_1^2} + \ldots + \cfrac{1}{b_d^2} = (1+s^2) \cdot \left(\cfrac{1}{a_1^2} + \ldots + \cfrac{1}{a_d^2} \right) \leq (1+s^2) \cdot \left(1 - \cfrac{1}{a_{d+1}^2}  \right) = (1 + s^2) \cdot \left( 1 - \cfrac{1}{1 + 1/s^2} \right) = 1.
\ees
Let $X_1, \ldots, X_{d+1} \in \cB(H)$ be self-adjoint with $||X_j|| \leq 1$. Apply the Halmos dilation (\ref{eq:scaledunitarydilation}): $Y_1, \ldots, Y_{d+1} \in \cB(H \otimes \bC^2)$ are self-adjoint dilations with $Y_j^2 = I$. The block operators
\bes
Z_j =
\begin{pmatrix}
Y_j & \cfrac{s}{2} \cdot(Y_{d+1}Y_j - Y_j Y_{d+1}) \\ \cfrac{s}{2} \cdot (Y_j Y_{d+1} - Y_{d+1}Y_j ) & Y_j
\end{pmatrix}, 1 \leq j \leq d, \hspace{.6 cm} Z_{d+1} = \begin{pmatrix} Y_{d+1} & \cfrac{1}{s} \cdot I \\\cfrac{1}{s} \cdot I & -Y_{d+1} \end{pmatrix}
\ees
are self-adjoint operators in $\cB(H \otimes \bC^4)$ such that $Z_j$ commutes with $Z_{d+1}$ for each $j$. Moreover, the block diagonal and off-diagonal entries of $Z_j$ anti-commute, so the norms of $Z_j$ are then bounded by Lemma \ref{lem:antinorm}, which gives
\bes
||Z_j|| \leq \sqrt{||Y_j||^2 + \left| \left| \frac{s}{2 } \cdot (Y_{d+1}Y_j - Y_j Y_{d+1}) \right| \right|^2} \leq  \sqrt{1 + s^2}, \hspace{.5 cm} 1 \leq j \leq d,
\ees
\bes
||Z_{d+1}|| \leq \sqrt{||Y_{d+1}||^2 + \left| \left| \frac{1}{s} \cdot I \right| \right|^2 } \leq \sqrt{1+1/s^2} = a_{d+1}.
\ees
Since $\sum\limits_{j=1}^d \cfrac{1}{b_j^2} \leq 1$, the inductive assumption (scaled by $\sqrt{1+s^2}$) shows that $(Z_1, \ldots, Z_d) \in \cB(H \otimes \bC^4)^d_{sa}$ admits a normal dilation $(N_1, \ldots, N_d) \in \cB(H \otimes \bC^4 \otimes \bC^{4^{d-1}})^d_{sa}$ with $||N_j|| \leq \sqrt{1 + s^2} \cdot b_j = a_j$ such that the block entries of each $N_j$ are in the real unital $C^*$-algebra generated by $Z_1, \ldots, Z_d$.
The $N_j$ therefore also commute with $N_{d+1} := \bigoplus\limits_{m=1}^{4^{d-1}} Z_{d+1}$, which has $||N_{d+1}|| = ||Z_{d+1}|| \leq a_{d+1}$.
Tracing back the steps used shows that the block entries of $N_1, \ldots, N_{d+1}$ are in the real unital $C^*$-algebra generated by the original self-adjoints $X_1, \ldots, X_{d+1}$, so the induction is complete. Applied to tuples of matrices, the above dilation procedure shows that if $\sum_{j=1}^d \cfrac{1}{a_j^2} \leq 1$, then $\Wmax{}([-1,1]^d) \subseteq \Wmin{}\left( \prod [-a_j, a_j] \right)$.
\end{proof}

As before, the constant sequence $a_1 = \ldots = a_d = \sqrt{d}$ gives
\be\label{eq:cubecontainment}
\Wmax{}([-1,1]^d) \subseteq \sqrt{d} \cdot \Wmin{}([-1,1]^d),
\ee
and this constant is optimal. That is, $\theta(\ol{\bB}_{\infty,d}) = \sqrt{d}$.

\begin{remark}
If one takes the containment of Theorem \ref{thm:cubedilationgeneral},
\bes
\Wmax{}([-1,1]^d) \subseteq \Wmin{}\left(\prod_{j=1}^d [-a_j, a_j]\right) \iff \sum \cfrac{1}{a_j^2} \leq 1,
\ees
and applies a scaling transformation $x_j \mapsto \frac{1}{a_j} x_j$, the result is
\be\label{eq:totheleft,totheleft}
\Wmax{}\left(\prod_{j=1}^d \left[-\cfrac{1}{a_j}, \cfrac{1}{a_j} \right]\right) \subseteq \Wmin{}([-1,1]^d) \iff \sum \cfrac{1}{a_j^2} \leq 1.
\ee
Since each set $\prod_{j=1}^d \left[-\frac{1}{a_j}, \frac{1}{a_j} \right]$ is contained in $\ol{\bB}_{2,d}$ due to the inequality $\sum \frac{1}{a_j^2} \leq 1$, it is very tempting to replace the left hand side of (\ref{eq:totheleft,totheleft}) with $\Wmax{}(\ol{\bB}_{2,d})$ in order to construct and prove a more general result. However, if $d \geq 2$, it holds that
\be\label{eq:hopeisnotgoodenough}
\Wmax{}(\ol{\bB}_{2,d}) \not\subseteq \Wmin{}([-1,1]^d),
\ee
as $[-1,1]^d \subseteq \sqrt{d} \cdot \ol{\bB}_{2,d}$ and $\theta(\ol{\bB}_{2,d}) = d > \sqrt{d}$.
\end{remark}

We have now computed $\theta(\ol{\bB}_{1,d}) = \sqrt{d} = \theta(\ol{\bB}_{\infty, d})$, where we note that the proof of optimality of the constant $\sqrt{d}$ essentially returns to the claim $\theta(\ol{\bB}_{2,d}) = d$ from \cite[Example 7.24]{DDSS}. Indeed, we may use the explicit dilations of the diamond and the cube, as well as the minimal dilation constant of the ball, to interpolate $\theta(\ol{\bB}_{p,d})$ for all $1 \leq p \leq \infty$.

\begin{theorem}\label{thm:balldilationgeneral}
For $1 \leq p \leq \infty$, $\theta(\ol{\bB}_{p,d}) = d^{1-|1/p - 1/2|}$.
\end{theorem}
\begin{proof}
It is known from Theorems \ref{thm:diamonddilationgeneral} and \ref{thm:cubedilationgeneral} that $\theta(\ol{\bB}_{\infty, d}) = \sqrt{d} = \theta(\ol{\bB}_{1, d})$, and the Banach-Mazur distance between $\ell^p$ norms is bounded by $\rho(\ol{\bB}_{p,d}, \ol{\bB}_{q,d}) \leq d^{|1/p - 1/q|}$. We first use the boundary cases $p = 1$ and $p = \infty$ to obtain dilation estimates from Proposition \ref{prop:scaletoscale}:
\bes
1 \leq p \leq 2 \implies \theta(\ol{\bB}_{p,d}) \leq \rho(\ol{\bB}_{p, d}, \ol{\bB}_{1,d}) \, \theta(\ol{\bB}_{1,d}) \leq d^{1 - 1/p} \cdot d^{1/2} = d^{3/2 - 1/p} = d^{1-|1/p - 1/2|},
\ees
\bes
2 \leq p \leq \infty \implies \theta(\ol{\bB}_{p,d}) \leq \rho(\ol{\bB}_{p,d}, \ol{\bB}_{\infty, d}) \, \theta(\ol{\bB}_{\infty,d}) \leq d^{1/p} \cdot d^{1/2} = d^{1/2 + 1/p} = d^{1-|1/p - 1/2|}.
\ees
Then we use the fact that $\theta(\ol{\bB}_{2,d}) = d$ (see \cite[Example 7.24]{DDSS} or Corollary \ref{cor:allshiftedballscales}) for minimality:
\bes
1 \leq p \leq \infty \implies d = \theta(\ol{\bB}_{2,d}) \leq \rho(\ol{\bB}_{2,d}, \ol{\bB}_{p,d}) \, \theta(\ol{\bB}_{p,d}) \leq d^{|1/p - 1/2|} \cdot \theta(\ol{\bB}_{p,d}) \implies
\ees
\bes
\theta(\ol{\bB}_{p,d}) \geq d^{1 - |1/p - 1/2|}.
\ees
\end{proof}

Because the constants $\theta(\ol{\bB}_{p,d})$ given in Theorem \ref{thm:balldilationgeneral} are achieved by pushing to the boundary cases $p = 1$ and $p = \infty$, note that different properties of the explicit dilations are inherited. For example, the dilation technique bearing witness to
\bes
1 \leq p \leq 2 \implies \Wmax{}(\ol{\bB}_{p,d}) \subseteq d^{3/2-1/p} \cdot \Wmin{}(\ol{\bB}_{p,d})
\ees
produces dilations which annihilate each other. This places tighter bounds on the joint spectrum, so that
\be\label{eq:smallpbettersets}
1 \leq p \leq 2 \implies \Wmax{}(\ol{\bB}_{p,d}) \subseteq d^{3/2-1/p} \cdot \Wmin{}(\ol{\bB}_{1,d})
\ee
holds.
Thus,
\[
\theta(\ol{\bB}_{p,d},\ol{\bB}_{1,d}) = \theta(\ol{\bB}_{p,d}) = d^{3/2-1/p}, \quad \textrm{ for } \quad p \in [1,2].
\]
Moreover, for $1 \leq p \leq 2$, the technique dilates $\cB(H)^d_{sa}$ to $\cB(H \otimes \bC^{2d+1})^d_{sa}$, although this dimension is not proved minimal.

On the other hand, we know that the dilation scheme
\bes
p > 2 \implies \Wmax{}(\ol{\bB}_{p,d}) \subseteq d^{1/2+1/p} \cdot \Wmax{}(\ol{\bB}_{p,d})
\ees
\textit{cannot} always be achieved with annihilating operators, as otherwise the joint spectrum of a dilation would actually be in $d^{1/2+1/p} \cdot \ol{\bB}_{1,d} \subset d^{1/2+1/p} \cdot \ol{\bB}_{2,d}$. Since $d^{1/2+1/p}$ is strictly less than $d$ and $\ol{\bB}_{2,d} \subseteq \ol{\bB}_{p,d}$, such a claim would contradict the minimality of $\theta(\ol{\bB}_{2,d}) = d$. However, this dilation is obtained from the cube:
\be\label{eq:bigpbettersets}
p \geq 2 \implies \Wmax{}(\ol{\bB}_{p,d}) \subseteq \Wmax{}(\ol{\bB}_{\infty,d}) \subseteq d^{1/2} \cdot \Wmin{}(\ol{\bB}_{\infty,d}).
\ee
Thus,
\[
\theta(\ol{\bB}_{p,d}, \ol{\bB}_{\infty,d}) = \theta(\ol{\bB}_{\infty,d}) = d^{1/2} , \quad \textrm{ for } \quad p \in [2, \infty].
\]
Similarly, for $p \geq 2$, the technique dilates $\cB(H)^d_{sa}$ to $\cB(H \otimes \bC^{4^{d-1}})^d_{sa}$, but again we do not know whether this dimension is necessarily minimal.

A similar computation holds for the closed unit balls of $\ell^p$ space in $\bC^d \cong \bR^{2d}$.

\begin{corollary}
Let $\ol{\bB}_{p,d}(\bC)$ denote the closed unit ball of $\bC^d$ in the $\ell^p$ norm, viewed as a subset of $\bR^{2d}$. Then
\bes
\theta(\ol{\bB}_{p,d}(\bC)) = 2 \, \theta( \ol{\bB}_{p,d}) = 2 \, d^{1 - |1/p - 1/2|}.
\ees
\end{corollary}
\begin{proof}
We proceed as in the real case, continuing to use the boundary cases from real $\ell^p$ balls: $\theta(\ol{\bB}_{\infty, 2d}) = \sqrt{2d} = \theta(\ol{\bB}_{1,2d})$. Note that the Banach-Mazur distance is governed by
\bes
\rho( \ol{\bB}_{p,d}(\bC), \ol{\bB}_{\infty, 2d}) \leq \sqrt{2} \, \rho( \ol{\bB}_{p,d}(\bC), \ol{\bB}_{\infty, d}(\bC)) \leq \sqrt{2} \, d^{1/p},
\ees
\bes
\rho( \ol{\bB}_{p,d}(\bC), \ol{\bB}_{1, 2d}) \leq \sqrt{2} \, \rho( \ol{\bB}_{p,d}(\bC), \ol{\bB}_{1, d}(\bC)) \leq \sqrt{2} \, d^{1-1/p},
\ees
so we have
\bes
1 \leq p \leq 2 \implies \theta(\ol{\bB}_{p,d}(\bC)) \leq \rho(\ol{\bB}_{p, d}(\bC), \ol{\bB}_{1,2d}) \, \theta(\ol{\bB}_{1,2d}) \leq \sqrt{2} \, d^{1 - 1/p} \, \sqrt{2d} = 2 \, d^{3/2 - 1/p},
\ees
\bes
2 \leq p \leq \infty \implies \theta(\ol{\bB}_{p,d}(\bC)) \leq \rho(\ol{\bB}_{p,d}(\bC), \ol{\bB}_{\infty, 2d}) \, \theta(\ol{\bB}_{\infty,2d}) \leq \sqrt{2} \, d^{1/p} \, \sqrt{2d} = 2 \, d^{1/2 + 1/p}.
\ees
Finally, since real and complex $\ell^2$-balls (of the appropriate dimension) are identical, we have
\bes
\rho( \ol{\bB}_{p,d}(\bC), \ol{\bB}_{2, 2d}) = \rho( \ol{\bB}_{p,d}(\bC), \ol{\bB}_{2, d}(\bC)) \leq d^{|1/p-1/2|},
\ees
which along with $\theta(\ol{\bB}_{2, 2d}) = 2d$ gives that
\bes
1 \leq p \leq \infty \implies 2d = \theta(\ol{\bB}_{2,2d}) \leq \rho(\ol{\bB}_{2,2d}, \ol{\bB}_{p,d}(\bC)) \, \theta(\ol{\bB}_{p,d}(\bC)) = d^{|1/p - 1/2|} \cdot \theta(\ol{\bB}_{p,d}) \implies
\ees
\bes
\theta(\ol{\bB}_{p,d}(\bC)) \geq 2 \, d^{1 - |1/p - 1/2|} = \left\{ \begin{array}{ccccc} 2 \, d^{3/2 - 1/p}, & & 1 \leq p \leq 2 \\ 2 \, d^{1/2 + 1/p}, & & 2 \leq p \leq \infty \end{array}\right. .
\ees
\end{proof}

A special case of the above is that the optimal dilation constant for the polydisk $\ol{\bB}_{\infty, d}(\bC) = \overline{\bD}^d$ is given by
\be\label{eq:polydisk}
\theta\left(\overline{\bD}^d \right) = 2 \sqrt{d},
\ee
which improves the estimate $2d$ of \cite[Corollary 7.10]{DDSS}.
Now, $\theta(\ol{\bD}^d) = 2 \sqrt{d}$ implies that every $d$-tuple $T$ of (not necessarily normal) contractions dilates to a normal tuple $N$ such that $\|N_i\| \leq 2 \sqrt{d}$ for all $i$.
Indeed, since any tuple of contractions $T = (T_1, \ldots, T_d)$ has $(\mathrm{Re}(T_1), \mathrm{Im}(T_1), \ldots, \mathrm{Re}(T_d), \mathrm{Im}(T_d)) \in \Wmax{}\left(\overline{\bD}^d \right)$, we may conclude from (\ref{eq:polydisk}) that $T$ admits a normal dilation consisting of operators with norm at most $2 \sqrt{d}$. In particular, it does not matter if $\cH$ is infinite-dimensional, as the explicit dilations presented capture this case (alternatively, see subsection \ref{subsec:finitetoinfinite}). However, the converse argument does \textit{not} hold; knowing that every $d$-tuple of contractions $T$ dilates to a normal tuple $N$ such that $\|N_i\| \leq C$ for all $i$, does not imply that $\theta\left(\ol{\bD}^d\right) \leq C$. Even for $d = 1$, $\Wmax{}(\ol{\bD})$ includes tuples $(X, Y)$ such that $X + iY$ is not a contraction. Similarly, if $T = T_1$ is a contraction, then the Halmos dilation
\be\label{eq:halmosgeneralcase}
\begin{pmatrix}T & \sqrt{I - TT^*} \\ \sqrt{I - T^*T} & -T^* \end{pmatrix}
\ee
is unitary, so it has norm $1$, which is strictly less than $\theta\left(\overline{\bD} \right)$.

We caution the reader that while the mathematical content and constants of \cite[Corollary 7.10]{DDSS} are correct, the presentation of the corollary suggests that this faulty converse may hold. (Indeed, while two facts are presented in the corollary, bridged by an implication, the proof provides the implication in reverse order.) Regardless, we give a refinement of that result.

\begin{corollary}\label{cor:complex_balls}
Fix $d \geq 1$. Then
\bes
\Wmax{}\left(\overline{\bD}^d \right) \subseteq 2 \sqrt{d} \cdot \Wmin{}\left(\overline{\bD}^d \right),
\ees
and this constant is optimal for all $d$.
\vspace{.2 cm}
Further, if $(T_1, \ldots, T_d) \in \cB(H)^d$ is a tuple of (not necessarily normal) contractions, then there exists a normal dilation $N = (N_1, \ldots, N_d)$ such that $||N_j|| \leq \min\{d, 2 \sqrt{d}\}$. This constant is not necessarily optimal for $d \geq 2$.
\end{corollary}
\begin{proof}
The first statement is just (\ref{eq:polydisk}), and the existence of a normal dilation with norms bounded by $2 \sqrt{d}$ was shown directly after that equation.

Alternatively, one may apply the Halmos dilation (\ref{eq:halmosgeneralcase}) to the members of $(T_1, \ldots, T_d)$ to obtain a dilation $(U_1, \ldots, U_d)$ of unitaries. Let $P_1, \ldots, P_d$ be mutually orthogonal projections in $M_d(\bC)$ such that each $P_j$ has $1/d$ in the top left entry; for example we may take
\bes
P_j = \textrm{ projection onto } \left(\cfrac{1}{\sqrt{d}}, \cfrac{\omega^j}{\sqrt{d}}, \ldots, \cfrac{\omega^{j(d-1)}}{\sqrt{d}}\right),
\ees
where $\omega$ is a primitive $d$th root of unity. Then $(U_1 \otimes d P_1, \ldots, U_d \otimes dP_d)$ is a normal dilation of $T$ (in particular, the normal operators $U_j \otimes d P_j$ annihilate each other) such that $||U_j \otimes d P_j || \leq d$ for each $j$.
\end{proof}

Since $d < 2\sqrt{d}$ for $d \in \{1, 2, 3\}$, it is evident that $\theta\left(\overline{\bD}^d \right) = 2 \sqrt{d}$ is not the same as the optimal constant for dilating tuples of contractions. On the other hand, because dilations of self-adjoint contractions are governed by $\theta([-1,1]^d) = \sqrt{d}$ (where we note that if a dilation is not self-adjoint, we can remove the imaginary component), we know we can achieve no better scale than $\sqrt{d}$. At the moment we do not see a clear gap between the self-adjoint and general cases. Indeed, since the Halmos dilation (\ref{eq:halmosgeneralcase}) gives normality of the individual operators $T_i$ for free, it is conceivable that no gap exists.

\begin{problem}
Let $T = (T_1, \ldots, T_d)$ be a tuple of contractions. Must there exist a normal dilation $N = (N_1, \ldots, N_d)$ of $T$ such that $||N_j|| \leq \sqrt{d}$ for all $j$?
\end{problem}

Finally, we note that the real and complex $\ell^p$-balls satisfy the hypotheses of Corollaries \ref{cor:bestshiftatzero?} and \ref{cor:bestaffineatzero?}, and the estimate for $\theta(\ol{\bB}_{p,d})$ or $\theta(\ol{\bB}_{p,d}(\bC))$ given by computing the Banach-Mazur distance to an $\ell^2$-ball is actually attained. Therefore we can conclude the following.

\begin{corollary}
Let $K, L \subset \bR^d$ or $K,L \subset \bR^{2d}$ be orthogonal images of the same real or complex $\ell^p$ unit ball.
If there is some $C > 0$ and a translation $b$ such that
\bes
\Wmax{}(K) \subseteq C \cdot \Wmin{}(b + L)
\ees
holds, then $C \geq \theta(K)$. That is, $\theta(K) \leq \theta(K, b + L)$, and in particular, $\mathring{\theta}(K) = \theta(K)$.
\end{corollary}

\begin{remark}
Since all the explicit dilations in this section dilate operators $X_1, \ldots, X_d \in \cB(H)$ to block-operators whose entries are in the real unital $C^*$-algebra generated by $X_1, \ldots, X_d$, it follows that such dilations can always be achieved continuously. For example, we may consider $C^*$-algebras such as $C(X, M_k(\bC))$ for any compact Hausdorff space $X$.
\end{remark}

%%%%%%%%%%%%%%%%%%%%%%%%%

\section{Connections with cones and operator systems}\label{sec:cones}

In \cite{DDSS}, the focus was on closed and bounded matrix convex sets --- these correspond to matrix ranges of $d$-tuples of bounded operators.
In \cite{FNT}, the focus was on closed and salient matrix convex cones with an interior point, which were referred to by the authors {\em abstract operator systems}  --- as these correspond to finite dimensional operator systems.
In this section, we describe how to go back and forth between both points of view, and see how the minimal and maximal matrix convex sets over a set behave with respect to this.

%%%%%%%%%%%%%%%%%%%%%%%%%%%%%%%%
\subsection{Operator systems and matrix convex cones}
Recall that a {\em $(d+1)$-dimensional abstract operator system} is a free set $\cC$ such that
\begin{enumerate}
\item $\cC = \cup \cC(n) \subseteq \bM^{d+1}_{sa}$ is a closed matrix convex set,
\item each $\cC(n)$ is a closed cone,
\item each $\cC(n)$ is salient: $\cC(n) \cap (-\cC(n)) = \{0\}$, and
\item there exists $u \in \bR^d$ such that $u \otimes I_n$ is an interior point for all $n$.
\end{enumerate}
The assumption $\cC(1) \cap (-\cC(1)) = \{0\}$ implies that after an affine change of coordinates, we may assume that $\cC(1) \setminus \{0\} \subseteq \{x \in \bR^{d+1} : x_{d+1} > 0\}$.
By applying another affine change of coordinates that fixes $e_1 = (1,0,\ldots, 0), \ldots, e_d = (0,\ldots, 0, 1,0)$, we may assume in addition that the unit is given by $u = e_{d+1} = (0,\ldots,0,1)$. 
Thus we shall make this a standing assumption. 

\vskip 5pt
\noindent {\bf Standing assumption.} 
All $(d+1)$-dimensional abstract operator systems considered below will be assumed to satisfy 
\be\label{SA1}
\cC(1) \setminus \{0\} \subseteq \{x \in \bR^{d+1} : x_{d+1} > 0\} , 
\ee
and 
\be\label{SA2}
\textrm{ the order unit of } \cC \textrm{ is given by } \, u = e_{d+1} = (0,\ldots,0,1) . 
\ee
\vskip 5pt

A {\em realization} of $\cC$ consists of a concrete operator system in $\cB(H)$ with an ordered basis $T_1, \ldots, T_d, I$, such that a $(d+1)$-tuple $X = (X_1, \ldots, X_{d+1}) $ is in $\cC(n)$ if and only if
\[
{}^h L_T(X) := T_1 \otimes X_1 + \ldots + T_d \otimes X_d + I \otimes X_{d+1} \geq 0,
\]
which means that $\cC = \cD_{{}^hL_T} : = \{X \in \bM^{d+1}_{sa} \ : {}^hL_T(X) \geq 0\}$.

It is easy to see why a matrix convex set $\cC$ as above is called an operator system: for every finite dimensional operator system  one can find a basis $T_1, \ldots, T_d, I$ consisting of self-adjoint elements.
Then letting $S = S_T$ be the operator system generated by $T_1, \ldots, T_d, I$, we see that elements in $M_n(S)$ correspond to sums of the form ${}^h L_T(X) := T_1 \otimes X_1 + \ldots + T_d \otimes X_d + I \otimes X_{d+1}$, where $X \in \bM^{d+1}$.
An element ${}^h L_T(X) \in M_n(S)$ is then in $M_n(S)_+$ if and only if ${}^h L_T(X) \geq 0$.
Note that since $T_1, \ldots, T_d$ are all self-adjoint and linearly independent, ${}^h L_T(X)$ can be self-adjoint only if $X \in \bM^{d+1}_{sa}$, i.e., if $X_1, \ldots, X_{d+1}$ are all self-adjoints. (If we used a basis $T_1, \ldots, T_d, I$ that is not required to be self-adjoint, we would need to consider tuples $X \in \bM^{d+1}$).

%%%%%%%%%%%%%%%%%%%%%%%%%%%%%%%%
\begin{lemma}\label{lem:boundedbasis}
Let $S$ be an operator system, spanned by a basis of self-adjoints $\{I, T_1, \ldots, T_d\}$.
Then there exists a basis of self-adjoints $\{I, A_1, \ldots, A_d\}$ for $S$ such that $0 \in \operatorname{int}\cW(A)$ and such that $\cD_{L_{A}}$ is bounded.
\end{lemma}
\begin{proof}
By \cite[Lemma 3.4]{DDSS}, the conditions (i) $0 \in \operatorname{int}\cW(A)$, and (ii) $\cD_{L_{A}}$ is bounded, are both equivalent to (iii) $0 \in \operatorname{int}\cW_1(A)$.
We therefore consider $\cW_1(T)$.
If the convex set $\cW_1(T)$ had no interior point, then it would have to be contained in a hyperplane.
This means that there are constants $a_0, a_1, \ldots, a_d \in \bR$, not all zero, such that
\[
\phi\left(\sum_{i=1}^d a_i T_i\right) = \sum_{i=1}^d a_i \phi(T_i) = a_0 = \phi\left(a_0 I\right).
\]
for every state $\phi$.
As $\{I, T_1, \ldots, T_d\}$ is a linearly independent set, this is impossible.

Choose an interior point $c = (c_1, \ldots, c_d) \in \operatorname{int}(\cW_1(T))$.
Now putting $A_i = T_i - c_i I$, we obtain a basis $\{I, A_1, \ldots, A_d\}$ such that $0 \in \operatorname{int}\cW_1(A)$.
\end{proof}

%%%%%%%%%%%%%%%%%%%%%%%%%%%%%%%%
\subsection{The operations $\cC \mapsto \check{\cC}$ and $\cS \mapsto \hat{\cS}$}
Given an abstract operator system $\cC \subseteq \bM^{d+1}_{sa}$, we define $\check{\cC} \subseteq \bMsad$ by
\[
\check{\cC}(n) = \{A = (A_1, \ldots, A_d) \in (M_n)_{sa}^d :(A,I_n) \in \cC(n)\}.
\]
Note that this is the same as defining
\[
\check{\cC}(n) = \{A = (A_1, \ldots, A_d) \in (M_n)_{sa}^d : \exists B \in (M_n)_{sa}, \|B\|\leq 1 \textrm{ and } (A,B) \in \cC(n)\}.
\]
It is straightforward to check that $\check{\cC}$ is a closed matrix convex set with $0 \in \operatorname{int}\check{\cC}$ (this makes use of the standing assumption \eqref{SA2}).
It follows from the definition and from the comments above that, if $\{T_1, \ldots, T_d, I\}$ is a realization for $\cC$, then  $\check{\cC}$ is the free spectrahedron determined by the pencil $L_T$, defined by
\[
L_T(A) := T_1 \otimes A_1 + \ldots + T_d \otimes A_d + I \otimes I.
\]
Thus $\check{\cC} = (\cD_{{}^hL_T})^{\vee} = \cD_{L_T} := \{A \in \bMsad : L_T(A) \geq 0\}$.

In general, given an abstract operators system $\cC$, the matrix convex set $\check{\cC}$ constructed as above might be unbounded; however the standing assumption \eqref{SA1} implies that $\check{\cC}(1)$ is bounded, and hence $\check{\cC}$ is too (by, e.g., \cite[Lemma 3.4]{DDSS}).

Conversely, if $\cS$ is a matrix convex set with $0$ in the interior, then $\cS  = \cD_{L_T}$ for some $T$, i.e., $\cS$ is the free spectrahedron determined by some pencil $L_T(A) = \sum_{i=1}^d T_i \otimes A_i + I \otimes I$ (see \cite[Proposition 3.5]{DDSS}).
But then the homogeneous pencil
\[
 {}^h L_T(X) := \sum_{i=1}^d T_i \otimes X_i + I \otimes X_{d+1}
\]
determines a free spectrahedron which is a closed matrix convex salient cone in $\bM^{d+1}_{sa}$, that is, a $(d+1)$-dimensional operator system.
We write $\hat{\cS}$ for this free spectrahedron $\cD_{{}^h L_T}$.
Note that if $\cS$ is bounded, then $\hat{\cS}(1) \setminus \{0\}$ is contained in the open halfspace $\{x \in \bR^{d+1} : x_{d+1} > 0\}$.
Moreover, given any $(d+1)$-dimensional operator system $S$, Lemma \ref{lem:boundedbasis} says that we can find a basis $\{I, T_1, \ldots, T_d\}$ such that $\cD_{L_T}$ is bounded, so there is no loss of generality.

Therefore, starting from a $(d+1)$-dimensional operator system, and choosing a basis $\{I, T_1, \ldots, T_d\}$ such that $\cD_{L_T}$ is bounded, we find that the abstract operator system $\cC = \cD_{{}^h L_T} = (\cD_{L_T})^\wedge$ satisfies the assumptions \eqref{SA1} and \eqref{SA2}, that is, that $\cC(1) \setminus \{0\} \subseteq \{x \in \bR^{d+1} : x_{d+1} > 0\}$ and that $u = e_{d+1} = (0,\ldots,0,1)$ is the order unit, i.e., the point corresponding to $I$ in the realization. 

By a Theorem of Zalar (\cite[Theorem 2.5]{Zalar}), under the assumption that $\cD_{L_T}$ is bounded,
\be\label{eq:Zalarinclusion}
\cD_{L_{T_1}} \subseteq \cD_{L_{T_2}} \Longleftrightarrow \cD_{{}^hL_{T_1}} \subseteq \cD_{{}^hL_{T_2}}.
\ee
Since a matrix convex set has the form $\cD_{L_T}$ if and only if it contains $0$ in its interior \cite[Proposition 3.5]{DDSS}, we conclude that the maps $\cC \mapsto \check{\cC}$ and $\cS \mapsto \hat{\cS}$ are mutual inverses between the set of $(d+1)$-dimensional abstract operators systems satisfying \eqref{SA1} and \eqref{SA2}, and closed and bounded matrix convex sets in $d$ variables that contain the origin in their interior.

%%%%%%%%%%%%%%%%%%%%%%%%%%%%%%%%
\subsection{The operations $C \mapsto \check{C}$ and $K \mapsto \hat{K}$ at the scalar level. $\Wmin{}$ and $\Wmax{}$ for cones}
Given a closed salient cone $C \subseteq \{x \in \bR^{d+1} : x_{d+1} > 0\}$ with $(0,\ldots, 0,1) \in \Int C$, we define
\[
\check{C} = \{x \in \bR^d : (x,1) \in C\}.
\]
Conversely, given a closed convex set $K \subseteq \bR^d$, we let $\hat{K}$ denote the closed convex cone generated by $0\in\bR^{d+1}$ and $\{(x,1) : x \in K\} \subseteq \bR^{d+1}$.

Recall from Section \ref{sec:notation} the definitions of the polar dual of a closed convex or matrix convex set.
For cones, a slightly different description of the polar duals is available:
if $C$ is a closed convex cone, then one sees that
\[
C' = \{y  : \langle x,y \rangle \leq 0 \, \textrm{ for all } \, x \in C\};
\]
if $\cC$ is an operator system, then it is also easy to check that $\cC^\bullet$ is also given by
\be\label{eq:dual_cone}
\cC^\bullet = \{Y : \sum_i X_i \otimes Y_i \leq 0 \, \textrm{ for all } \, X \in \cC\}.
\ee
It is not hard to check that --- up to a sign change in the inequality --- this polar duality corresponds with the duality of operator spaces defined by Paulsen, Todorov and Tomforde in \cite[Definition 4.1]{PTT11}.

%%%%%%%%%%%%%%%%%%%%%%%%%
\begin{proposition}\label{prop:Wminmaxhat}
Let $K \subseteq \bR^d$ be a compact convex set with $0 \in \Int K$.
Then
\[
\Wmin{}(K)^{\wedge} = \Wmin{}(\hat{K}) \quad \hspace{.2 cm}\mathrm{ and } \hspace{.2 cm} \quad \Wmax{}(K)^{\wedge} = \Wmax{}(\hat{K}) .
\]
\end{proposition}
\begin{proof}
Recall that $K = K'' = \{x \in \bR^d: \sum_{i=1}^d x_i y_i \leq 1 \textrm{ for all } y \in K'\}$, and therefore
\[
\Wmax{}(K) = \left\{X = (X_1, \ldots, X_d) : \sum_{i=1}^d y_i X_i \leq I \, \textrm{ for all }\, y \in K'\right\}.
\]
Then
\[
\Wmax{}(K)^{\wedge} = \left\{X = (X_1, \ldots, X_{d+1}) : X_{d+1}-\sum_{i=1}^{d}  y_i X_i \geq 0 \, \textrm{ for all }\, y \in K'\right\}.
\]
On the other hand,
\[
\hat{K} = \left\{(x_1, \ldots, x_{d+1}) : x_{d+1} - \sum_{i=1}^d x_i y_i \geq 0 \, \textrm{ for all }\, y \in K'  \geq 0\right\},
\]
so
\[
\Wmax{}(\hat{K}) = \left\{(X_1, \ldots, X_{d+1}) : X_{d+1} - \sum_{i=1}^d  y_i X_i \geq 0 \, \textrm{ for all }\, y \in K' \right\}.
\]
That proves $\Wmax{}(K)^{\wedge} = \Wmax{}(\hat{K})$.

For the other equality, first let us note that $\Wmin{}(K)^{\wedge} \supseteq \Wmin{}(\hat{K})$ is immediate, since $\hat{K} \subseteq \Wmin{}(K)^{\wedge}$.
Now, suppose that $\Wmin{}(K) = \cD_{L_T}$.
Then $\Wmin{}(K)^{\wedge} = \cD_{{}^h L_T}$, and note that $\cD_{{}^h L_T}(1) = \hat{K}$.
The set $\Wmin{}(\hat{K})$ is also a matrix convex cone, which is salient and contains an order unit $(0,\ldots,0,1)$, so it is therefore realized as the operator system generated by a $(d+1)$-tuple $(S_1, \ldots, S_d,I)$. Equivalently, it is the free spectrahedron determined by a homogeneous linear pencil: $\Wmin{}(\hat{K}) = \cD_{{}^h L_S}$.
We claim that
\[
\cD_{L_T} \subseteq \cD_{L_S}.
\]
Assuming the claim, we invoke \cite[Theorem 2.5]{Zalar} to find that
\[
\Wmin{}(K)^{\wedge} = \cD_{{}^h L_T} \subseteq \cD_{{}^h L_S} = \Wmin{}(\hat{K}).
\]
It remains to prove the claim.
Now, $\cD_{L_S}$ is a matrix convex set, and to prove that it contains $\cD_{L_T} = \Wmin{}(K)$, it suffices to show that $K = \cD_{L_T}(1) \subseteq \cD_{L_S}(1)$.
But if $x \in K$, then $(x,1) \in \hat{K} = \cD_{{}^h L_S}(1)$, so ${}^h L_S(x,1) = L_S(x) \geq 0$.
Whence $x \in \cD_{L_S}(1)$, and the proof is complete.
\end{proof}

Below we will require the following closely related lemma, the proof of which is omitted.
%%%%%%%%%%%%%%%%%%%%%%%%%
\begin{lemma}\label{lem:hatKtag}
Let $K$ be a compact convex subset of $\bR^d$, and assume that $0 \in \operatorname{int}K$.
Then $(-\hat{K})' = (-K')^{\wedge}$.
\end{lemma}

%%%%%%%%%%%%%%%%%%%%%%%%%%%%%%%%
\subsection{Inclusions of cones and finite-dimensional representations}

Following \cite{FNT}, for a closed salient cone $C \subseteq \{x \in \bR^{d+1} : x_{d+1} > 0\}$ with $(0,\ldots, 0,1) \in \Int C$, and a positive number $\nu > 0$, let us define $\nu \uparrow C := (\nu \check{C})^\wedge$.
In \cite[Section 5]{FNT}, Fritz, Netzer and Thom found constants $\nu$ such that $\Wmax{}(\nu \uparrow C) \subseteq \Wmin{}(C)$.
For example, in \cite[Theorem 5.8]{FNT} they show that if $\check{C}$ is symmetric with respect to $0$, then $\Wmax{}(\frac{1}{d} \uparrow C) \subseteq \Wmin{}(C)$.
They note that for applications one is interested in finding the largest $\nu$ for which this inclusion occurs.
Our results apply to this setting: for example, we see that
\be
\check{C} = [-1,1]^d \implies \Wmax{}\left(\frac{1}{\sqrt{d}} \uparrow C \right) \subseteq \Wmin{}(C)
\ee
holds by Theorem \ref{thm:cubedilationgeneral} and Proposition \ref{prop:Wminmaxhat}, and this improves upon the previously known constant $\nu = \frac{1}{d}$.

The problem of containment of matrix convex cones is relevant to interpolation problems of UCP maps as well.
By \cite[Theorem 2.5]{Zalar}, under assumptions \eqref{SA1} and \eqref{SA2}, whenever $\cC = \cD_{{}^hL_T}$ and $\cC' = \cD_{{}^h L_{T'}}$, then $\cC\subseteq \cC'$ if and only if there exists a UCP map $S(T) \to S(T')$ that maps $T_i$ to $T'_i$.
This also follows from \eqref{eq:Zalarinclusion} together with results (such as Corollary 5.10 and Remark 5.11 in \cite{DDSS}) saying that $\cD_{L_T} \subseteq \cD_{L_{T'}}$ if and only if there exists a UCP map $S(T) \to S(T')$ that maps $T_i$ to $T'_i$.

Any closed and bounded matrix convex set $\cS \subseteq \bMsad$ has the form $\cS = \cW(A)$ for some tuple $A \in \cB(H)_{sa}^d$ (\cite[Proposition 3.5]{DDSS}).
It is natural to ask, under what geometric conditions on $\cS$ can this tuple $A$ be chosen to act on a finite dimensional Hilbert space $H$? A dual version of this problem has been treated for cones in  \cite[Theorems 3.2 and 4.7]{FNT}. Their results show that the maximal operator system over a cone $C$ is \textit{finite-dimensionable realizable}, i.e. can be realized as $\cD_{{}^h L_A}$ for $A$ acting on a finite-dimensional space, if and only if $C$ is polyhedral. Further, under the assumption that $C$ is a salient \textit{polyhedral} cone, it holds that the minimal operator system over $C$ is finite-dimensional realizable if and only if $C$ is simplicial, which holds if and only if the minimal and maximal operator systems over $C$ are equal. (Recall that a cone $C\subseteq \bR^{d+1}$ is said to be a {\em simplicial cone} if it is affinely isomorphic to a positive  orthant $\{x \in \bR^{d+1}: x_1 \geq 0 , \ldots, x_{d+1} \geq 0\}$.
A cone $C$ such that $C \setminus \{0\} \subseteq \{x \in \bR^{d+1} : x_{d+1} > 0\}$ is simplicial if and only $\check{C}$ is a simplex.)

In this subsection we treat the corresponding problems for $\Wmin{}(K)$ and $\Wmax{}(K)$, giving alternative proofs and improvements of \cite{FNT} whenever possible. The following theorem may be proved by importing the results of \cite{FNT} (where we note that ``min" and ``max" are interchanged when translating the results to our setting), but we give an alternative proof for the statement regarding $\Wmin{}(K)$.

%%%%%%%%%%%%%%%%%%%%%%%%%%%%%%%%

\begin{theorem}\label{thm:WminmaxWA}
Let $K$ be a compact convex subset of $\bR^d$.
Then $\Wmin{}(K) = \cW(A)$ for $A \in \bMsad$ if and only if $K$ is a polytope.
If $K$ is a polytope, then
$\Wmax{}(K) = \cW(A)$ for $A \in \bMsad$ if and only if $K$ is a simplex.
\end{theorem}
\begin{proof}
By \cite[Corollary 2.8]{DDSS}, if $A$ is a normal tuple, then $\cW(A)$ is the smallest matrix convex set containing $\sigma(A)$, so $\cW(A) = \Wmin{}(\conv(\sigma(A))$.
If $K$ is a polytope, then letting $A$ be the normal tuple that has joint eigenvalues at the vertices of $K$, we have that $\Wmin{}(A) = \cW(A)$.
Moreover, if $K$ is a simplex, this still holds, and by Theorem \ref{thm:simplex_unique}, $\Wmax{}(K) = \Wmin{}(K) = \cW(A)$, and we are done with one direction.

Now for the converse direction. If $\Wmin{}(K) = \cW(A)$ for $A \in \bMsad$, then $A \in \Wmin{}(K)$, so by Proposition \ref{prop:dilation_finite_infinite} $A$ has a normal dilation $N \in \bMsad$ with $\sigma(N) \subseteq K$.
But then $\cW(A) \subseteq \cW(N)$, so $K \subseteq \cW_1(A) \subseteq \cW_1(N) = \conv \sigma(N)$.
We conclude that $\conv\sigma(N) = K$.
Since the spectrum of $N$ has finitely many points, $K$ is polytope.

Suppose now that $K$ is polytope.
Assume without loss of generality that $0 \in \operatorname{int} K$, and that $\cW(A) = \Wmax{}(K)$.
Then, using Proposition \ref{prop:Wminmaxhat}, $\cW(A)^\wedge = \Wmax{}(K)^\wedge = \Wmax{}(\hat{K})$ is an abstract operator system.
Taking polar duals and applying Lemma \ref{lem:hatKtag}, we find that $\Wmax{}(\hat{K})^\bullet = \Wmin{}((\hat{K})') = -\Wmin{}({(-K')^\wedge})$.
Since $\cW(A) = \Wmax{}(K)$, we get $\cW(A)^\bullet = \Wmin{}(K')$.
On the other hand, $\cW(-A)^\bullet = \cD_{L_A}$ (this follows from \cite[Proposition 3.1]{DDSS}; note the change in sign convention).
Thus, $-K' = \cD_{L_A}(1)$,
and using Proposition \ref{prop:Wminmaxhat} again,
\[
\Wmin{}({(-K')^\wedge}) = \Wmin{}(-K')^\wedge = (\cD_{L_A})^\wedge = \cD_{{}^h L_A} .
\]
So $A$ is a finite dimensional realization for $\Wmin{}((-K')^\wedge)$.
By \cite[Theorem 4.7]{FNT}, $(-K')^\wedge$ must be simplicial, therefore $K=K''$, together with $K'$, must be simplices.

\end{proof}

\begin{remark}
Dual to the comments in \cite{FNT}, note that even if $K$ is not a polytope, it is still possible that $\Wmax{}(K)$ could equal the matrix range of a tuple of matrices.
For this, let $F = (F^{[2]}_1, F^{[2]}_2) \in (M_2)^2_{sa}$ be as in Section \ref{sec:ball}.
By \cite[Corollary 14.15]{HKMS15}, $\cD_{L_F} = \Wmin{}(\ol{\bB}_{2,2})$.
It follows that $\cW(F) = \Wmax{}(\ol{\bB}_{2,2})$.
\end{remark}

While Theorem \ref{thm:WminmaxWA} is a direct translation of results from \cite{FNT}, the following theorem is a slight improvement of its corresponding result, having removed the assumption that $C$ is polyhedral.

\begin{corollary}\label{cor:Wminmaxcones}
Let $C \subseteq \bR^{d+1}$ be a closed salient convex cone with nonempty interior.
Then $\Wmin{}(C) = \Wmax{}(C)$ if and only if $C$ is a simplicial cone.
\end{corollary}
\begin{proof}
As we discussed in the beginning of the section, we may assume that $C \setminus \{0\} \subseteq \{x \in \bR^{d+1} : x_{d+1} > 0\}$ and $(0,\ldots, 0,1) \in \Int C$.
Suppose that $\Wmin{}(C) = \Wmax{}(C)$.
By Proposition \ref{prop:Wminmaxhat},
\[
\Wmin{}(\check{C})^\wedge = \Wmin{}(C) = \Wmax{}(C) = \Wmax{}(\check{C})^\wedge.
\]
Because the maps $\cC \mapsto \check{\cC}$ and $\cS \mapsto \hat{\cS}$ are mutual inverses, we find that $\Wmin{}(\check{C}) = \Wmax{}(\check{C})$.
By Theorem \ref{thm:simplex_unique}, $\check{C}$ is a simplex, so $C$ is a simplicial cone.

The converse follows from Theorem \ref{thm:simplex_unique} and Proposition \ref{prop:Wminmaxhat} in a similar way.
\end{proof}

%%%%%%%%%%%%%%%%%%%%%%%%%%%%%%%%
\begin{problem}\label{prob:matrixrangeequalsmore}
Characterize the tuples $A$ for which there exists some $K$ such that $\mathcal{W}(A) = \Wmin{}(K)$.
Likewise, characterize the tuples $A$ for which there exists some $K$ such that $\mathcal{W}(A) = \Wmax{}(K)$.
\end{problem}

%%%%%%%%%%%%%%%%%%%%%%%%%%%%%%%%%%%%%%%%%%%%%%%%
\section{Dilation constants versus minimal dilation hulls}\label{sec:minhulls}

Given compact convex sets $K,L \subseteq \bR^d$, we have considered the dilation constants $\theta(K,L)$, $\theta(K)$ and $\mathring{\theta}(K)$ (see equations \eqref{eq:thetadef1} \eqref{eq:thetadef2} and \eqref{eq:mathringthetadef}). We know from Theorems \ref{thm:BplusWminmax} and \ref{thm:balldilationgeneral} that $\theta(\ol{\bB}_{p,d}^+) = d^{1-1/p}$ and $\theta(\ol{\bB}_{p,d}) = d^{1-|1/2-1/p|}$. However, the explicit dilations which were used to compute $\theta(\cdot)$ often had joint spectrum in smaller sets, which were merely \textit{contained} in a multiple of $\ol{\bB}_{p,d}$ or $\ol{\bB}_{p,d}^+$. We now consider a sense of minimality among dilation hulls which can see this distinction.

Define
\be\label{eq:the_set_of_dilation_hulls}
\cE(K) = \{L \supseteq K \textrm{ compact and convex } : \Wmax{}(K) \subseteq \Wmin{}(L)\}.
\ee
From the definitions \eqref{eq:Wmax_def2} and \eqref{eq:Wmin_def1} of $\Wmax{}$ and $\Wmin{}$, we know that a compact convex set $L \subseteq \bR^d$ is in $\cE(K)$ if and only if for every $X \in (M_n)^d_{sa}$ with $\cW_1(X) \subseteq K$, there exists a tuple $N \in \cB(H)^d_{sa}$ of commuting self-adjoint operators such that $\sigma(N) \subseteq L$ and $N$ is a dilation of $X$. (Recall that $N$ is a dilation of $X$ if and only if
\[
X_i = V^*N_iV \,\, , \,\, i=1, \ldots, d,
\]
holds for some isometry $V : \bC^n \to H$.)
If $L \in \cE(K)$, then we say that $L$ is a {\em dilation hull} of $K$.
It is worth pointing out that $\cE(K)$ is not closed under intersections (e.g., by using Theorem \ref{thm:cubedilationgeneral}).

%%%%%%%%%%%%%%%%%%%%%%%%%%%%%%%%
\begin{proposition}\label{prop:minhullsdoexist}
Let $K \subseteq \bR^d$ be a compact convex set. 
If $\cE(K)$ is ordered by inclusion, then $\cE(K)$ has a minimal element.
\end{proposition}
\begin{proof}
Suppose that $\{L_\alpha\}$ is a chain (decreasing by inclusion) in $\cE(K)$.
We must show that $L:= \cap_\alpha L_\alpha \in \cE(K)$.
Let $X \in \Wmax{n}(K)$.
For every $\alpha$, there is an isometry $V^{(\alpha)} : \bC^n \to H^{(\alpha)}$ and a normal tuple $N^{(\alpha)} \in \cB(H^{(\alpha)})_{sa}^d$ such that $X = V^{(\alpha)*} N^{(\alpha)} V^{(\alpha)}$ and $\sigma(N^{(\alpha)}) \subseteq L_\alpha$.
By Proposition \ref{prop:dilation_finite_infinite}, we may assume that $H^{(\alpha)} = \bC^N$, for $N = {2n^3(d+1)+1}$.
Therefore, one can choose convergent subnets $V^{(\beta)} \to V$, $N^{(\beta)}\to N$, so that $X = V^* N V$ and $\sigma(N) \subseteq \cap_\beta L_\beta = \cap_\alpha L_\alpha = L$.
\end{proof}

Let us call a minimal element of $\cE(K)$ a {\em minimal dilation hull} and define
\[
\operatorname{md}(K) :=\{L \in \cE(K) : L \textrm{ is minimal}\} .
\]
First, there is no reason for elements in $\operatorname{md}(K)$ to resemble $K$ in any way; the $\ell^2$-ball provides a counterexample.

\begin{proposition}
There is no $\ell^2$-ball in $\operatorname{md}(\ol{\bB}_{2,d})$.
\end{proposition}
\begin{proof}
Suppose a ball $L = x + C \cdot \ol{\bB}_{2,d}$ is in $\cE(\ol{\bB}_{2,d})$. Then from Theorem \ref{thm:everythingball} it follows that there is a simplex $\Pi$ with $\ol{\bB}_{2,d} \subset \Pi \subset L$, and certainly $\Pi \in \cE(\Pi) \subseteq \cE(\ol{\bB}_{2,d})$.
Therefore $L$ is not minimal: $L \not\in \operatorname{md}(\ol{\bB}_{2,d})$.
\end{proof}

However, note that while Theorem \ref{thm:everythingball} shows that if an $\ell^2$-ball belongs to $\cE(\ol{\bB}_{2,d})$, then there is a smaller simplex which is also a dilation hull, the theorem does not apply for other shapes. For example, the $d$-diamond $d \cdot \ol{\bB}_{1,d}$ is a minimal dilation hull of $\ol{\bB}_{2,d}$.

\begin{proposition}\label{prop:mindil_B1B2}
It holds that $d \cdot \ol{\bB}_{1,d} \in \operatorname{md}(\ol{\bB}_{2,d})$.
\end{proposition}
\begin{proof}
The containment (\ref{eq:smallpbettersets}) shows that $d \cdot \ol{\bB}_{1,d} \in \cE(\ol{\bB}_{2,d})$. Suppose that there is a proper compact convex subset $L \subset d \cdot \ol{\bB}_{1,d}$ such that $\Wmax{}(\ol{\bB}_{2,d}) \subseteq \Wmin{}(L)$. Then one of the extreme points of the diamond is missing, so by rotational symmetry of the ball we may suppose that $(-d, 0, \ldots, 0) \not\in L$. Therefore
\bes
L \subseteq M = \{ \vec{x} \in d \cdot \ol{\bB}_{1,d}: -d + \varepsilon < x_1 \leq d\}
\ees
for some $\varepsilon > 0$, which we may suppose has $\varepsilon < 1$. Now, $M$ is the convex hull of points of the form $(-d+\varepsilon, b)$ with $||b||_{\ell^1} \leq \varepsilon$, as well as the extreme points of $d \cdot \ol{\bB}_{1,d}$, except $(-d, 0, \ldots, 0)$. The expression $\sqrt{(x_1 - \varepsilon)^2 + \sum_{k=2}^d x_k^2}$ is bounded by $\sqrt{d^2 + \varepsilon^2}$ on these points, so as any $\ell^2$-ball is convex, it follows that $L \subseteq M$ is contained in the ball of radius $C = \sqrt{d^2 + \varepsilon^2}$ centered at $y = (\varepsilon, 0, \ldots, 0)$. From Theorem \ref{thm:everythingball}, we conclude
\bes
\sqrt{d^2 + \varepsilon^2} \geq \sqrt{\varepsilon^2 + (d-1)^2} + 1,
\ees
but this is a contradiction.
\end{proof}

Certainly this also implies that there is no simplex $\Pi$ with $\ol{\bB}_{2,d} \subset \Pi \subset d \cdot \ol{\bB}_{1,d}$. On the other hand, while $[-\sqrt{d}, \sqrt{d}]^d \in \cE(\ol{\bB}_{2,d})$, $[-\sqrt{d}, \sqrt{d}]^d$ is not generally a minimal dilation hull.

\begin{proposition}
The claim $[-\sqrt{d}, \sqrt{d}]^d \in \operatorname{md}(\ol{\bB}_{2,d})$ is true for $d = 2$ and false for $d = 4$.
\end{proposition}
\begin{proof}
If $d = 2$, then $[-\sqrt{2}, \sqrt{2}]^2$ is an orthogonal image of the diamond $2 \cdot \ol{\bB}_{1,2}$, which is a minimal dilation hull of the ball $\ol{\bB}_{2,2}$. Since the ball is orthogonally invariant, $[-\sqrt{2}, \sqrt{2}]^2$ is also minimal.

On other other hand, if $d = 4$, then while $[-\sqrt{4}, \sqrt{4}]^4 = [-2,2]^4$ is a dilation hull of $\ol{\bB}_{2,4}$, so is every orthogonal image of the diamond $4 \cdot \ol{\bB}_{1,4}$. Such an image may be found properly inside $[-2,2]^4$; consider the convex hull of the mutually orthogonal vectors
\bes
(2,2,2,2), \hspace{.5 cm} (2,-2,2,-2), \hspace{.5 cm} (2,2,-2,-2),  \hspace{.5 cm} (2,-2,-2,2),
\ees
and their opposites, all of which have norm $4$.
\end{proof}

We note that the shape of minimal dilation hulls of $K$ is not necessarily unique, once again seen when $K$ is an $\ell^2$-ball. Namely, Proposition \ref{prop:mindil_B1B2} and the following proposition produce diamonds and simplices as minimal dilation hulls, respectively.

\begin{proposition}\label{prop:minsimplexball}
Let $\Delta \subset \bR^d$ be any $d$-simplex such that $\ol{\bB}_{2,d} \subset \Delta$ and each vertex of $\Delta$ lies on $d \cdot \partial \ol{\bB}_{2,d}$. Then $\Delta \in \operatorname{md}(\ol{\bB}_{2,d})$.
\end{proposition}
\begin{proof}
Since $\Wmax{}(\Delta) = \Wmin{}(\Delta)$, it is clear that $\Delta$ is a dilation hull of $\ol{\bB}_{2,d}$. Suppose a proper compact convex subset $L \subset \Delta$ is a dilation hull, so one of the vertices of $\Delta$ is missing from $L$. It follows that $L$ may be contained in a ball of radius $d$ centered at some $x \not= 0$, so $\Wmax{}(\ol{\bB}_{2,d}) \subseteq \Wmin{}(x + d \cdot \ol{\bB}_{2,d})$. This contradicts Theorem \ref{thm:everythingball}.
\end{proof}

An example of such a simplex $\Delta = \Delta(d)$ can be defined by the base case $\Delta(1) = [-1,1]$ and the recursive definition of $\Delta(d)$ as the convex hull of $\left( \sqrt{\frac{d+1}{d-1}} \cdot \Delta(d-1) \right) \times \{-1\}$ and $(0,\ldots, 0, d)$. By repeated application of Lemma \ref{lem:optimalsimplexplacement}, it is seen that $\ol{\bB}_{2,d} \subseteq \Delta$. Further, all of the vertices of $\Delta(d)$ have norm $d$ by the computation $\left( \sqrt{\frac{d+1}{d-1}} \cdot (d-1) \right)^2 + (-1)^2 = d^2$.

\begin{problem}
For what types of collections $\mathcal{K}, \mathcal{L}$ of compact convex sets does it hold that if $K \in \mathcal{K}, L \in \mathcal{L}$, and $\Wmax{}(K) \subseteq \Wmin{}(L)$, then there is a simplex $\Pi$ with $K \subseteq \Pi \subseteq L$?
\end{problem}
If $\mathcal{K} = \mathcal{L}$ is the collection of all $\ell^2$-balls in $\bR^d$, then interpolating simplices can be found, but they cannot necessarily be found when $\mathcal{K} = \mathcal{L}$ is the collection of cubes, or diamonds, and so on. Can this be generalized?

We now consider another special case in which a circumscribed simplex over $K$ is a minimal dilation hull of $K$. This also gives a partial answer to Problem \ref{prob:whatdimensionmin=max}.

\begin{definition}
We say that a convex body $K \subset \bR^d$ is \textit{simplex-pointed at} $x$ if $x \in K$ and there is an open set $O \subset \bR^d$ such that $x \in O$ and $ \overline{O \cap K}$ is a $d$-simplex with $x$ as a vertex. Equivalently, there is an invertible affine transformation $A$ on $\bR^d$ such that
\bes
A(x) = (0, \ldots, 0), \hspace{2 cm}A(K) \subset [0, \infty)^d,
\ees
and
\bes
\exists \varepsilon > 0 \text{ with } \{(x_1, \ldots, x_d) \in \bR^d: 0 \leq x_j \leq \varepsilon \textrm{ for all } j \} \subseteq K.
\ees

\begin{figure}[htbp]
\centering
\begin{minipage}{.5\textwidth}
  \centering
  \includegraphics[width=.76\linewidth]{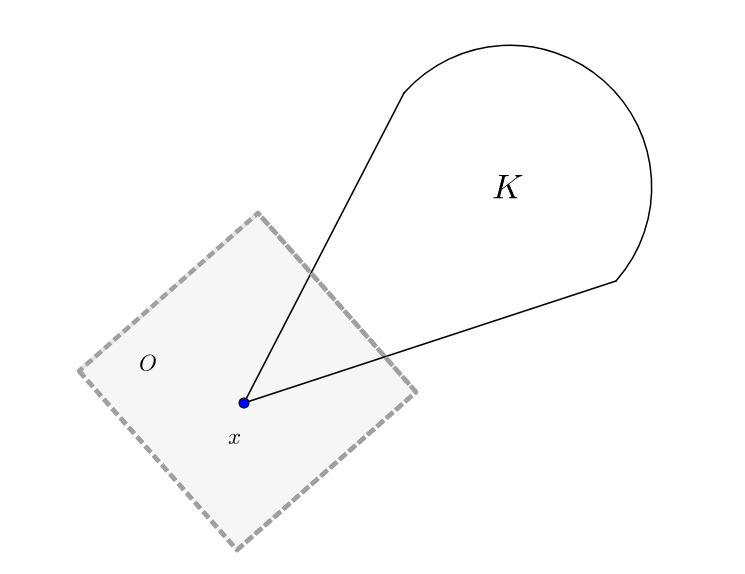}
 \captionof{figure}{\small $K$ is simplex-pointed at $x$, but not polyhedral.}
  \label{fig:simplexpointed}
\end{minipage}%
\begin{minipage}{.5\textwidth}
  \centering
  \includegraphics[width=.76\linewidth]{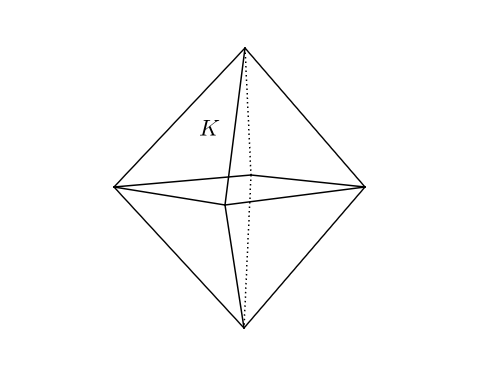}
  \captionof{figure}{\small $K$ is polyhedral, but not simplex-pointed at any $x$.}
  \label{fig:notsimplexpointed}
\end{minipage}
\end{figure}

\end{definition}

\begin{theorem}\label{thm:pointednonsense}
Let $K \subset \bR^d$ be a convex body which is simplex-pointed at $x \in K$. Suppose $\Delta$ is a $d$-simplex such that $K \subseteq \Delta$, $x$ is a vertex of $\Delta$, and the edges of $\Delta$ which emanate from $x$ point in the same direction as the edges of the simplex $\overline{O \cap K}$ based at $x$. Let $F \subset \Delta$ be the face of $\Delta$ which does not include $x$, and suppose that the interior of the face $F$ includes a point $y \in K$. Then $\Delta \in md(K)$.

\vspace{.25 cm}

\noindent Moreover, if $K$ meets these conditions and $\Wmax{2}(K) = \Wmin{2}(K)$, then $K = \Delta$, and in particular, $K$ is a simplex.

\end{theorem}

\begin{proof}
After an affine transformation, we may suppose  $K \subset [0, \infty)^d$, $x = (0, \ldots, 0) \in K$, $K \subseteq \Delta = \ol{\bB}_{1,d}^+$, and there exists $\alpha > 0$ such that
\be\label{eq:positiveinteriorpoint}
0 \leq a_j \leq\alpha \hspace{.2 cm} \forall j \in \{1, \ldots, d\} \hspace{.05 cm} \implies \hspace{.05 cm} (a_1, \ldots, a_d) \in K.
\ee
 Moreover, there is a point $y = (y_1, \ldots, y_d) \in K$ such that $y_j > 0$ for all $j$, and $||y||_{\ell^1} = 1$.

Suppose $L \subseteq \ol{\bB}_{1,d}^+$ is a compact convex set such that $\Wmax{}(K) \subseteq \Wmin{}(L)$. Since $(0, \ldots, 0) \in K$, it certainly follows that $(0, \ldots, 0) \in L$. We will show that each standard basis vector $e_j$ is also in $L$, which implies $L = \ol{\bB}_{1,d}^+$.
Fix $1 \leq j \leq d$ and note that since $||y||_{\ell^1} = 1$ and each $y_j$ is strictly greater than zero, we can fix $r \in (0, 1)$ such that for each $j$, $\sum\limits_{k \not= j} y_k < r < 1$. Approximate $y$ with an interior point $c$ of $K$ such that
\be\label{eq:technicalbs}
2 \cdot \cfrac{1 - ||c||_{\ell^1}}{1 - r} < 1 \hspace{1 cm} \textrm{ and } \hspace{1 cm}  \sum_{k \not= j} c_k < r.
\ee
Also, fix any $\varepsilon > 0$ small enough that a strengthening of (\ref{eq:technicalbs}) holds:
\be\label{eq:gotchajournalism}
\delta := 2 \cdot \cfrac{1 - ||c||_{\ell^1}\cdot (1 - \varepsilon)}{1 - r} < 1.
\ee
Next, choose projections $P_1, \ldots, P_d \in M_2(\bC)$ (dependent on $\varepsilon$) such that $||P_j - P_k|| < \varepsilon$ for each $j$ and $k$, but $P_1, \ldots, P_d$ project onto lines which pairwise intersect trivially. The tuple $(c_1P_1, c_2 P_2 \ldots, c_d P_d)$ has joint numerical range which is contained in $[0, \infty)^d$ and is within $\varepsilon$ of the joint numerical range of $(c_1 P_1, c_2 P_1, \ldots, c_d P_1)$, i.e. the line segment between $c$ and $0$. Since $c$ is an honest interior point of $K$ and (\ref{eq:positiveinteriorpoint}) shows that there is a \textit{nonnegative} neighborhood of $0$ within $K$, it follows that for sufficiently small $\varepsilon > 0$, the joint numerical range of $(c_1P_1, \ldots, c_d P_d)$ is within $K$. That is, $(c_1P_1, \ldots, c_d P_d) \in \Wmax{}(K)$.
Since we have assumed $\Wmax{}(K) \subseteq \Wmin{}(L)$, there is a normal dilation $N$ of $(c_1P_1, \ldots, c_d P_d)$ with joint spectrum in $L$.
By design, $(c_1P_1, \ldots, c_d P_d)$ satisfies the conditions of  Lemma \ref{lem:poscommgivesann}, and we may find a different normal dilation $Z$ such that $Z_i Z_j = 0$ for $i \not= j$ and $\sigma(Z) \subseteq \sigma(N) \cup \{0\} \subseteq L$.
Because the operators $Z_k$ annihilate each other and have norm bounded by $1$ (as $L \subseteq \ol{\bB}_{1,d}^+)$, it follows that for each $j$,
\bes
\left|\left|Z_j + (1 - \delta) \sum\limits_{k\not=j} Z_k \right|\right| \leq \max\{||Z_j||, (1 - \delta)||Z_k||\} \leq \max\{||Z_j||, 1 - \delta\}.
\ees
On the other hand,
\bes\begin{aligned}
\left|\left| c_j P_j + (1 - \delta) \sum\limits_{k \not = j} c_k P_k \right|\right| & \geq \left| \left| \left(\sum_{k=1}^d c_k \right) P_j \right| \right| - \left| \left| \sum_{k \not= j} c_k (P_k - P_j) \right| \right| - \delta \left| \left| \sum_{k\not=j} c_k P_k \right| \right| \\
&\geq ||c||_{\ell^1} - ||c||_{\ell^1} \cdot \varepsilon - \delta \cdot \sum_{k \not= j} c_k \\
& \geq ||c||_{\ell^1} \cdot (1 - \varepsilon) - \delta \cdot r.
\end{aligned} \ees
We may therefore conclude that
\bes
\max\{||Z_j||, 1 - \delta\} \geq ||c||_{\ell^1} \cdot (1 - \varepsilon) - \delta \cdot r.
\ees
The definition of $\delta = 2 \cdot \cfrac{1 - ||c||_{\ell^1} \cdot (1 - \varepsilon)}{1 - r}$ shows that this inequality cannot be satisfied by $1 -\delta$. It follows that
\bes\begin{aligned}
||Z_j|| \geq ||c||_{\ell^1} \cdot (1 - \varepsilon) - \delta \cdot r = ||c||_{\ell^1} \cdot (1 - \varepsilon) - 2r \cdot \cfrac{1 - ||c||_{\ell^1} \cdot (1 - \varepsilon)}{1 - r} := b_j(c, \varepsilon).
\end{aligned}\ees
As $(Z_1, \ldots, Z_d)$ is a tuple of mutually annihilating operators with joint spectrum in $L$, and $0 \in L$, it follows that $b_j(c, \varepsilon) \cdot e_j \in L$.
Since this claim holds for all $\varepsilon > 0$ which are sufficiently small, it follows that
\bes
\left(||c||_{\ell^1} - 2r \cdot \cfrac{1 - ||c||_{\ell^1}}{1 - r} \right) \cdot e_j \in L.
\ees
As $c$ approaches  $y$, and consequently $||c||_{\ell^1}$ approaches $1$, this implies $e_j \in L$. Finally, $L = \ol{\bB}_{1,d}^+$.

Similarly, if $\Wmax{2}(K) = \Wmin{2}(K)$, then we may apply the same affine transformation to suppose $K \subseteq \ol{\bB}_{1,d}^+$ is positioned as above, repeating the argument with $K = L$.
Because the argument above only dilates $2 \times 2$ matrices, we may similarly conclude that $K$ is equal to $\ol{\bB}_{1,d}^+$, which in particular is a simplex.
\end{proof}

Theorem \ref{thm:pointednonsense} describes a particular scenario in which a circumscribing simplex over $K$ is a minimal dilation hull of $K$. It may be used to refine the dilation constants found in Theorem \ref{thm:BplusWminmax}.

\begin{corollary}\label{cor:mindil_B1plus}
For $1 \leq p \leq \infty$, $d^{1-1/p} \cdot \ol{\bB}_{1,d}^+ \in \operatorname{md}(\ol{\bB}_{p,d}^+)$.
\end{corollary}
\begin{proof}
The simplex $d^{1-1/p} \cdot \ol{\bB}_{1,d}^+$ contains $\ol{\bB}_{p,d}^+$, so it is certainly a dilation hull. Further, $\bB_{p,d}^+$ is simplex-pointed at $(0, \ldots, 0)$, and the simplex $d^{1-1/p} \cdot \ol{\bB}_{1,d}^+$ is positioned over $\ol{\bB}_{p,d}^+$ exactly as demanded in Theorem \ref{thm:pointednonsense}, with base point $x = (0, \ldots, 0)$ and opposite point $y = (d^{-1/p}, \ldots, d^{-1/p})$.
\end{proof}

\begin{remark} In Theorem \ref{thm:pointednonsense}, it is crucial that $y \in K \cap \Delta$ is in the \textit{interior} of the face $F$ of $\Delta$ which excludes $x$, and that the edges of $\Delta$ based at $x$ point in the same directions as the edges of the simplex $\overline{O \cap K}$. Indeed, without these assumptions it is possible that there is another simplex $\Pi$ such that $K \subseteq \Pi \subsetneq \Delta$. Because we know that $\Wmax{}(\Pi) = \Wmin{}(\Pi)$, this fact would certainly rule out $\Delta$ as a minimal dilation hull of $K$.
 The figures below show one example which meets the conditions of Theorem \ref{thm:pointednonsense}, and two examples which do not meet the conditions. 
\begin{figure}[htbp]
  \includegraphics[scale=0.65]{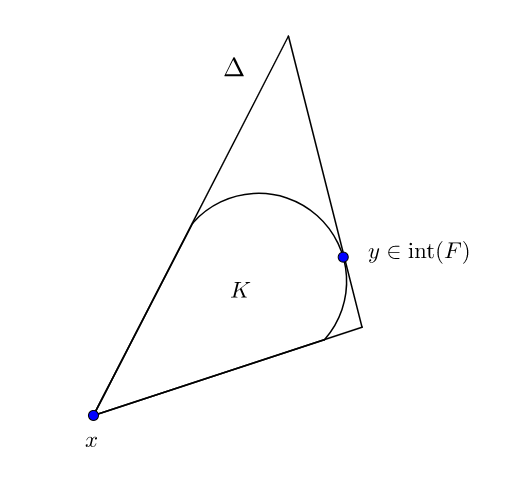}
  \caption{$K$ and $\Delta$ meet the conditions of Theorem \ref{thm:pointednonsense}.}
  \label{fig:circumscribedsimplex}
\end{figure}

\begin{figure}[htbp]
  \includegraphics[scale=0.5]{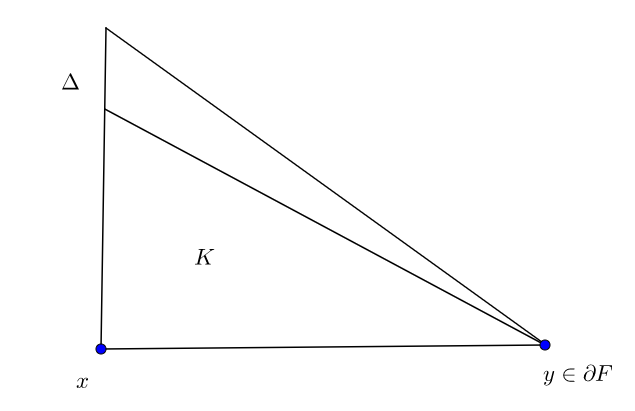}
  \caption{$\Delta$ is not a minimal dilation hull of $K$. Theorem \ref{thm:pointednonsense} does not apply, as the chosen intersection point $y$ is in $\partial F$, where $F$ is the face of $\Delta$ that excludes $x$. }
  \label{fig:verticalcontainment}
\end{figure}

\begin{figure}[htbp]
  \includegraphics[scale=0.5]{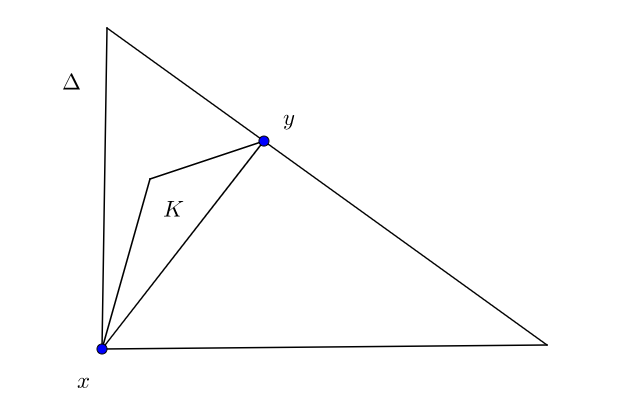}
  \caption{$\Delta$ is not a minimal dilation hull of $K$. Theorem \ref{thm:pointednonsense} does not apply, as the edges of $\Delta$ do not point in the same direction as the edges of $K$ that are defined near $x$.}
  \label{fig:pointingorigin}
\end{figure}

\end{remark}

We conclude with the following problem, which could potentially generalize Theorem \ref{thm:pointednonsense} and Proposition \ref{prop:minsimplexball}.

\begin{problem}
Let $K$ be a convex body and let $\Delta$ be a simplex with $K \subseteq \Delta$. If $\Delta$ is a circumscribing simplex, must $\Delta \in \operatorname{md}(K)$?
\end{problem}

\section*{acknowledgments}

We would like to thank Emanuel Milman for giving us the lead to the reference \cite{Palmon}. We also thank the referee for helpful feedback.

%%%%%%%%%%%%%%%%%%%%%%%%%%%%%%%%%%%%%%%%%%%%%%%%%%%%%%%%%%%%%%%%%%%%%%%%%%%%
%%%%%%%%%%%%%%%%%%%%%%%%%%%%%%%%%%%%%%%%%%
\bibliographystyle{amsplain}

\end{document}